\documentclass[12pt]{amsart}
\usepackage{amsfonts,amsrefs,latexsym,amsmath, amssymb}
\usepackage{mathrsfs,MnSymbol}
\usepackage{url,color}
\usepackage{upgreek}
\usepackage{fancyhdr}
\usepackage{hyperref}
\usepackage{scalefnt}
\usepackage{url}

\newtheorem{proposition}{Proposition}[section]
\newtheorem{lemma}[proposition]{Lemma}
\newtheorem{corollary}[proposition]{Corollary}
\newtheorem{theorem}{Theorem}[section]

\theoremstyle{definition}
\newtheorem{definition}{Definition}[section]
\newtheorem{remark}{Remark}[section]

\numberwithin{equation}{section}

\newcommand{\SecondFund}{k}

\newcommand{\Ric}{Ric}

\newcommand{\Mfour}{\mathbf{M}}
\newcommand{\gfour}{\mathbf{g}}

\newcommand{\Chfour}{\mathbf{\Gamma}}
\newcommand{\kfour}{\mathbf{k}}
\newcommand{\Ricfour}{\mathbf{Ric}}
\newcommand{\Riemfour}{\mathbf{Riem}}
\newcommand{\Rfour}{\mathbf{R}}
\newcommand{\Tfour}{\mathbf{T}}

\newcommand{\Nml}{\hat{\mathbf{N}}}

\newcommand{\Dfour}{\mathbf{D}}

\newcommand{\dlap}{\Theta}

\newcommand{\grenormalized}{h}

\newcommand{\christrenormalizedarg}[3]{{^{(h)} \mkern-2mu \Gamma_{#1 \ #3}^{\ #2}}}
\newcommand{\Ricrenormalized}{{^{(h)} \mkern-4mu Ric}}
\newcommand{\Ricrenormalizedarg}[2]{{^{(h)} \mkern-4mu Ric_{\ #2}^{#1}}}
\newcommand{\currenormalized}{{^{(h)} \mkern-4mu R}}

\newcommand{\LinSecondFund}{\upkappa}
\newcommand{\tracefreeLinSecondFund}{\hat{\upkappa}}

\newcommand{\SFRenormalized}{\varphi}
\newcommand{\LapseRenormalized}{\upnu}

\newcommand{\gKasner}{\mathring{g}}
\newcommand{\ginverseKasner}{\mathring{g}^{-1}}
\newcommand{\SecondFundKasner}{\mathring{k}}
\newcommand{\tracefreeSecondFundKasner}{\hat{\mathring{k}}}

\newcommand{\sfKasner}{\mathring{\phi}}

\newcommand{\gdata}{{^{0} \mkern-2mu g}}
\newcommand{\gdatadarg}[2]{{^{0} \mkern-2mu g_{#1 #2}}}
\newcommand{\SecondFunddata}{{^{0} \mkern-2mu k}}
\newcommand{\SecondFunddatadarg}[2]{{^{0} \mkern-2mu k_{#1 #2}}}
\newcommand{\SecondFundupdatadarg}[2]{{^{0} \mkern-2mu k_{\ #2}^{#1}}}
\newcommand{\firstsfdata}{{^{0} \mkern-3mu \phi}}
\newcommand{\secondsfdata}{{^{0} \mkern-2mu \psi}}
\newcommand{\squaresecondsfdata}{{^{0} \mkern-2mu \psi^2}}
\newcommand{\curdata}{{^{0} \mkern-3mu R}}

\newcommand{\nabladataarg}[1]{{^{0} \mkern-1.5mu \nabla_{\mkern-2mu#1}}}

\newcommand{\highnorm}[1]{\mathscr{S}_{(Frame);#1}}
\newcommand{\parabolichighnorm}[1]{\mathscr{S}_{(Parabolic \ Frame);#1}}

\newcommand{\smallparameter}{\uptheta}
\newcommand{\tracefreeparameter}{\upeta}

\newcommand{\Euc}{E}
\newcommand{\ID}{I}

\topmargin - .23 in

\begin{document}


\title{A Regime of Linear Stability for the Einstein-Scalar Field System with Applications to Nonlinear Big Bang Formation}
\author{Igor Rodnianski$^{*}$
\and 
Jared Speck$^{**}$}

\thanks{$^*$Princeton University, Department of Mathematics, Fine Hall, Washington Road, Princeton, NJ 08544-1000, USA. \texttt{irod@math.princeton.edu}}

\thanks{$^{**}$Massachusetts Institute of Technology, Department of Mathematics, 77 Massachusetts Ave, Room 2-265, Cambridge, MA 02139-4307, USA. \texttt{jspeck@math.mit.edu}}

\thanks{$^{*}$ IR gratefully acknowledges support from 
NSF grant \# DMS-1001500.}	

\thanks{$^{**}$ JS gratefully acknowledges support from NSF grant \# DMS-1162211,
from NSF CAREER grant \# DMS-1454419,
from a Sloan Research Fellowship provided by the Alfred P. Sloan foundation,
and from a Solomon Buchsbaum grant administered by the Massachusetts Institute of Technology.
}	

\begin{abstract}
	We linearize the Einstein-scalar field equations,
	expressed relative to constant mean curvature (CMC)-transported spatial coordinates gauge,
	around members of the well-known family of Kasner solutions on $(0,\infty) \times \mathbb{T}^3$.
	The Kasner solutions model a spatially uniform scalar field evolving 
	in a (typically) spatially anisotropic spacetime that expands towards the future
	and that has a ``Big Bang'' singularity at $\lbrace t = 0 \rbrace$.
	We place initial data for the linearized system along $\lbrace t = 1 \rbrace \simeq \mathbb{T}^3$
	and study the linear solution's behavior in the collapsing direction $t \downarrow 0$.
	Our first main result is the proof of an
	\emph{approximate $L^2$ monotonicity identity}
	for the linear solutions.
	Using it, we prove a linear stability result that holds
	when the background Kasner solution is
	sufficiently close to the Friedmann-Lema\^{\i}tre-Robertson-Walker (FLRW) solution.
	In particular, we show that 
	as $t \downarrow 0$,
	various time-rescaled components of the linear solution converge to
	regular functions defined along $\lbrace t = 0 \rbrace$.
	In addition, we motivate the preferred direction of the approximate monotonicity
	by showing that the CMC-transported spatial coordinates gauge can be viewed as
	a limiting version of a family of parabolic gauges for the lapse variable;
	an approximate monotonicity identity and corresponding linear stability results
	also hold in the parabolic gauges, 
	but the corresponding parabolic PDEs are locally well-posed only in the direction
	$t \downarrow 0$.
	Finally, based on the linear stability results,
	we outline a proof of the following result,
	whose complete proof will appear elsewhere:
	the FLRW solution is globally nonlinearly stable in the collapsing direction 
	$t \downarrow 0$ under small perturbations of its data at $\lbrace t = 1 \rbrace$.
\bigskip

\bigskip


\end{abstract}

\maketitle

\centerline{\today}

\setcounter{tocdepth}{1}
\pagenumbering{roman} \tableofcontents \newpage \pagenumbering{arabic}

\section{Introduction} \label{S:INTRO}
This is the first of two papers in which we 
derive a new approximate monotonicity identity
for two Einstein-matter systems and use it to prove
linear and nonlinear stability results
for cosmological\footnote{By ``cosmological,'' we mean that 
the spacetime manifold $\Mfour$ has compact Cauchy hypersurfaces and that the Ricci curvature of 
the spacetime metric $\gfour_{\mu \nu}$ verifies $\Ricfour_{\alpha \beta}\mathbf{X}^{\alpha}\mathbf{X}^{\beta} \geq 0$ for all timelike vectors $\mathbf{X}^{\mu}$. 
For the Einstein-scalar field system,
this Ricci curvature condition is always verified by solutions 
because Einstein's equations imply that 
$
\Ricfour_{\mu \nu} 
= \Tfour_{\mu \nu} 
	- 
	\frac{1}{2} 
	(\gfour^{-1})^{\alpha \beta} \Tfour_{\alpha \beta} \gfour_{\mu \nu} 
$ 
and because the energy-momentum tensor $\Tfour_{\mu \nu}$ 
of a scalar field verifies the strong energy condition.} 
solutions featuring Big Bang singularities. 
By a ``Big Bang'' singularity in a spacetime, 
we roughly mean a spacelike hypersurface such that the solution
exhibits curvature blowup along the entire hypersurface.
In particular, our nonlinear result constitutes a proof of stable 
curvature blowup along a spacelike hypersurface
for an open set of solutions. 
We now briefly summarize the nonlinear result,
which is proved\footnote{More precisely,
Theorem~\ref{T:FLRWROUGHFIRSTVERSION} is a special case
of the results of \cite{iRjS2014b}:
in \cite{iRjS2014b}
we prove an analog of Theorem~\ref{T:FLRWROUGHFIRSTVERSION}
for the stiff fluid matter model. Theorem~\ref{T:FLRWROUGHFIRSTVERSION}
follows as a special case in which the fluid's vorticity is zero; see Subsect.\ \ref{SS:EINSTEINSCALARFIELD} for further
clarification of this point.}  
in our second paper \cite{iRjS2014b};
see Theorem~\ref{T:STABILITYOFFLRW} for a precise statement and
\cite{iRjS2014b} for an even more detailed statement.

\begin{theorem}[\textbf{Stable Big Bang Formation for near-FLRW solutions} (Rough version)]
	\label{T:FLRWROUGHFIRSTVERSION}
	Consider initial data 
	for the Einstein-scalar field system given 
	on the manifold\footnote{Throughout, 
	$\mathbb{T}^n := [-\pi,\pi]^n$
	(with the ends identified) is an $n$-dimensional torus.} 
	$\Sigma_1 = \mathbb{T}^3$,
	which we identify 
	with a Cauchy hypersurface of constant time $t=1$,
	i.e., $\Sigma_1 = \lbrace 1 \rbrace \times \mathbb{T}^3$.
	If the data are close in a suitable Sobolev norm 
	to the data of the 
	Friedmann-Lema\^{\i}tre-Robertson-Walker (FLRW)
	solution (see Subsect.\ \ref{SS:OVERVIEWNONLINEAR}),
	then there exists a system of constant mean curvature-transported spatial coordinates
	$(t,x^1,x^2,x^3)$ such that the perturbed solution exists for 
	$(t,x) \in (0,1] \times \mathbb{T}^3$.
Like the FLRW solution,
the perturbed solution's
Kretschmann scalar $\Riemfour^{\alpha \beta \gamma \delta} \Riemfour_{\alpha \beta \gamma \delta}$
blows up like $t^{-4}$ as $t \downarrow 0$.
Moreover, the solution exhibits \underline{asymptotically velocity term dominated} 
(AVTD) behavior, which means that near $t=0$, the dynamics are
dominated by time derivative terms 
(that is, the spatial derivative terms in the equations become negligible), and
certain $t$-rescaled components of the solution converge in a monotonic fashion 
to regular functions of 
$x$ as $t \downarrow 0$.
In particular, as $t \downarrow 0$, the solution
is asymptotic to a solution of the VTD equations,
which are obtained by setting all spatial derivative terms equal to $0$
in the Einstein-scalar field equations
(expressed relative to the CMC-transported spatial coordinates gauge).
\end{theorem}
See Subsect.\ \ref{SS:OVERVIEWNONLINEAR} 
for further discussion of Theorem~\ref{T:FLRWROUGHFIRSTVERSION}, Subsect.\ \ref{SS:ROUGHSTATEMENTLINEARRESULTS} for 
a summary of our linear results,
and Subsect.\ \ref{SS:FIVESTEPS}
for a discussion of the relationship between the various results.

In addition to deriving stability results,
we also identify a new one-parameter family of parabolic gauges for the lapse function,
which, like the well-known constant mean curvature (CMC)-transported-spatial coordinates gauge, 
leads to a formulation of the equations exhibiting the key structural features
that allow us to prove the main results.
For our purposes here, none of the gauges that we employ are manifestly superior. 
The parabolic lapse gauges are more general/flexible in that one does not need to
construct\footnote{For a general spacetime, 
such a CMC hypersurface does not exist.
However, CMC hypersurfaces do exist for the spacetime solutions studied here; see \cite{iRjS2014b}
for a proof of this fact.} 
a CMC hypersurface to employ them. However, 
in the present context,
they are a bit more unwieldy to use.
For this reason, most of our results here rely on CMC foliations of spacetime.
However, it is conceivable that the parabolic gauges will be useful in 
future studies of cosmological spacetimes.
For this reason, 
in Sect.\ \ref{S:PARABOLICMONOTONICITY},
we provide these gauges in detail 
and re-derive our linear results relative to them.
We stress that all of the gauges under consideration 
lead to a formulation of the equations exhibiting 
infinite speed\footnote{The fundamental (gauge-independent) dynamic variables
in the Einstein-scalar field equations propagate at a finite speed. 
It is only our description of them that involves an infinite speed.}  
of propagation. The infinite speed is fundamental for our analysis 
since our approach is based on \emph{synchronizing}
the singularity across a spacelike hypersurface 
of constant time;
in a purely hyperbolic gauge involving a time coordinate $t$, 
it is not generally possible 
to ensure that a blowup-hypersurface (should one exist)
is of the form $\lbrace t = \mbox{\upshape const} \rbrace$.

\subsection{The Einstein-scalar field equations}
\label{SS:EINSTEINSCALARFIELD}
In the present article, we restrict our attention to the study of the Einstein-scalar field equations
\begin{subequations}
\begin{align}
	\Ricfour_{\mu \nu} - \frac{1}{2}\Rfour \gfour_{\mu \nu} & = \Tfour_{\mu \nu},
		\label{E:EINSTEINSF} \\
	(\gfour^{-1})^{\alpha \beta} \Dfour_{\alpha} \Dfour_{\beta} \phi & = 0  \label{E:WAVEMODEL}
\end{align}
\end{subequations}
with data given on the Cauchy hypersurface $\Sigma_1 = \lbrace 1 \rbrace \times \mathbb{T}^3$. 
Above and throughout,  
$\Ricfour$ denotes the Ricci tensor of 
the spacetime\footnote{By ``spacetime,'' we mean a four-dimensional time-orientable
manifold $\Mfour$ equipped with a Lorentzian metric $\gfour$ of signature $(-,+,+,+)$.} 
metric $\gfour$, 
$\Rfour = (\gfour^{-1})^{\alpha \beta}\Ricfour_{\alpha \beta}$ 
denotes the scalar curvature of $\gfour$, 
$\Dfour$ denotes the Levi--Civita connection of $\gfour$, 
and $\Tfour$ denotes the energy-momentum tensor of the scalar field $\phi$:
\begin{align} \label{E:EMTSCALARFIELD}
	\Tfour_{\mu \nu} = \Dfour_{\mu}\phi \Dfour_{\nu} \phi 
		- \frac{1}{2} \gfour_{\mu \nu} (\gfour^{-1})^{\alpha \beta} \Dfour_{\alpha} \phi \Dfour_{\beta} \phi.
\end{align}
The scalar field is a simple matter model that has
been well-studied in mathematical general relativity in the context 
of asymptotically flat spacetimes; 
see \cites{dC1991,dC1999b,dC1993,dC1987,dC1986b,dC1986c}.
In our complementary article \cite{iRjS2014b},
we study the Einstein-stiff fluid system,
where a stiff fluid has sound speed equal to unity (that is, equal to the speed of light).
The stiff fluid model is more general in the sense that
it reduces\footnote{For scalar fields with a timelike gradient and under an exactness condition tied to the fluid velocity and enthalpy per particle; see \cite{iRjS2013} for further discussion on this point.} 
to the scalar field model when the fluid's vorticity vanishes.
Due to our gauge choices (which we explain in detail below),
one should identify the ``data hypersurface''
$\Sigma_1$ with a surface of constant time $1$. 
We will study the behavior of solutions as $t \downarrow 0$.
The singular behavior that we will uncover occurs along $\Sigma_0$,
which will be identified with a surface of constant time $0$.

Although our results apply when the initial Cauchy hypersurface is $\mathbb{T}^3$,
they can easily be generalized to the case of n spatial dimensions, that is, to
the case of $\mathbb{T}^n$ for $n \geq 1$. We anticipate that similar results might also hold
for some other matter models with special properties and, in the case of very
high spatial dimensions, for the Einstein-vacuum equations; see the discussion
in Subsect.\ \ref{SS:OTHERMATTERMODELSECT}.

\subsection{Paper outline} 
\label{SS:PAPEROUTLINE}

\begin{itemize}
	\item In the remainder of Sect.\ \ref{S:INTRO},
		we summarize our linear and nonlinear stability results,
		discuss their relationship,
		and provide context by discussing prior work.
	\item In Sect.\ \ref{S:NOTATION}, we introduce some notation and conventions that we use throughout the article.
	\item In Sect.\ \ref{S:CMCFORMULATIONOFEINSTEIN}, we provide the Einstein-scalar field
		equations in CMC-transported spatial coordinates. We then linearize
		the equations around members of the (generalized) Kasner family.
	\item In Sect.\ \ref{S:NORMSANDENERGIES}, we provide the norms and energies that we use 
		in our analysis of linear solutions.
	\item In Sect.\ \ref{S:MONOTONICITYIDENTITIES}, we prove
		an approximate monotonicity identity for linear solutions.
		The identity lies at the heart of all of our results.
	\item In Sect.\ \ref{S:ENERGYESTIMATES},
		we use the approximate monotonicity identity to derive
		mildly singular energy estimates for linear solutions
		in the case that the Kasner background is nearly spatially isotropic.
	\item In Sect.\ \ref{S:LINEARSTABILITY},
		we use the mildly singular energy estimates to prove
		a linear stability result for nearly spatially isotropic 
		Kasner backgrounds.
	\item In Sect.\ \ref{S:NONLINEARPROBLEM}, we 
		use the results of the previous sections to
		outline a proof of the nonlinear stability of the FLRW solution 
		near its Big Bang singularity; complete details are located in
		\cite{iRjS2014b}.
	\item In Sect.\ \ref{S:PARABOLICMONOTONICITY},
		we introduce a family of parabolic lapse gauges 
		and re-derive our linear results in these gauges.
\end{itemize}

\subsection{The FLRW solution and 
preliminary context for the results}
\label{SS:OVERVIEWNONLINEAR}
A quintessential example of a Big Bang spacetime 
is the FLRW solution
(referred to in Theorem~\ref{T:FLRWROUGHFIRSTVERSION})
to the Einstein-scalar field system,
which plays a prominent role in cosmology in view of
its spatially isotropic nature. 
It can be expressed in the well-known form
\begin{align} \label{E:FLRWBACKGROUNDSOLUTION}
	\gfour_{FLRW} & = - dt^2 + g_{FLRW}, 
	\qquad g_{FLRW} = t^{2/3} \sum_{i=1}^3 (dx^i)^2, \qquad \phi_{FLRW} = \sqrt{\frac{2}{3}} \ln t, 
	\qquad (t,x) \in (0,\infty) \times \mathbb{T}^3.
\end{align}	
One can compute that the Kretschmann scalar of $\gfour_{FLRW}$, 
namely $\Riemfour^{\alpha \beta \gamma \delta} \Riemfour_{\alpha \beta \gamma \delta}$,
blows up like $t^{-4}$ as $t \downarrow 0$. That is, the FLRW solution 
has a Big Bang singularity at $t=0$.

Theorem~\ref{T:FLRWROUGHFIRSTVERSION}  
shows that like the FLRW solution,
perturbed solutions also exhibit
the same kind of curvature blowup.
We provide the complete proof of Theorem~\ref{T:FLRWROUGHFIRSTVERSION} 
in the companion article \cite{iRjS2014b} 
(for the Einstein-stiff fluid system).
The proof is part of a ``five-step program'' 
encompassing the results of both papers,
which we summarize in Subsect.\ \ref{SS:FIVESTEPS}.
Some key steps in the program 
are of independent interest and hold in a more
general context than the form in which they are used in the proof 
of Theorem~\ref{T:FLRWROUGHFIRSTVERSION}. In this article, we identify 
such a more general context
and give rigorous proofs of those key steps that remain valid.
In particular, here
we study a large family of linearized versions of the Einstein-scalar field
equations, where the backgrounds around which we linearize 
have been well-studied in the mathematical general relativity literature. 
For each linearized system, 
we derive the aforementioned \emph{approximate monotonicity identity}
for the linear solutions. Specifically, we linearize the equations
around members of the family of \emph{generalized Kasner solutions}, 
which are explicit spatially homogeneous (that is, non-$x$-dependent) solutions
whose unique spatially isotropic member is the FLRW solution.
For generalized Kasner solutions, 
the spacetime metric is of the form 
$\gfour = - dt \otimes dt + \sum_{I=1}^3 t^{2 q_I} \omega^I \otimes \omega^I$,
where the $q_I$ are constants verifying certain constraints
and the $\omega^I:= \omega_a^I dx^a$
are a set of three $\gfour$-orthogonal one-forms on $\mathbb{T}^3$.
In particular, relative to standard coordinates $\lbrace x^a \rbrace_{a=1,2,3}$ on $\mathbb{T}^3$,
we have $\omega_a^I = \omega_a^I dx^a$, where the $\omega_a^I$ are constants
and $\mbox{\upshape det} (\omega_a^I) \neq 0$.
See Subsect.\ \ref{SS:KASNER} for more details regarding these 
generalized Kasner solutions.
Here we only note that for brevity, 
we will often refer to these (nonlinear) Einstein-scalar field solutions as 
Kasner solutions. This breaks with the traditional convention,
which reserves the label ``Kasner solution''
for Einstein-vacuum solutions.
A fundamental aspect of the Kasner backgrounds 
(around which we linearize) is that, like the FLRW solution, 
they have Big Bang singularities at $t=0$ (aside from some exceptional cases).
In addition to deriving the approximate monotonicity identity,
we also use it to prove a linear stability result for 
a subset of the Kasner backgrounds,
specifically those that are nearly spatially isotropic
(that is, for near-FLRW Kasner backgrounds).
Before further describing the five-step program
and how our linear/nonlinear stability results fit into it,
we first provide context that clarifies the significance of
Theorem~\ref{T:FLRWROUGHFIRSTVERSION}.

\begin{itemize}
	\item Although the data we consider
		fall under the scope of the
		Hawking-Penrose ``singularity'' theorems\footnote{More precisely,
		see \cite{rW1984}*{Theorem 9.5.1} for a version
		of ``Hawking's theorem'' that can be applied to the initial data considered in Theorem~\ref{T:FLRWROUGHFIRSTVERSION}.}
	\cites{sH1967,rP1965},
	Theorem~\ref{T:FLRWROUGHFIRSTVERSION} goes beyond
	the soft conclusion of geodesic incompleteness
	provided by those theorems in that it
	shows that the incompleteness is due to curvature
	blowup along the hypersurface $\lbrace t = 0 \rbrace$.
	As such, the solutions of Theorem~\ref{T:FLRWROUGHFIRSTVERSION}
	exhibit Strong Cosmic Censorship-type behavior,
	by which we mean that the solution variables cannot be
	extended as $C^2$ tensorfields
	beyond the boundary portion
	$\lbrace t = 0 \rbrace$ of the maximal development
	of the data. This is the first result of this type
	for Einstein's equations that does not involve 
	symmetry or analyticity assumptions on the data.
	\item The AVTD behavior proved in Theorem~\ref{T:FLRWROUGHFIRSTVERSION},
		though predicted via heuristic arguments for the scalar field
		model in \cite{vBiK1972} and for the stiff fluid in \cite{jB1978},
		had not previously been shown
		in solutions without symmetry, except under the assumption of
		spatial analyticity \cite{lAaR2001}; see Subsect.\ \ref{SS:PREVIOUSWORK}
		for further discussion on these works.
		Moreover, as we describe below, 
		the solutions from Theorem~\ref{T:FLRWROUGHFIRSTVERSION} are such that
		at each fixed spatial point $x$, its
		asymptotic behavior is Kasner-like,
		by which we mean that its limiting behavior is well-described
		by fields that are related to members of the aforementioned Kasner family.
		The belief that the ``end states'' should, at each fixed $x$, be Kasner-like
		was part of the heuristics given in \cites{vBiK1972,jB1978}.
		More precisely, the authors in \cite{vBiK1972}
		\emph{assumed} that all spatial derivative terms in the evolution equations
		become negligible near the singularity $\lbrace t=0 \rbrace$.
		The authors then argued that
		the spacetime metric should asymptotically behave like
		$- dt \otimes dt + \sum_{I=1}^3 t^{2 q_I(x)} \omega^I(x) \otimes \omega^I(x)$
		near the singularity,
		that is, like Kasner solutions in which the exponents
		and one-forms are $x$-dependent.
		See just below Theorem~\ref{T:ROUGHVERSIONLINEARSTABILITY} for further 
		comments on the asymptotic behavior
		of solutions to the linearized equations.
	\item The monotonic behavior of the solution as $t \downarrow 0$
		was also predicted in \cites{vBiK1972,jB1978} and 
		in fact is accounted for by the authors' posited 
		asymptotic form of the metric
		$- dt^2 + \sum_{I=1}^3 t^{2 q_I(x)} \omega^I(x) \otimes \omega^I(x)$.
			This \emph{existence} of an interesting set of
			\emph{spatially analytic} solutions
			to the Einstein-scalar field and Einstein-stiff fluid systems
			exhibiting this kind of monotonic asymptotic
			behavior was rigorously shown
			in the aforementioned work \cite{lAaR2001}. 
			Like the heuristic arguments given in \cites{vBiK1972,jB1978}
			and the rigorous results of \cite{lAaR2001},
			our proof of the monotonic behavior
			(via the approximate monotonicity identity and its consequences)
			relies on the particular structure of the scalar field and stiff fluid 
			matter models; see Subsect.\ \ref{SS:PREVIOUSWORK}
			for further discussion on this point.
\end{itemize}

\subsection{Initial value problem formulation of the Einstein equations and gauges}
\label{SS:INITIALVALUEANDCMC}
Before further discussing our results, we
first discuss some basic issues concerning the initial value problem for the
(nonlinear) Einstein-scalar field system \eqref{E:EINSTEINSF}-\eqref{E:WAVEMODEL}
and our gauge choices. The fundamental results \cite{CB1952} and \cite{cBgR1969},
which are respectively by Choquet--Bruhat 
and Choquet--Bruhat + Geroch,
showed that the system \eqref{E:EINSTEINSF}-\eqref{E:WAVEMODEL}
has an initial value problem formulation in which sufficiently regular data
give rise to a unique maximal globally hyperbolic development.\footnote{Roughly, this is the largest possible
classical solution to the Einstein-scalar field equations that is uniquely determined by the data.} 
The rest of our discussion here is adapted to the setup of the present article, 
where the initial Cauchy hypersurface is $\mathbb{T}^3$.
The ``geometric data'' (for the nonlinear equations) consist of the following fields on $\mathbb{T}^3$: 
$(\gdatadarg{i}{j},\SecondFunddatadarg{i}{j},\firstsfdata,\secondsfdata)$.
Here, $\gdatadarg{i}{j}$ is a Riemannian metric,
$\SecondFunddatadarg{i}{j}$ is a symmetric two-tensor,
and $\firstsfdata$ and $\secondsfdata$ are a pair of functions.
A solution launched by the data consists of a four-dimensional time-oriented spacetime
$(\mathbf{M}, \gfour)$, 
a scalar field $\phi$ on $\mathbf{M}$, and  
an embedding $\mathbb{T}^3 \overset{\iota}{\hookrightarrow}\mathbf{M}$ such that 
$\iota(\mathbb{T}^3)$ is a Cauchy hypersurface in $(\mathbf{M},\gfour)$. 
The spacetime fields must verify the equations 
\eqref{E:EINSTEINSF}-\eqref{E:WAVEMODEL}
and be such that $\iota^* \gfour = \gdata$, 
$\iota^* \kfour = \SecondFunddata$, 
$\iota^* \phi = \firstsfdata$, 
$\iota^* \Nml \phi = \secondsfdata$, 
where $\kfour$ is the second fundamental form of $\iota(\mathbb{T}^3)$
(our sign convention is given in \eqref{E:SECONDFUNDDEF}), 
$\Nml \phi$ is the derivative of $\phi$ in the direction of the future-directed normal $\Nml$ to $\iota(\mathbb{T}^3)$,
and $\iota^*$ denotes pullback by $\iota$. 
Throughout the article, we will often suppress the embedding and identify $\mathbb{T}^3$ with $\iota(\mathbb{T}^3)$. 

It is well-known (see also Prop.\ \ref{P:EINSTEINSFCMC}) 
that the data are constrained by the \emph{Gauss} and \emph{Codazzi} equations, 
which take the following form for the Einstein-scalar field system:
\begin{subequations}
\begin{align}
	\curdata 
	- \SecondFundupdatadarg{a}{b} \SecondFundupdatadarg{b}{a}
	+ (\SecondFundupdatadarg{a}{a})^2 
	& = 2 \Tfour(\Nml,\Nml)|_{\mathbb{T}^3} 
	= \squaresecondsfdata + \nabla^a \firstsfdata \nabla_a \firstsfdata, 
	\label{E:GAUSSINTRO} \\
	\nabla_a \SecondFundupdatadarg{a}{j} 
		- \nabladataarg{j} \SecondFundupdatadarg{a}{a}  
		& = - \Tfour(\Nml,\frac{\partial}{\partial x^j})|_{\mathbb{T}^3}
		= - \secondsfdata \nabla_j \firstsfdata. 
		\label{E:CODAZZIINTRO}
\end{align}
\end{subequations}
Above, $\Tfour(\Nml,\Nml) := \Tfour_{\alpha \beta} \Nml^{\alpha} \Nml^{\beta}$,
$\nabla$ denotes the Levi--Civita connection of $\gdata$, $\curdata$ denotes 
the scalar curvature of $\gdata$, and indices are lowered and raised with
$\gdata$ and its inverse. Equations \eqref{E:GAUSSINTRO}-\eqref{E:CODAZZIINTRO} 
are known, respectively, as the \emph{Hamiltonian} and \emph{momentum} constraints.

As is well known, to obtain a hyperbolic formulation, 
an elliptic-hyperbolic formulation, or a parabolic-hyperbolic formulation
of equations \eqref{E:EINSTEINSF}-\eqref{E:WAVEMODEL}, 
suitable for studying the initial value problem,
one must impose gauge choices.
As we mentioned at the beginning,
there are two gauges in which we are able to derive our main results.
The first is the well-known CMC-transported-spatial-coordinates gauge,
which we recall in detail in Sect.\ \ref{S:CMCFORMULATIONOFEINSTEIN}.
In this gauge, the spacetime metric 
$\gfour$ is decomposed into the \emph{lapse} $n$ and the Riemannian $3$-metric $g$ on
$\Sigma_t := \lbrace (s,x) \in (0,1] \times \mathbb{T}^3 \ | \ s = t \rbrace$ as follows:
\begin{align} \label{E:GFOURCMCTRANSPORTED}
	\gfour & = - n^2 dt^2 + g_{ab} dx^a dx^b. 
\end{align}
The spatial coordinates\footnote{Technically, the spatial coordinates are only locally defined on $\mathbb{T}^3$,
even though the coordinate partial derivative vectorfields $\partial_i := \frac{\partial}{\partial x^i}$ 
can be globally defined so as to be smooth.} 
$\lbrace x^a \rbrace_{a=1,2,3}$ are called ``transported" because they are constant
along the integral curves of the vectorfield $\Nml = n^{-1} \partial_t$, 
which is the future-directed unit normal to $\Sigma_t$.
The basic variables to be solved for in the nonlinear equations are
$g_{ij}$, 
$\SecondFund_{ij} := - \frac{1}{2} n^{-1} \partial_t g_{ij}$, 
$n$, 
and $\phi$. 
The hypersurfaces $\Sigma_t$ have mean curvature $\frac{1}{3} \SecondFund_{\ a}^a$
that is constant, that is, that depends only on $t$.
To achieve this, $n$ must verify an elliptic PDE on $\Sigma_t$.
Hence, this gauge leads to an elliptic-hyperbolic formulation of the equations.
Above and throughout, $\SecondFund_{\ j}^i = g^{ia} \SecondFund_{a j}$ 
denotes the (mixed) second fundamental form of the constant-time
hypersurface $\Sigma_t$. We normalize the time coordinate so that 
$\SecondFund_{\ a}^a(t,x) = - t^{-1}$
and we identify $\Sigma_1$ with the initial Cauchy hypersurface. 
To be admissible under this setup, 
the initial mixed second fundamental form must verify
$\SecondFundupdatadarg{a}{a} = - 1$.
See Sect.\ \ref{S:CMCFORMULATIONOFEINSTEIN} for a more detailed discussion of this gauge.
In particular, we provide the corresponding constraint and evolution equations in Prop.\ \ref{P:EINSTEINSFCMC}.
Until Sect.\ \ref{S:PARABOLICMONOTONICITY}, we will work
with CMC-transported spatial coordinates gauge.

The second gauge suitable for our purposes
is a one-parameter family of gauges that is in many ways like
the CMC-transported spatial coordinates gauge, except that the
elliptic CMC lapse equation is replaced with 
a parabolic evolution equation for $n$ that is well-posed in the 
past direction; see Sect.\ \ref{S:PARABOLICMONOTONICITY} for the details.
Gauges for Einstein's equations involving parabolic equations
have been considered in the general relativity literature for
several decades. For example, 
in the work \cite{jBgDeS1996}, 
the authors introduced a family of gauges 
in which the lapse solves a parabolic equation,
and they suggested that such gauges should
lead to efficient and accurate numerical simulations.
We also point out the work \cite{cUhvEjWgE2003} on 
the Euler-Einstein equations under 
the equation of state $p = c_s^2 \rho$, 
where $p$ is the fluid pressure, $\rho$ is its proper energy density,
and the constant $c_s$ verifies $0 < c_s \leq 1$.
In \cite{cUhvEjWgE2003},
the authors introduced the \emph{separable volume gauge},
which is a parabolic gauge that can be viewed as
a Lorentzian version of inverse mean curvature flow.
They posited that the separable volume gauge should be useful for 
proving rigorous theorems concerning
the behavior of inhomogeneous cosmological solutions near 
a spacelike singularity. Their main result was geometric:
they identified a set that is invariant under the flow of their equations
and conjectured that it is the past attractor of the flow.
Interestingly, well-posedness for the equations studied in \cite{cUhvEjWgE2003}
is not known because their principal part is not of any standard type.
In the work \cite{dGcG2005}, the authors slightly modified
the equations of \cite{cUhvEjWgE2003} to produce a system of 
transport-diffusion equations, which they showed to be well-posed. 
Readers can also consult \cite{cGjMG2006} 
for a discussion of local well-posedness 
for the Einstein equations under various gauge conditions involving 
a parabolic equation for the lapse.

\subsection{The five-step program}
\label{SS:FIVESTEPS}
We now summarize the five-step program mentioned in Subsect.\ \ref{SS:OVERVIEWNONLINEAR}.
In particular, we briefly introduce our linear results
and explain in what sense they are tied to/ constitute an extension of
the proof of Theorem~\ref{T:FLRWROUGHFIRSTVERSION}
given in \cite{iRjS2014b}.

\begin{enumerate}
	\item \textbf{(Approximate monotonicity identity)}
		In this article, for \emph{all Kasner backgrounds},
		we first establish an approximate monotonicity identity for solutions
		to the linearized equations.
		More precisely, we derive an integral identity for solutions 
		in which, due to some special cancellations,
		some unfavorable integrals are shown to be equal to
		favorably signed integrals, up to error terms.
		See Theorem~\ref{T:ROUGHVERSIONMONOTONICITY}
		for a rough summary of the integral identity
		and Theorem~\ref{T:CMCMONOTONICITYID} for the precise statement.
		The favorably signed integrals
		encourage some of the linear solution variables
		to decay as $t \downarrow 0$, that is, chronologically
		\emph{towards} the Kasner background's Big Bang.
		The monotonicity is indeed only approximate in the sense that 
		some of the unsigned error terms in the integral identity
		compete against the favorably signed integrals.
		It turns out that for nearly spatially isotropic backgrounds 
		(that is, for near-FLRW backgrounds), 
		the favorably signed integrals are sufficiently strong to \emph{absorb}
		most of the unsigned error terms, which is crucial for the next step.
	\item \textbf{(Mildly singular energy estimates at the lowest order for near-FLRW backgrounds)}
		Next, for \emph{nearly spatially isotropic} Kasner backgrounds,
		we use the approximate monotonicity identity from Step (1) to establish an 
		energy estimate and elliptic estimates 
		for solutions to the linearized equations.
		The elliptic estimates are needed to control the lapse,
		which verifies an elliptic equation in CMC
		gauge. If we were to instead use the parabolic lapse gauge mentioned 
		above, then the elliptic estimates
		would be replaced with parabolic energy estimates; see Sect.\ \ref{S:PARABOLICMONOTONICITY}.
		These estimates are at the level of the non-differentiated
		linearized equations.
		A key aspect is that
		the energy can blow up at the \emph{mild} rate 
		$t^{- c \tracefreeparameter}$
		as $t \downarrow 0$,
		where $c > 0$ is a universal constant and
		the constant $\tracefreeparameter \geq 0$ is a measure of how non-spatially-isotropic
		the Kasner background is. In particular, $\tracefreeparameter = 0$
		for the FLRW background; 
		see \eqref{E:TRACEFREEPARAMETER} for the
		precise definition of $\tracefreeparameter$.
		Because of the energy blowup and because of the precise structure of the 
		$t$-weights in the energies (see Def.\ \ref{D:ENERGIES}),
		\emph{the energy estimate is not in itself sufficient to establish
		linear stability results} that are consistent with the 
		nonlinear stable blowup result provided by Theorem~\ref{T:FLRWROUGHFIRSTVERSION}.
		Another key aspect of the energy estimate is that its proof
		crucially relies on the approximate monotonicity identity from
		Step (1). Without the combined strength of 
		the cancellations and favorably signed integrals provided by the identity, 
		we would have only been able to establish a more severe energy blowup-rate of $t^{-C}$ as 
		$t \downarrow 0$, where $C$ is a large constant.
		Such a severe energy blowup-rate would 
		not have been sufficient for establishing 
		the linear stability of the solution
		(see Step (4) for clarification on this point),
		which in turn would have prevented us from
		controlling the nonlinear error terms that
		we encounter in the proof of Theorem~\ref{T:FLRWROUGHFIRSTVERSION}.
		See Theorem~\ref{T:ROUGHVERSIONENERGYESTIMATESNODERIVATIVESLOSS}
		for a rough statement of the energy estimate
		and Theorem~\ref{T:L2MILDENERGYBLOWUPCMCGAUGE} for the precise statement.
	\item \textbf{(Mildly singular energy estimates up to top-order for near-FLRW backgrounds)}
		Next, we establish energy estimates and elliptic estimates
		for the linear solution's higher 
		spatial derivatives. Specifically, we show that the higher-order energies
		verify the same bounds as the base-level energy from Step (2).
		Since the Kasner backgrounds are spatially homogeneous,
		this step is analytically trivial though conceptually important, 
		as will become clear in Step (4). 
		Again, see Theorem~\ref{T:ROUGHVERSIONENERGYESTIMATESNODERIVATIVESLOSS}
		for a rough statement of the higher-order energy estimates
		and Theorem~\ref{T:L2MILDENERGYBLOWUPCMCGAUGE} for the precise statement.
		We emphasize that \emph{these energy estimates do not incur any loss of derivatives},
		which is of course crucial for closing the nonlinear problem.
	\item \textbf{(Linear stability and AVTD behavior)}
		Next, still within the class of nearly spatially isotropic
		Kasner backgrounds, we prove linear stability 
		using the energy estimates and elliptic estimates from Step (3).
		In particular, we use the energy estimates for the linear solution 
		and its higher-order spatial derivatives
		to establish \emph{improved estimates for the linear solution
		at the lower derivative levels}, 
		including convergence results consistent with 
		the AVTD behavior stated in Theorem~\ref{T:FLRWROUGHFIRSTVERSION}.
		In fact, this step constitutes a proof of the linear solution's AVTD behavior,
		which is a result that
		\emph{does not directly follow from the singular energy estimates of the previous step}.
		This step incurs a loss of derivatives, roughly because
		in deriving the convergence results and proving the AVTD behavior,
		we ``put all spatial derivative terms on the right-hand side'' of the 
		evolution equations.
		Thus, from the perspective of regularity,
		it is critically important
		that we have been able to independently establish the non-derivative-losing
		energy estimates from Step (3). 
		It is also critically important that the energy blowup-rate 
		$t^{- c \tracefreeparameter}$
		from Step (3) is mild for nearly spatially isotropic Kasner backgrounds;
		the mild blowup-rate results in the following: 
		many of the spatial-derivative-involving terms in the linearized equations are
	 \emph{integrable in time near the singularity}, which
		is the key to establishing linear stability.
		By integrable in time, we are roughly referring to the fact that
		$\int_{s=t}^1 s^p \, ds < C_p$
		whenever $p > - 1$, uniformly for $t \in (0,1]$;
		the integrability in time of the error terms
		is one of the main analytical aspects of the solution's AVTD behavior.
	\item \textbf{(Control of nonlinear error terms)}
		To prove Theorem~\ref{T:FLRWROUGHFIRSTVERSION},
		we must similarly establish 
		the following results for solutions to the nonlinear equations:
		\textbf{I)} an approximate monotonicity identity;
		\textbf{II)} a priori energy estimates and elliptic estimates up to top-order;
		and \textbf{III)} improved/AVTD estimates at the lower derivative levels.
		In the usual fashion, we rely on a bootstrap argument to accomplish this.
		Most aspects of the proofs of \textbf{I)}-\textbf{III)} are similar
		to the linear analysis.
		The new feature is that we must 
		also control the nonlinear error terms.
		It turns out that given the 
		framework we have established in
		Steps (1)-(4), the nonlinear terms are
		not too difficult to control. The main
		thing that needs to be checked is that 
		in all of the estimates,
		the ``borderline'' error terms
		(borderline in the sense of their blowup-rate as $t \downarrow 0$)
		generated by the nonlinear interactions
		can either \textbf{i)} be absorbed into the favorably-signed 
		integrals generated by the approximate monotonicity identity
		from Step (1) or \textbf{ii)} are multiplied by a coefficient
		that remains $L^{\infty}$-small as $t \downarrow 0$.
		This allows us to prove that the energy blowup-rate 
		in the nonlinear problem is also mild, 
		roughly at worst $t^{- \delta}$, where $\delta > 0$ is small whenever
		the data are near the FLRW data. 
		The detailed proofs are located
		in the companion article \cite{iRjS2014b}.
		In Sect.\ \ref{S:NONLINEARPROBLEM}, we outline all of the main
		ideas and show how to control several representative nonlinear
		error integrals, including a borderline one.
		All of the main ingredients needed to 
		control the nonlinear terms and to prove the theorem
		are provided by Steps (1)-(4).
\end{enumerate}

\subsection{The (generalized) Kasner solutions}
\label{SS:KASNER}
Before further discussing our results, we first formally
introduce the Kasner solutions. They can be expressed as
\begin{align} \label{E:KASNER}
	\mathring{\gfour} & = - dt^2 + \gKasner,
	\qquad \gKasner = \sum_{i=1}^3 t^{2q_i} (dx^i)^2,
	\qquad \sfKasner = A \ln t,
		\qquad (t,x) \in (0,\infty) \times \mathbb{T}^3,
\end{align}
where the constants $q_i$ are called the \emph{Kasner exponents} 
and $A \geq 0$ is a constant denoting the value of $\partial_t \phi$
at $t=1$. 
Note that we have the following identity (in a slight abuse of notation):
\begin{align} \label{E:KASNERSECONDFUNDONEUP} 
	t \SecondFundKasner_{\ j}^i 
	& = -\mbox{\upshape diag}(q_1,q_2,q_3).
\end{align}
The exponents $q_i$ and $A$ are constrained by the equations
 \begin{subequations} 
\begin{align} 
	\sum_{i = 1}^3 q_i & = 1, 
	\label{E:KASNERTRACECONDITION} 
	\\
	\sum_{i=1}^3 q_i^2 & = 1 - A^2.
	\label{E:KASNERHAMILTONIANCONSTRAINT}
\end{align}
\end{subequations}
\eqref{E:KASNERTRACECONDITION} corresponds to our 
gauge condition $\SecondFund_{\ a}^a(t,x) = - t^{-1}$, 
while \eqref{E:KASNERHAMILTONIANCONSTRAINT} is a consequence of 
the gauge condition $\SecondFund_{\ a}^a(t,x) = - t^{-1}$ and the
Hamiltonian constraint equation \eqref{E:GAUSSINTRO}. 

\begin{remark}
	For convenience, in \eqref{E:KASNER},
	we have written the Kasner metric in diagonal form.
	The diagonal form is a specific case of the more general form
	$- dt^2 + \sum_{I=1}^3 t^{2 q_I} \omega^I \otimes \omega^I$
	mentioned in Subsect.\ \ref{SS:OVERVIEWNONLINEAR} 
	(where the $\omega^I$ are, by assumption, orthogonal with respect to the Kasner metric itself).
	The diagonal form can always be achieved by a change of spatial coordinates.
\end{remark}

Exceptional cases aside, the Kasner solutions
have Big Bang singularities along the past boundary
$\lbrace t = 0 \rbrace$ where their Kretschmann scalars
$\Riemfour^{\alpha \beta \gamma \delta} \Riemfour_{\alpha \beta \gamma \delta}$
blow up like\footnote{One can compute that in terms 
of the Kasner exponents from \eqref{E:KASNER}, 
the Kretschmann scalar is equal to
$4 t^{-4}
		\left\lbrace
			\sum_{i=1}^3 q_i^4
			+ \sum_{1 \leq i < j \leq 3} q_i^2 q_j^2
			+ \sum_{i=1}^3 q_i^2
			- 2 \sum_{i=1}^3 q_i^3
		\right\rbrace 
\geq 4 t^{-4} \sum_{1 \leq i < j \leq 3} q_i^2 q_j^2.
$}
$t^{-4}$.
In our study of solutions to the linearized equations, an important role is played by  
the constants $q_{Max} > 0$ and $0 \leq \tracefreeparameter \leq \sqrt{\frac{2}{3}}$ defined by
\begin{subequations}
\begin{align} 
	q_{Max} 
	& := \max \lbrace q_1, q_2, q_3 \rbrace,
		\label{E:BIGGESTKASNEREXPONENT} \\ 
	\tracefreeparameter^2
	& := \sum_{i=1}^3 q_i^2 - \frac{1}{3} 
		= \sum_{i=1}^3 \left(q_i - \frac{1}{3} \right)^2
		= \frac{2}{3} - A^2.
		\label{E:TRACEFREEPARAMETER}
\end{align}
\end{subequations}
As we have mentioned, many of the results in this article hold only 
for nearly spatially isotropic Kasner backgrounds, that is,
when all three $q_i$ are near $1/3$.
It is important to note that it is not 
even possible to have all three $q_i > 0$ in the absence of matter
due to the Hamiltonian constraint.
The nearly spatially isotropic assumption is equivalent to
$\tracefreeparameter$ being small.
The analytic relevance of $\tracefreeparameter$ is: for Kasner metrics \eqref{E:KASNER},
the trace-free part of the second fundamental form $\SecondFundKasner_{\ j}^i$ 
of $\Sigma_t$ (see \eqref{E:SECONDFUNDDEF}),
defined by $\hat{\SecondFundKasner}_{\ j}^i := \SecondFundKasner_{\ j}^i - \frac{1}{3} \SecondFundKasner_{\ a}^a \ID_{\ j}^i 
= \SecondFundKasner_{\ j}^i + \frac{1}{3} t^{-1} \ID_{\ j}^i$
(where $\ID_{\ j}^i = \mbox{diag}(1,1,1)$ denotes the identity transformation),
verifies 
(with $|\hat{\SecondFundKasner}|_{\gKasner}^2 := \gKasner_{ab} \gKasner^{ij} \hat{\SecondFundKasner}_{\ i}^a \hat{\SecondFundKasner}_{\ j}^b$)
\begin{align} \label{E:KASNERTRACEFREEPARTOFSECONDFUNDAMENTALFORMSIZE}
	|\hat{\SecondFundKasner}|_{\gKasner} = \tracefreeparameter t^{-1}.
\end{align}
We again stress that the parameter $\tracefreeparameter$ drives the 
blowup-rate of our $L^2$-based energies for the linear solutions
as $t \downarrow 0$; see, for example,
inequality \eqref{E:ROUGHVERSIONENERGYESTIMATESNODERIVATIVESLOSS}

\subsection{Rough statement of the main linear results and further discussion}
\label{SS:ROUGHSTATEMENTLINEARRESULTS}
In this subsection, we summarize the main linear results of this paper.
We start by summarizing the approximate monotonicity identity;
see Theorem~\ref{T:CMCMONOTONICITYID} for the precise statement.
The proof is based on combining a collection of 
integration by parts identities in suitable proportions and 
judiciously using the constraint and lapse equations, 
which in total yields the cancellation of dangerous terms
and the emergence of favorable ones.

\begin{theorem}[\textbf{The approximate monotonicity identity} (Rough version)]
	\label{T:ROUGHVERSIONMONOTONICITY}
	Consider the Einstein-scalar field equations,
	written relative to CMC-transported-spatial coordinates
	(see Prop.\ \ref{P:EINSTEINSFCMC}),
	linearized (see Prop.\ \ref{P:LINEARIZEDCMCEQUATIONS})
	about any member of the Kasner family \eqref{E:KASNER},
	and with initial data given at time $1$.
	Then with ``Potential Terms''
	denoting the linearized lapse and its spatial derivatives,
	the spatial derivatives of the linearized scalar field,
	and the spatial derivatives of the linearized spatial metric;
	with ``Solution'' denoting the Potential Terms
	together with the linearized second fundamental form
	and the time derivative of the linearized scalar field;
	and with ``Data'' denoting quantities determined
	by the initial data,
	we have the following schematic identity,
	valid for $t \in (0,1]$:
	\begin{align} \label{E:ROUGHVERSIONMONOTONICITY}
		\int_{\Sigma_t}
			|\mbox{\upshape Solution}|^2
		\, dx
		& =
			\int_{\Sigma_1}
				\mbox{\upshape Data}
			\, dx
				\\
		& \ \
			-
			\int_{s=0}^t
				s^{-1}
				\int_{\Sigma_s}
					|\mbox{\upshape Potential Terms}|^2
				\, dx
			\, ds
			+
			\int_{s=0}^t
				s^{-1}
				\int_{\Sigma_s}
					\mbox{\upshape Error terms}
				\, dx
			\, ds.
			\notag
	\end{align}
\end{theorem}

Next, we roughly summarize the energy estimates that follow as a consequence
of Theorem~\ref{T:ROUGHVERSIONMONOTONICITY}.
See Theorem~\ref{T:L2MILDENERGYBLOWUPCMCGAUGE} for the precise statement
of the energy estimates.

\begin{theorem}[\textbf{Mildly singular energy estimates without derivative loss} (Rough version)]
	\label{T:ROUGHVERSIONENERGYESTIMATESNODERIVATIVESLOSS}
	Consider the linearized equations from the statement of Theorem~\ref{T:ROUGHVERSIONMONOTONICITY}.
	Let $\tracefreeparameter \geq 0$ be as defined by \eqref{E:TRACEFREEPARAMETER}.
	Then there exists an energy
	$\mathscr{E}_{(Total)}(t)$
	for the linear solution
	(see \eqref{E:TOTALENERGY} for the precise definition),
	whose \textbf{square} has the strength of the left-hand side of \eqref{E:ROUGHVERSIONMONOTONICITY},
	and constants $C > 0$ and $c > 0$
	such that the following estimate holds for $t \in (0,1]$
	whenever the Kasner background is nearly spatially isotropic 
	(that is, as long as $\tracefreeparameter$ is sufficiently small):
	\begin{align} \label{E:ROUGHVERSIONENERGYESTIMATESNODERIVATIVESLOSS}
	\mathscr{E}_{(Total)}(t)
	& \leq 
		C	
		\mathscr{E}_{(Total)}(t)(1)
		t^{- c \tracefreeparameter}.
\end
{align}
Moreover, the higher-order spatial derivatives 
of the linear solution verify similar energy estimates featuring the same
blowup-rate $t^{- c \tracefreeparameter}$.
\end{theorem}	

Finally, we roughly summarize our linear stability results, whose
proof relies on the energy estimates of
Theorem~\ref{T:ROUGHVERSIONENERGYESTIMATESNODERIVATIVESLOSS}.
See Theorem~\ref{T:CMCLINEARSTABILITY} for the precise statement.

\begin{theorem}[\textbf{Linear stability} (Rough version)]
	\label{T:ROUGHVERSIONLINEARSTABILITY}
	Let $N \geq 2$ be an integer.
	Consider a Kasner solution
	$(1,\gKasner_{ij},\SecondFundKasner_{\ j}^i,\sfKasner)$
	(where $1$ is the Kasner lapse)
	and let $\tracefreeparameter$
	be as in Theorem~\ref{T:ROUGHVERSIONENERGYESTIMATESNODERIVATIVESLOSS}.
	Consider data (at time $1$) for the linearized (about the Kasner solution) system with 
	enough regularity so that
	the norm $\highnorm{N}$ (see Def.\ \ref{D:NORMS})
	is initially finite, that is,
	$\highnorm{N}(1) < \infty$.
	Let
	$(\LapseRenormalized,\grenormalized_{ij},\LinSecondFund_{\ j}^i,\SFRenormalized)$
	be a solution to the linearized (about the Kasner solution)
	equations of Prop.\ \ref{P:LINEARIZEDCMCEQUATIONS},
	where
	$\LapseRenormalized$ is the linearized variable
	corresponding\footnote{See the beginning of Subsect.\ \ref{SS:LINEARIZEDQUANTITIES}
	for further discussion on the linearization procedure and the linearized variables.} 
	to $n-1$,
	$\grenormalized_{ij}$ is the linearized variable corresponding to
	$g_{ij} - \gKasner_{ij}$,
	$\LinSecondFund_{\ j}^i$
	is the linearized variable corresponding
	to $\SecondFund_{\ j}^i 
		- 
		\SecondFundKasner_{\ j}^i$,
	and
	$\SFRenormalized$
	is the linearized variable corresponding to
	$\phi - \sfKasner$.
	Then there exists a \underline{Kasner footprint state}
	(see below for further discussion)
	such that the linear solution converges towards it 
	as $t \downarrow 0$.
	Specifically, there exist a symmetric type $\binom{0}{2}$ tensorfield
	$h_{Regular} \in H_{Frame}^{N-1}(\mathbb{T}^3)$
	(the norms $\| \cdot \|_{H_{Frame}^M}$ are defined by \eqref{E:SOBOLEVNORMS}),
	a type $\binom{1}{1}$ tensorfield
	$K_{Bang} \in H_{Frame}^{N-1}(\mathbb{T}^3)$
	verifying $(K_{Bang})_{\ a}^a = 0$,
	and a constant $C > 0$
	such that
	if $\tracefreeparameter$ is sufficiently small,
	then the following estimates hold\footnote{On the left-hand sides of 
	\eqref{E:ROUGHMETRICCONVERGESQISQJ}-\eqref{E:ROUGHMETRICCONVERGESQIISNOTQJ}, we do not sum over $i$ or $j$.}
	for $t \in (0,1]$, $(i,j=1,2,3)$:
	\begin{subequations}
	\begin{align}
	\left\|
		\LapseRenormalized 
	\right\|_{H^{N-2}}
	& \leq \frac{C}{\tracefreeparameter} \highnorm{N}(1)
		t^{4/3 - c \tracefreeparameter},
		  \label{E:ROUGHSLAPSEHNMINUSTWOESTIMATE} 
				\\
		\left\|
			t^{-2 q_j} \grenormalized_{ij} 
			+ 
			2 \ln(t) (K_{Bang})_{\ j}^i 
			- 
			(h_{Regular})_{ij}
		\right\|_{H^{N-1}}
		& \leq C \highnorm{N}(1) t^{ 2/3 - c \tracefreeparameter},
		&& (\mbox{\upshape if} \ q_i = q_j),
		\label{E:ROUGHMETRICCONVERGESQISQJ}	\\
		\left\|
			t^{-2 q_j} \grenormalized_{ij} 
			+ \frac{1}{q_i - q_j} t^{2(q_i - q_j)} (K_{Bang})_{\ j}^i 
			- (h_{Regular})_{ij}
		\right\|_{H^{N-1}}
		& \leq C \highnorm{N}(1) t^{2/3 - c \tracefreeparameter},
		&& (\mbox{\upshape if} \ q_i \neq q_j),
		\label{E:ROUGHMETRICCONVERGESQIISNOTQJ}	\\
		\left\|
			t \LinSecondFund_{\ j}^i
			- 
			(K_{Bang})_{\ j}^i
		\right\|_{H^{N-1}}
		& \leq C \highnorm{N}(1) t^{2/3 - c \tracefreeparameter},
		\label{E:ROUGHSECONDFUNDCONVERGENES}	\\
		\left\|
			t \partial_t \SFRenormalized 
			- 
			\Psi_{Bang}
		\right\|_{H^{N-1}}
	& \leq C \highnorm{N}(1)
			t^{2/3 - c \tracefreeparameter},
			\label{E:ROUGHPARTIALTPHICONVERGES} \\
	  \left\|
			\partial_i \SFRenormalized 
			-  
			\ln(t) \partial_i \Psi_{Bang}
		\right\|_{H^{N-2}}
	& \leq C \highnorm{N}(1).
			\label{E:ROUGHPARTIALPHICONVERGES}
	\end{align}
\end{subequations}
\end{theorem}

We now explain the significance of the above convergence estimates,
starting with \eqref{E:ROUGHSLAPSEHNMINUSTWOESTIMATE}.
We first recall that in studying the nonlinear solution,
we decompose the spacetime metric as
$\gfour = - n^2 dt \otimes dt + g$
and that $\LapseRenormalized$ is the linearized variable corresponding to 
$n - 1$. Hence, \eqref{E:ROUGHSLAPSEHNMINUSTWOESTIMATE} shows that
at the linear level, the perturbation of the lapse
converges to $0$, that is, the lapse itself
converges at the linear level to the Kasner state $n = 1$.
To further explain the convergence results stated in Theorem~\ref{T:ROUGHVERSIONLINEARSTABILITY},
we first explain what we mean by a 
``Kasner footprint state.''
Specifically, 
we mean a collection of variables
$(\widetilde{\LapseRenormalized},\widetilde{\grenormalized}_{ij},\widetilde{\LinSecondFund}_{\ j}^i,\widetilde{\SFRenormalized})$
defined by
$\widetilde{\LapseRenormalized} = 0$,
$\widetilde{\LinSecondFund}_{\ j}^i = t^{-1} (K_{Bang})_{\ j}^i$,
$\widetilde{\grenormalized}_{ij} 
=    -
			2 \ln(t)
			t^{2 q_j}   
			(K_{Bang})_{\ j}^i 
			+
			t^{2 q_j} 
			(h_{Regular})_{ij}$
if $q_i = q_j$,
$
\widetilde{\grenormalized}_{ij} 
=
-
\frac{1}{q_i + q_j} t^{2q_i} (K_{Bang})_{\ j}^i 
+
t^{2 q_j} 
(h_{Regular})_{ij}
$
if $q_i \neq q_j$,
$\partial_t \widetilde{\SFRenormalized} 
= t^{-1} \Psi_{Bang}$,
and
$
\partial_i \widetilde{\SFRenormalized}
=  
\ln(t) \partial_i \Psi_{Bang}
$,
where
$(h_{Regular})_{ij}$, 
$(K_{Bang})_{\ j}^i$,
and
$\Psi_{Bang}$
are functions (of $x$) on $\mathbb{T}^3$.
Note that the above definitions of the Kasner footprint states 
are obtained by setting the terms inside the norms on the left-hand sides of the
estimates of Theorem~\ref{T:ROUGHVERSIONLINEARSTABILITY}
equal to $0$.
Roughly, Theorem~\ref{T:ROUGHVERSIONLINEARSTABILITY}
shows that the solutions to the
linearized equations of Prop.\ \ref{P:LINEARIZEDCMCEQUATIONS}
are asymptotic to a Kasner footprint state
$(\widetilde{\LapseRenormalized},\widetilde{\grenormalized}_{ij},\widetilde{\LinSecondFund}_{\ j}^i,\widetilde{\SFRenormalized})$
as $t \downarrow 0$.
Note that the Kasner footprint states
are generally \emph{not solutions to the linear equations}
of Prop.\ \ref{P:LINEARIZEDCMCEQUATIONS}.
For this reason, we will now explain why one might expect them
to emerge as the ``end states'' of linear solutions
and why the the $t$-behaviors stated
on the \emph{left-hand} side of the
estimates of Theorem~\ref{T:ROUGHVERSIONLINEARSTABILITY}
can be saturated.
We will give two explanations, the first
being completely heuristic and the second
one rigorously illustrating the saturation of the $t$-behavior.
First, one can easily check that given any
(sufficiently regular)
functions 
$(h_{Regular})_{ij}$,
$(K_{Bang})_{\ j}^i$,
and
$\Psi_{Bang}$
on $\mathbb{T}^3$,
the corresponding Kasner footprint state
is a solution to a \emph{truncated} version
of the linear equations 
of Prop.\ \ref{P:LINEARIZEDCMCEQUATIONS} in which all spatial derivative terms
are set equal to $0$. The truncated linear equations 
are linear analogs of the VTD equations mentioned 
at the end of the statement of Theorem~\ref{T:FLRWROUGHFIRSTVERSION}.
Thus, Theorem~\ref{T:ROUGHVERSIONLINEARSTABILITY} shows that linear solutions converge
towards solutions of the linear VTD system, which is quite natural since our
proof of Theorem~\ref{T:ROUGHVERSIONLINEARSTABILITY} relies on showing that
spatial derivative terms become negligible as $t \downarrow 0$.

Our second explanation concerning the end state behavior of
linear solutions is through the notion of
variations of one-parameter families of Kasner solutions.
For the sake of illustration, we only consider a one-parameter 
family of Kasner spatial metrics and mixed second fundamental forms.
That is, for convenience, in this part of the discussion, 
we ignore the scalar field by setting it equal to $0$;
this will not have any substantial effect on the main ideas behind our discussion.
Specifically, we consider the $\upalpha$-parameterized family
(where $\upalpha \in \mathbb{R}$)
defined by
$\gKasner_{ij}[\upalpha] := \mbox{\upshape diag}(t^{2Q_1[\upalpha]},t^{2Q_2[\upalpha]},t^{2Q_3[\upalpha]})$
and
$\SecondFundKasner_{\ j}^i[\upalpha] := - t^{-1} \mbox{\upshape diag}(Q_1[\upalpha],Q_2[\upalpha],Q_3[\upalpha])$,
where the $Q_i[\upalpha]$ are a one-parameter family of Kasner exponents.\footnote{The $Q_i$ must verify the constraint conditions \eqref{E:KASNERTRACECONDITION} 
and \eqref{E:KASNERHAMILTONIANCONSTRAINT}, but this is not important for our discussion
here.}
We assume that $Q_i[0] := q_i$, where the $q_i$ are constants.
For each fixed $\upalpha$, 
$(\gKasner_{ij}[\upalpha],\SecondFundKasner_{\ j}^i[\upalpha])$
is a solution to the nonlinear Einstein equations of Prop.\ \ref{P:EINSTEINSFCMC}
(where the lapse is identically $1$ and the scalar field is identically $0$).
Thus, $(\gKasner_{ij}[\upalpha],\SecondFundKasner_{\ j}^i[\upalpha])$ can be viewed
as a family of diagonal Kasner solutions that vary from ``point to point,'' that is, 
that vary with $\upalpha$,
in analogy with the $x$-dependent Kasner-type behavior 
of solutions to the nonlinear equations near singularities that was
predicted in \cites{vBiK1972,jB1978}.
To more fully explain the results of Theorem~\ref{T:ROUGHVERSIONLINEARSTABILITY}, 
we must also account for the following additional degrees of freedom: 
for each fixed $\upalpha$, 
we can perform a change of spatial coordinates.
We can account for this freedom by 
introducing a one-parameter family of
invertible matrices $M_{\ j}^i[\upalpha]$
(not depending on $t$)
that represent a change of spatial coordinates at each fixed $\upalpha$.
From these considerations, we see that a general picture of a
family of Kasner solutions varying from point to point 
can be captured by a one-parameter family of
Kasner solutions $(g_{ij}[\upalpha],k_{\ j}^i[\upalpha])$
of the form\footnote{For fixed $\upalpha$,
the form $g_{ij}[\upalpha]$ of
the Kasner spatial metric given by \eqref{E:ONEPARAMETERKASNERSPATIALMETRIC} 
is equivalent to the form
$\sum_{I=1}^3 t^{2 q_I} \omega^I \otimes \omega^I$
mentioned in Subsect.\ \ref{SS:OVERVIEWNONLINEAR}.}
\begin{align}
	g_{ij}[\upalpha] 
	& := M_{\ i}^a[\upalpha] M_{\ j}^b[\upalpha] \gKasner_{ab}[\upalpha],
		\label{E:ONEPARAMETERKASNERSPATIALMETRIC} \\
	k_{\ j}^i[\upalpha] 
	& := (M^{-1})_{\ a}^i[\upalpha] M_{\ j}^b[\upalpha] \SecondFundKasner_{\ b}^a[\upalpha].
	\label{E:ONEPARAMETERKASNERSECONDFUNDAMENTALFORM}
\end{align}
In what follows, we will 
use the notation 
$Q_i'[0] := \frac{d}{d \upalpha}Q_i[\upalpha]|_{\upalpha = 0}$,
and we use similar notation for other quantities that depend on $\upalpha$.
We now compute that 
\begin{align}
	g'[0]
	& = M^{\top}[0] \cdot \gKasner'[0] \cdot M[0] 
		+ 
		(M')^{\top}[0] \cdot \gKasner[0] \cdot M[0]
		+ 
		M^{\top}[0] \cdot \gKasner[0] \cdot M'[0],
			\label{E:KANSERMETRICVARIATION} \\
	k'[0] 
	& = M^{-1}[0] \cdot \SecondFundKasner'[0] \cdot M[0] 
		- 
		M^{-1}[0] \cdot M'[0] \cdot M^{-1}[0] \cdot \SecondFundKasner[0] \cdot M[0]
		+ 
		M^{-1}[0] \cdot \SecondFundKasner[0] \cdot M'[0],
		\label{E:KANSERSECONDFUNDVARIATION}
\end{align}
where
\begin{align}
	\gKasner[0]
	& = \mbox{\upshape diag}(t^{2q_1},t^{2q_2},t^{2q_3}),
		\label{E:NONVARIEDKASNERMETRIC} \\
	\gKasner'[0]
	& = 2 \ln(t) \mbox{\upshape diag}(t^{2q_1} Q_1'[0], t^{2q_2} Q_2'[0], t^{2q_3} Q_3'[0]),
		\label{E:KASNERNONDISTORTEDMETRICVARIATION} \\
	\SecondFundKasner[0]
	& = - t^{-1} \mbox{\upshape diag}(q_1,q_2,q_3),
		\label{E:NONVARIEDKASNERSECONDFUND} \\
\SecondFundKasner'[0] 
& = - t^{-1} \mbox{\upshape diag}(Q_1'[0],Q_2'[0],Q_3'[0]).
\label{E:KASNERNONDISTORTEDSECONDFUNDVARIATION}
\end{align}
In \eqref{E:KANSERMETRICVARIATION}-\eqref{E:KANSERSECONDFUNDVARIATION},
$\top$ denotes matrix transpose and $\cdot$ denotes matrix multiplication.
We now compare the above computations with
the results of Theorem~\ref{T:ROUGHVERSIONLINEARSTABILITY}.
The key point is to observe that
\emph{the variations 
$g'[0]$ and $k'[0]$
solve the linearized Einstein equations},
where the background spatial metric and second fundamental form 
(about which the equations are linearized)
are respectively $\gKasner[0]$ and $\SecondFundKasner[0]$.
Indeed, one way to obtain the linearized equations is 
by differentiating a one-parameter family of 
nonlinear solutions with respect to the parameter; see Subsect.\ \ref{SS:LINEARIZEDQUANTITIES}
for further discussion on this point.
Thus, to each one-parameter family 
$g_{ij}[\upalpha]$, $k_{\ j}^i[\upalpha]$
of the form \eqref{E:ONEPARAMETERKASNERSPATIALMETRIC}-\eqref{E:ONEPARAMETERKASNERSECONDFUNDAMENTALFORM},  
there exists an associated variation
$g'[0]$ and $k'[0]$
that solves the corresponding linearized equations.
Thus, the variations $g'[0]$ and $k'[0]$
are special (spatially homogeneous)
examples of the Kasner footprint states
stated in Theorem~\ref{T:ROUGHVERSIONLINEARSTABILITY}.
To further connect with the results of 
Theorem~\ref{T:ROUGHVERSIONLINEARSTABILITY},
we will investigate the structure of the variations.
From 
\eqref{E:KANSERSECONDFUNDVARIATION},
\eqref{E:NONVARIEDKASNERSECONDFUND},
and
\eqref{E:KASNERNONDISTORTEDSECONDFUNDVARIATION}, 
it follows that $t k'[0]$ is a $3 \times 3$ matrix with constant entries.
The key point is that this agrees with the fact that the limiting field
$K_{Bang}$ 
from Theorem~\ref{T:ROUGHVERSIONLINEARSTABILITY}
does not depend on $t$.
Moreover, from 
\eqref{E:KANSERMETRICVARIATION},
\eqref{E:NONVARIEDKASNERMETRIC},
and
\eqref{E:KASNERNONDISTORTEDMETRICVARIATION},
we see that the entries of the matrix $g'[0]$ are sums of two kinds of terms: 
pure power-law terms proportional to factors of type $t^p$ (where $p$ is a constant), 
which come from the factors of $\gKasner[0]$ in \eqref{E:KANSERMETRICVARIATION},
and similar power-law terms that are multiplied by a factor of
$\ln(t)$, which come from the factor of $\gKasner'[0]$ in \eqref{E:KANSERMETRICVARIATION}.
This agrees with the limiting behavior of
$\grenormalized_{ij}$ as $t \downarrow 0$
shown by the estimates
\eqref{E:ROUGHMETRICCONVERGESQISQJ}-\eqref{E:ROUGHMETRICCONVERGESQIISNOTQJ}.
To summarize, our consideration of one-parameter Kasner families led us to 
conclude that all variations $g'[0](t)$ and $k'[0](t)$
are spatially homogeneous Kasner footprint states that are solutions
to the linearization of the Einstein equations about the Kasner solution
$(\gKasner[0](t),\SecondFundKasner[0](t))$.
The results of Theorem~\ref{T:ROUGHVERSIONLINEARSTABILITY} 
show that for near-FLRW backgrounds, all linear solutions are asymptotic to 
$x$-dependent Kasner footprint states whose time behavior at each fixed $x$
is similar to the time behavior of one of the variations.
Similar results hold for the scalar field, as is
shown by the estimates \eqref{E:ROUGHPARTIALTPHICONVERGES} 
and \eqref{E:ROUGHPARTIALPHICONVERGES}.
In total, the above picture is closely aligned with the vision
of \cites{vBiK1972,jB1978}, in which the end state of nonlinear solutions 
was posited to be a family Kasner-like solutions parameterized
by the spatial point $x$.

Consistent with the nonlinear stable blowup result provided by
Theorem~\ref{T:FLRWROUGHFIRSTVERSION},
we could also extend the linear stability results of Theorem~\ref{T:ROUGHVERSIONLINEARSTABILITY}
to apply when the background solutions are near-FLRW as measured by a Sobolev norm
(and hence are spatially dependent).
We do not provide such an extension here 
because it would significantly lengthen the paper without contributing substantially
to the main ideas. A related issue connected to the nonlinear problem
is that in our proof of the existence and curvature blowup aspects
Theorem~\ref{T:FLRWROUGHFIRSTVERSION},
we do not rely on having precise knowledge 
of the solution's ``end state'' 
(that is, the asymptotics near $\lbrace t = 0 \rbrace$)
in advance; it suffices to control the difference between the perturbed solution 
and the FLRW solution. Put differently, in proving Theorem~\ref{T:FLRWROUGHFIRSTVERSION},
we could derive the sharp asymptotics/convergence results as $t \downarrow 0$
as a separate argument, after we have already shown that the solution
exists for $(t,x) \in (0,1] \times \mathbb{T}^3$ and that the Kretschmann scalar
blows up as $t \downarrow 0$.
For this reason, 
our proof of Theorem~\ref{T:FLRWROUGHFIRSTVERSION}
would allow for the following margin of error:
the proof would go through if we controlled the difference between
the perturbed solution and any near-FLRW Kasner solution
rather than the perturbed solution and the FLRW solution.

\subsection{Previous work on singularities}
\label{SS:PREVIOUSWORK}
Previous work has provided related results showing
the stability of singular solutions to the Einstein equations
in various contexts, but only under
under symmetry assumptions that reduce the problem to the study of 
$1+1$ dimensional PDEs\footnote{There also are stable singularity formation results in the class of
spatially homogeneous solutions (in which case the equations reduce to ODEs); see
\cite{aR2005b} or \cite{jWgE1997} for an overview.} 
\cites{pCjIvM1990,jIvM1990,hR2009b,hR2010}.
There also is a body of work that provides the construction of (but not the stability of)
singularity-containing solutions to select nonlinear Einstein-matter systems,
but only under the assumption of symmetry
\cites{sKjI1999,gR1996,sKaR1998,kA2000a,yCBjIvM2004,fS2002,fBpL2010c,eAfBjIpL2013b} 
and/or spatial analyticity \cites{lAaR2001,tDmHrAmW2002}.
Readers can also consult \cite{eAfBjIpL2013} for a more general well-posedness result
for singular initial value problems that applies to a class of symmetric hyperbolic
quasilinear systems in more than one spatial dimension. 
More precisely, in \cite{eAfBjIpL2013}, the authors prescribe Sobolev-class asymptotics featuring singular behavior. 
The main result of \cite{eAfBjIpL2013}
is the existence of a Sobolev-class solution that realizes the singular asymptotics.
We note, however, that \cite{eAfBjIpL2013} does not treat Einstein's equations.
A related approach to studying Big Bang singularities involves devising a 
formulation of Einstein's equations that allows one to solve a Cauchy problem
with initial data given on the singular hypersurface
$\lbrace t = 0 \rbrace$ itself;\footnote{This method is based on
formulating the equations in terms of a rescaled metric, conformal to the physical spacetime metric,
in such a way that the rescaled metric remains regular throughout the entire evolution. 
As such, this method can be viewed as an extension of
Friedrich's \emph{conformal method} \cites{hF1986,hF2002}.} 
see, for example,
\cites{kApT1999a,cCpN1998,rN1993a,rN1993b,pT1990,pT1991,pT2002}. 
In some cases, these works included a proof that the singular solutions exhibit AVTD behavior.
Readers can consult \cite{aR2000b} for a precise comparison of these results 
as well as an extension of them to prove the existence of singular solutions to the Einstein-vacuum equations
with Gowdy symmetry.\footnote{Gowdy solutions are a subset of the $\mathbb{T}^2$-symmetric solutions
characterized by the vanishing of the \emph{twist constants} 
$(\gfour^{-1})^{\mu \mu'} \pmb{\epsilon}_{\alpha \beta \mu \nu} \mathbf{X}^{\alpha} \mathbf{Y}^{\beta} \Dfour_{\mu'} \mathbf{X}^{\nu}$
and 
$(\gfour^{-1})^{\mu \mu'} \pmb{\epsilon}_{\alpha \beta \mu \nu} \mathbf{X}^{\alpha} \mathbf{Y}^{\beta} \Dfour_{\mu'} \mathbf{Y}^{\nu}$,
where $\pmb{\epsilon}$ is the volume form of $\gfour$ and $\mathbf{X}$ and $\mathbf{Y}$ are the Killing fields corresponding to the two symmetries.}

In contrast to the regular Cauchy problem studied here and in the companion article \cite{iRjS2014b},
the above works are based on prescribing the asymptotics as $t \downarrow 0$ and then constructing a solution
that achieves those asymptotics. Most of those works are based on
solving a Fuchsian PDE system that is singular at $\lbrace t = 0 \rbrace$. 
We now describe some aspects of the Fuchsian approach.
A representative work is \cite{eAfBjIpL2013b}, 
in which the authors construct singular solutions to the Einstein-vacuum equations\footnote{More general Fuchsian systems 
in one spatial dimension are also treated in \cite{eAfBjIpL2013b}.}
with $\mathbb{T}^2$ symmetry under the polarized or half-polarized condition.
In Sect.\ \ref{S:ENDSTATESCOMMENTS},
 we provide a simple model problem suggesting
that results similar to those of \cite{eAfBjIpL2013b} 
might also hold for the Einstein-scalar field system
without symmetry assumptions.
The Fuchsian PDEs\footnote{Specifically, the PDEs are the $\mathbb{T}^2$-symmetric polarized or half-polarized Einstein-vacuum equations 
in areal coordinates with the singularity at $\lbrace t=0 \rbrace$.}
treated in \cite{eAfBjIpL2013b} are of the form
\begin{align} \label{E:FUCHSIAN}
	A^0(t,x,u) t \partial_t u + A^1(t,x,u)t \partial_x u + B(t,x,u) u = f(t,x,u),
\end{align}
where $u$ is the array of unknowns, $A^{\alpha}$ and $B$ are symmetric matrices
(the energy estimates rely on the symmetric hyperbolic framework), 
and $f$ is an array, all of which verify a collection of technical assumptions.
The analysis in \cite{eAfBjIpL2013b} is based on splitting the solution
as $u = u_0 + w$, where $u_0$ is the ``leading order''
part and $w$ is an error term that one would like to show is small compared to $u_0$ as $t \downarrow 0$.
An important technical assumption made in \cite{eAfBjIpL2013b}, which is used for deriving energy estimates, 
is that for small $w$, one can split $A^0(t,x,u_0 + w) = A_0^0(x,u_0) + A_1^0(t,x,u_0+w)$,
where $A_0^0(x,u_0)$ is symmetric positive definite, 
and the map $w \rightarrow A_1^0(t,x,u_0+w)$ maps certain time-weighted Sobolev spaces
into other time-weighted Sobolev spaces.
There are various methods for constructing $u_0$.
The most relevant way in the context of the present article is
to choose $u_0$ to be a prescribed solution to a truncated ``VTD version'' of \eqref{E:FUCHSIAN}
in which the spatial derivative terms are discarded. 
This approach is complementary to the one taken in the present 
article and \cite{iRjS2014b}, in which we show that AVTD behavior 
\emph{dynamically emerges}
in solutions to the non-truncated equations.
From the VTD system and \eqref{E:FUCHSIAN}, 
one computes that the error term $w$ solves an ``error equation'' 
depending on $u_0$.
The main result of \cite{eAfBjIpL2013b} is that under suitable additional assumptions, 
there exists a solution $w$ 
to the error equation that becomes small relative to $u_0$ as $t \downarrow 0$
and that $w$ is unique within appropriate time-weighted Sobolev spaces.
The main idea of the proof is to derive uniform a priori
symmetric hyperbolic energy estimates
for a sequence $\lbrace w_n \rbrace_{n=1}^{\infty}$ of error equation solutions
on intervals of the form $[t_n,\updelta]$. More precisely, the $w_n$ 
solve a standard symmetric hyperbolic Cauchy problem (to the future) with $0$ initial data at time $t_n$.
Here, $\updelta > 0$ is a small constant and $\lbrace t_n \rbrace_{n=1}^{\infty}$ is a sequence of times decreasing to $0$.
A key aspect of the analysis in \cite{eAfBjIpL2013b}
is that the authors were able to close their estimates by
inserting time weights by hand into the energies. 
More precisely, in the approach of \cite{eAfBjIpL2013b},
one derives energy estimates for
$t^{-P} w_n$, where $t^{-P}$ is a diagonal matrix whose non-zero
entries are well-chosen negative powers of $t$ that are allowed to depend on $x$ (that is, $P = P(x)$).
Another aspect of the approach of \cite{eAfBjIpL2013b} is that the energies are
weighted by an additional overall scalar factor of 
$\displaystyle e^{-\upkappa t^{\upgamma}}$, where $\upkappa$ and $\upgamma$ are positive constants.
The time weights must be chosen to be compatible with the nonlinearities
in the sense that the nonlinear error integrals arising in the energy estimates 
must be controllable. When successfully implemented, 
this leads to controlled energy growth towards the future 
(away from the singularity) in a neighborhood of the singularity. 
In particular, for well-chosen $t$-weights 
(as we illustrate in Remark~\ref{R:FREEDOMINCHOOSINGEXPONENTS},
there is some freedom in choosing them),
one can derive uniform estimates for the $\lbrace w_n \rbrace_{n=1}^{\infty}$
showing that the weighted energies cannot grow too fast towards the future;
see Sect.\ \ref{S:ENDSTATESCOMMENTS} for a very simple linear model problem.
Then through a standard limiting procedure, one can produce a 
solution $w$ to the error equation that exists on 
the interval $(0,\updelta]$, and it is unique within suitable
time-weighted Sobolev spaces.

Although the Fuchsian approach furnishes the existence of 
a set of solutions with singularities, it is inadequate for treating
the true stability problem of solving down towards
$\lbrace t = 0 \rbrace$ starting from Cauchy data for $u$ given along
a hypersurface $\lbrace t = \mbox{\upshape const} \rbrace$ with $\mbox{\upshape const} > 0$.
One difficulty that we encounter in our study of the Einstein equations, 
which we stressed at the beginning,
is that in order to synchronize the singularity across space,
one cannot work with a purely hyperbolic
formulation of the equations such as the one afforded by 
wave coordinates; gauges involving an 
infinite speed of propagation, such as the elliptic and parabolic 
ones for the lapse employed in the present article and in \cite{iRjS2014b}, 
seem essential. Hence, our approach to proving stability 
lies outside of the standard Fuchsian framework, 
which applies only to hyperbolic equations.
Moreover, the Fuchsian strategy of inserting suitable time weights
by hand into the energies is not sufficient for deriving our stability results
because some of the terms in the equations are too singular to be treated in this fashion;
see our discussion in Subsect.\ \ref{SS:OTHERMATTERMODELSECT} 
for further discussion on this point, where we highlight similar difficulties
that would arise in an attempt to extend our approach to prove stability results for far-from-FLRW solutions.
For near-FLRW solutions, our approach is viable only because of the
cancellations that occur in our approximate monotonicity identity,
which are tied to the special structure of the Einstein-scalar field system in our gauges.


The scalar field and stiff fluid
matter models have some special properties that 
we exploit in deriving our results. 
We describe some of these properties in more detail in Subsect.\ \ref{SS:OTHERMATTERMODELSECT}. 
In particular, we expect that our approximate monotonicity/stability 
results do not hold for general matter models. 
Actually, as we now explain, 
for certain fluid models,
Ringstr\"{o}m obtained rigorous results 
showing that solutions behave in a drastically non-monotonic fashion.
In \cite{hR2001}, Ringstr\"{o}m studied fluids verifying the
equation of state $p = c_s^2 \rho$, where the constant $c_s$ 
verifies $0 < c_s \leq 1$ and physically represents the speed of sound. 
For the Euler-Einstein equations with 
a sub-stiff equation of state (that is, with $0 < c_s < 1$), 
he showed that spatially homogeneous solutions
with Bianchi IX symmetry\footnote{Members of the Bianchi symmetry classes are spatially homogeneous and hence the corresponding solutions depend on only a time variable.
For a precise definition of these symmetry classes and the others that we mention, 
readers can consult \cite{pCgGdP2010}.}
generically (that is, for non-Taub solutions) have limit points
in the approach towards the singularity that must be either \emph{vacuum} Bianchi type I (that is, vacuum Kasner), vacuum Bianchi type $\mbox{VII}_0$, or vacuum Bianchi type II.
In particular, Ringstr\"{o}m's work showed that a sub-stiff fluid has a negligible effect on Bianchi IX solutions 
near the singularity. Furthermore, he showed that almost all such solutions are \emph{oscillatory} in the sense that there are at least three distinct limit points, \emph{which stands in stark contrast to the approximately monotonic behavior
of our linear solutions and the nonlinear solutions in \cite{iRjS2014b}.}

Ringstr\"{o}m's work \cite{hR2001} also applied to the Einstein-vacuum equations in Bianchi IX
symmetry and thus yielded the first examples of the oscillatory behavior
conjectured in the work \cite{vBiKeL1970} of Belinsky, Khalatnikov, and Lifschitz (BKL). 
Specifically, in \cite{vBiKeL1970}, the authors gave heuristic arguments suggesting that 
general solutions to the Einstein-vacuum equations 
containing incomplete timelike geodesics
should exhibit highly oscillatory behavior 
near the boundary where the geodesics terminate.
Moreover, their arguments suggested that
the boundary should be a spacelike singularity.
These so-called ``BKL conjectures''\footnote{The statements in \cite{vBiKeL1970} are somewhat vague
and thus it is imprecise to refer to them as ``conjectures.''} 
have been seminal in stimulating the investigation
of solutions to Einstein's equations near singularities.
However, as we now explain, 
despite Ringstr\"{o}m's work,
there is immense controversy surrounding the conjectures.
First, they are false as stated because of, for example, the existence of
Taub solutions, which develop a Cauchy horizon\footnote{Roughly, a Cauchy horizon is a 
boundary along which the solution remains regular
but beyond which it cannot be continued uniquely as a solution due to
lack of information for how to continue.} 
rather than a true singularity.
One might be tempted to weaken the conjectures
by replacing the phrase ``general solutions'' with ``generic solutions.''
However, Luk has constructed
\cite{jL2013} a class of solutions to the Einstein-vacuum equations
without symmetry assumptions such that the boundary of the maximal development
contains \emph{a null portion} along which the metric
remains $C^0$ but its Christoffel symbols blow-up in $L^2$.
His examples, which are stable in a certain sense, contradict the 
BKL vision of spacelike singularities.
Moreover, outside of the class of spatially homogeneous solutions,
there are currently no examples of Einstein-vacuum solutions 
that are rigorously known to exhibit the kind of oscillatory behavior
near a singularity conjectured in \cite{vBiKeL1970}.
In total, given the present-day state of knowledge,
it is not clear to what extent the vision of BKL
is realized in Einstein-vacuum solutions.

In the opposite direction,
we recall the aforementioned work of
Belinsky and Khalatnikov \cite{vBiK1972},
who were the first to suggest
the existence of non-spatially homogeneous
\emph{approximately monotonic} singular solutions to the Einstein-scalar field system.
In a later article \cite{jB1978}, Barrow argued that fluids verifying the equation of state $p = c_s^2 \rho$ 
(where $c_s$ is a non-negative constant)
should induce a similar effect if and only if $c_s = 1;$ he referred to the 
mollifying effect of a stiff fluid as \emph{quiescent cosmology}.
The first rigorous construction of such solutions 
without symmetry was provided by the aforementioned work of Andersson and Rendall \cite{lAaR2001}.
They constructed a family of spatially analytic solutions to the Einstein-scalar field
and Einstein-stiff fluid systems
that have Big Bang singularities and that exhibit approximately monotonic behavior near them.
Their proof was based on a two-step process. In the first step, they constructed
a family of spatially analytic solutions to VTD equations, which were obtained by throwing away the spatial derivative terms
from the Einstein-matter equations.\footnote{In \cite{lAaR2001}, the Einstein equations were formulated
relative to a Gaussian coordinate system in which the spacetime metric
takes the form $\gfour = - dt^2 + g_{ab} dx^a dx^b$.} 
In the second step, they constructed spatially analytic solutions
to the Einstein-matter equations by writing the true solution as a
solution to the VTD equations plus error terms that were shown, by Fuchsian analysis, to go to $0$ as $t \downarrow 0$.
The results of \cite{lAaR2001} were extended to higher dimensions and other matter models in \cite{tDmHrAmW2002}.
The family of solutions constructed in this fashion
is large in the sense that its number of degrees of freedom 
coincides with the number of free functions
in the Einstein initial data. However, since the results
are based on prescribing the asymptotics near the Big Bang 
within the class of spatially analytic solutions, 
they are not true stable singularity formation results. 
In particular, the work left open the possibility that the
map from the set of spatially analytic asymptotic states realized in \cite{lAaR2001}
to the set of Cauchy data (say at $t=1$) 
might be highly degenerate in the sense that it 
cannot be extended as a map (with reasonable properties)
between more physically relevant function spaces such as
Sobolev spaces; see, however, the discussion in Sect.\ \ref{S:ENDSTATESCOMMENTS}.
The primary ingredient needed to upgrade
the work of Andersson and Rendall to a true stable singularity formation result
corresponding to solving a regular Cauchy problem
is a suitable statement of linear stability,
strong enough to control the nonlinear terms.
Our linear stability result (Theorem~\ref{T:CMCLINEARSTABILITY}) provides this missing ingredient 
in the near-FLRW case.

\subsection{Comments on other matter models, higher dimensions, and the analysis of far-from-FLRW-solutions}
\label{SS:OTHERMATTERMODELSECT}
The scalar field and stiff fluid
matter models have two important properties, described in the next paragraph, 
that allow us to prove the stability results of the present paper
and those of \cite{iRjS2014b}.
We anticipate that other matter models with 
similar properties might allow for proofs of similar results. 
Readers can consult \cite{tDmHrAmW2002} for a class of candidate matter models, 
where the authors used Fuchsian techniques to construct families of non-spatially homogeneous
solutions with Big Bang singularities to various nonlinear Einstein-matter systems.
We note that the authors' construction also applied to the 
Einstein-vacuum equations in $10$ or more spatial dimensions
and thus yielded rigorous examples of the 
non-oscillatory and non-spatially-homogeneous
solutions that were heuristically argued to exist
in \cite{jDmHpS1985}. The existence of these spatially inhomogeneous Kasner-like vacuum solutions 
is relevant for the discussion three paragraphs below.

The first important property of the scalar field and stiff fluid matter models is 
simply that they allow for the existence of spatially isotropic
and nearly spatially isotropic Kasner solutions to the Einstein-matter system.
We recall that nearly spatially isotropic Kasner solutions 
have second fundamental forms with trace-free parts that blow up at the rate $\tracefreeparameter t^{-1}$,
where $\tracefreeparameter$ is small (see \eqref{E:KASNERTRACEFREEPARTOFSECONDFUNDAMENTALFORMSIZE}), 
and that this blowup-rate ultimately leads to the mild energy blowup-rate 
\eqref{E:ROUGHVERSIONENERGYESTIMATESNODERIVATIVESLOSS}.
We now contrast this against the case of the Einstein-vacuum equations in three spatial dimensions.
In vacuum, we have $A=0$ in \eqref{E:KASNERHAMILTONIANCONSTRAINT}
and thus \eqref{E:TRACEFREEPARAMETER} and \eqref{E:KASNERTRACEFREEPARTOFSECONDFUNDAMENTALFORMSIZE}
imply that the trace-free part of the Kasner second fundamental form
blows up at the rate $\sqrt{\frac{2}{3}}t^{-1}$. 
Combining this blowup-rate with the methods of this paper, one would only be able to derive 
energy estimates in the spirit of \eqref{E:ROUGHVERSIONENERGYESTIMATESNODERIVATIVESLOSS}
showing that the energy blows up like $t^{-c \sqrt{2/3}}$ as $t \downarrow 0$.
Unfortunately, such a bound for the energy does not appear to be useful
for controlling error terms in the nonlinear problem.
In fact, an energy blowup-rate of $t^{-c \sqrt{2/3}}$ seems to be insufficient
even for proving linear stability results of the type proved 
in Theorem~\ref{T:ROUGHVERSIONLINEARSTABILITY};
see the next paragraph for further discussion on this point.
The second important property of the scalar field matter model
is that \emph{its time derivatives do not appear} in the evolution equations 
for the metric
(equations \eqref{E:PARTIALTGCMC}-\eqref{E:PARTIALTKCMC})
nor in the elliptic PDE for the lapse (equation \eqref{E:LAPSE}).
This property, which \emph{for the scalar field matter model requires the assumption of three spatial dimensions}, 
is closely tied to the fact that the characteristics 
of the scalar field agree with those of the Einstein field equations
(that is, the characteristics for the Einstein-scalar field system 
are precisely the null hypersurfaces relative to $\gfour$).
This property plays a critically important role in allowing us to prove 
our stability results because to close our estimates, 
we rely on the fact that spatial derivatives
are small compared to time derivatives, at least at the lower derivative levels.
Although the stiff fluid matter model exhibits similar good properties,
fluids verifying the sub-stiff equation of state 
$p = c_s^2 \rho$ with $0 < c_s < 1$ do not enjoy these properties, 
even if the fluid is irrotational 
(roughly because the sound cones are necessarily distinct from the gravitational null cones
in the sub-stiff case).
This is consistent with the oscillatory behavior for solutions to the Euler-Einstein system
observed by Ringstr\"{o}m \cite{hR2001} 
in the Bianchi IX symmetry class when $0 < c_s < 1$
(see the discussion in Subsect.\ \ref{SS:PREVIOUSWORK}).

We now further explain some of the obstacles to deriving stability results 
for the Einstein-scalar field system in the far-from-FLRW case
(e.g., when $\tracefreeparameter$ is no longer small in the linear problem).
Although our methods could be used to obtain 
estimates for solutions to the linearized systems,
they do not seem to be strong enough to allow for a proof of 
linear stability or stable blowup in the nonlinear problem.
Our goal is to highlight why, for parameters corresponding to
far-from-spatially isotropic Kasner backgrounds, 
our methods do \emph{not} allow us to prove that 
$|t \LinSecondFund_{\ j}^i|$ remains uniformly bounded over the interval $t \in (0,1]$,
where $\LinSecondFund$ is the linearized second fundamental form variable.
In the nonlinear problem, the same difficulty would arise,
and it is tantamount to not even being able to recover 
(in the context of a bootstrap argument) 
the blowup-rate of $t^{-1}$
exhibited by the trace-free part of the second fundamental form of a Kasner metric.
In the nonlinear problem, such a bad estimate would lead (by a Gronwall estimate) 
to energy estimates that are drastically worse than \eqref{E:ROUGHVERSIONENERGYESTIMATESNODERIVATIVESLOSS}: 
the top-order energies would be allowed to blow up faster than 
$data \times t^{-C}$ for all constants $C > 0$,
where ``$data$'' denotes a term that is controlled by the initial data.
Consequently, our entire approach to linear and nonlinear stability would break down,
and in the nonlinear problem,
we would not even be able to show that 
the solution exists near $\lbrace t = 0 \rbrace$.
To explain the source of the difficulty, we first explain how
we prove the uniform boundedness of
$|t \LinSecondFund_{\ j}^i|$
in the nearly spatially isotropic case.
The main idea is that we can use the evolution equation for
$\LinSecondFund_{\ j}^i$ (see \eqref{E:LINEARIZEDKEVOLUTION})
and the mildly singular energy estimates of Theorem~\ref{T:L2MILDENERGYBLOWUPCMCGAUGE}
to prove (see \eqref{E:PARTIALTSECONDFUNDHNMINUSONE}) the estimate 
$|\partial_t(t \LinSecondFund_{\ j}^i)| \lesssim data \times t^{-1/3 - c \tracefreeparameter}$.
The key point is that the right-hand side is integrable in time over the time interval $(0,1]$
for $\tracefreeparameter$ small. That is, if $\tracefreeparameter$ is small,
then we can express $t \LinSecondFund_{\ j}^i$
as an integral of
$\partial_t(t \LinSecondFund_{\ j}^i)$ and use the time-integrability
to obtain the desired bound
$|t \LinSecondFund_{\ j}^i| \lesssim data$.
In contrast, if $\tracefreeparameter$ is large,
then the bound $|\partial_t(t \LinSecondFund_{\ j}^i)| \lesssim data \times t^{-1/3 - c \tracefreeparameter}$
does \emph{not} imply the time-integrability of
$|\partial_t(t \LinSecondFund_{\ j}^i)|$,
and thus our approach does not work in its current form.

Finally, we make some comments on extending our stability results to higher dimensions. 
For brevity, we limit our discussion to the Einstein-scalar field and
Einstein-vacuum systems. For the Einstein-scalar field system in any number of spatial dimensions, 
we expect that the proofs of our linear and nonlinear stability results 
(in the near-FLRW setting)
would go through without any significant changes. Moreover, in the case of 
the Einstein-vacuum equations in $n$ spatial dimensions with
$n$ sufficiently large, there exists a class of Kasner solutions for which
it might be possible to prove sufficiently strong versions of linear stability 
(similar to the linear stability results of the present paper),
suitable for deriving nonlinear stable blow-up results like those proved in \cite{iRjS2014b}. 
As we mentioned above, the existence of (but not the stability of)
non-spatially homogeneous solutions to the Einstein-vacuum equations with Big Bang singularities
has already been shown in \cite{tDmHrAmW2002} when $n \geq 10$.
We now provide some motivation for our speculation on the existence of stable Einstein-vacuum singularities.
First, we note that it is possible to derive an approximate monotonicity 
identity for the linearized (around a vacuum Kasner solution) 
Einstein-vacuum equations that parallels 
the results for the Einstein-scalar field model
provided by 
Theorems~\ref{T:CMCMONOTONICITYID} and \ref{T:PARABOLICMONOTONICITYID}.
More precisely, the approximate monotonicity identities of
Theorems~\ref{T:CMCMONOTONICITYID} and \ref{T:PARABOLICMONOTONICITYID}
remain valid in the vacuum case;
just set the scalar field and its amplitude $A$ equal to $0$ in the equations.
However, one faces the difficulty that in vacuum,
the trace-free part of the Kasner second fundamental form has
large size 
$
(1 - 1/n)^{1/2} t^{-1}
$, 
a fact that follows from the vacuum Kasner exponent constraints:
\begin{align} \label{E:VACUUMKASNEREXPONENTCONSTRAINTSNDIMENSIONS}
&\sum_{i=1}^n q_i = 1,
&
&
\sum_{i=1}^n q_i^2 = 1.
\end{align}
The expression 
$
(1 - 1/n)^{1/2} t^{-1}
$ 
suggests that the energy blowup-rate for solutions to the linearized 
Einstein-vacuum equations becomes worse as $n \to \infty$, which
seems to be an obstacle to proving stable blowup.
Nonetheless, it might be possible to overcome this difficulty, at least in a certain regime.
The main idea is the following observation: the proof of the energy blowup-rate 
can be somewhat sharpened compared to the proof that leads 
to inequality \eqref{E:ROUGHVERSIONENERGYESTIMATESNODERIVATIVESLOSS}.
More precisely, many of the error terms that contribute to 
the blowup-rate of the energy can be controlled by 
the \emph{eigenvalues} of the second fundamental form\footnote{Specifically, we mean the version of the second
fundamental form with one index up and one down.} 
and its trace-free part.
In particular, a more careful analysis, 
not carried out in this article,\footnote{See the proof of inequality \eqref{E:N2POINTWISE}
regarding the role that the eigenvalues play in deriving energy estimates.} 
shows that most error terms in the energy estimates 
that involve the second fundamental form can cause the energies
energies to blow up at worst like
$\displaystyle \lesssim data \times t^{-c \upalpha}$, 
where $c > 0$ is a universal constant independent of $n$
and $\displaystyle \upalpha := \max_{i=1}^n \lbrace |q_i| \rbrace$.
Moreover, it is not difficult to see that there exists a family of vacuum Kasner solutions such that
$\displaystyle \upalpha \downarrow 0$ as $n \to \infty$.
However, there are a few anomalous terms in the energy estimates that 
could in principle lead to a blowup-rate that is worse than $data \times t^{-c \upalpha}$,
and these terms are therefore a potential obstacle for proving stability.
If one were able to sufficiently control the anomalous terms, then
we expect that one would be able to prove that Kasner solutions with $\upalpha$ 
sufficiently small\footnote{One can think of $1/n$ as a parameter that 
one would like to choose to be sufficiently small to close the estimates.}
are linearly and nonlinearly stable in a neighborhood of the Big Bang by using the methods of 
the present article and those of \cite{iRjS2014b}.
We note that for fixed large $n$, only a small portion of the vacuum Kasner solutions could in principle be 
shown to be stable through this approach. 
If the argument goes through, then it
would also be interesting to discover the threshold value of $n$ beyond which
the stable Kasner solutions exist; it is conceivable that the threshold value $n \geq 10$
from \cite{tDmHrAmW2002}, which is sufficient for the \emph{existence} of non-spatially homogeneous solutions,
is not large enough to imply their \emph{stability}.

\subsection{A related instance in which monotonicity led to global results}
\label{SS:MONOTONICITYINRELATEDCONTEXTS}
We now describe the work \cite{lAvM2011} by Andersson and Moncrief,
in which they proved global existence results for the Einstein-vacuum equations 
using techniques that have some overlap 
with the ones used in the present article and in \cite{iRjS2014b}.
In the next paragraph, we compare and contrast the approach of
\cite{lAvM2011} with that of the present work. We first describe their
result in more detail.
In \cite{lAvM2011}, the authors proved a future-global existence theorem
(that is, in the expanding direction) for perturbations
of spatially compact versions of 
FLRW-like vacuum spacetimes in $1 + m$ dimensions for $m \geq 3$.
The background solutions were of the 
``continuously self-similar'' form 
$
\displaystyle
-dt^2 + \frac{t^2}{m^2} \gamma
$, 
where the spatial metric $\gamma$ verifies the
\emph{Einstein condition}
$
\displaystyle
\Ric = - \frac{m-1}{m^2} \gamma
$, where $\Ric$ is the Ricci curvature of $\gamma$.
Readers can also consult \cites{lAvM2004,mR2009} for proofs of the results of 
\cite{lAvM2011} in the case $m=3$, where
unlike in \cite{lAvM2011} and the present article,
the latter two works rely on curvature-based energies
constructed from the Bel-Robinson tensor.
Andersson and Moncrief made some technical assumptions
on $\gamma$, notably one\footnote{The authors also made additional assumptions. 
Specifically, they assumed that either the moduli space of $\gamma$ is trivial or
that $\gamma$ is contained in an integrable moduli space of Einstein structures.} 
that they called being ``stable.''
This condition states that the eigenvalues of the operator 
$
\displaystyle
h_{ij} \rightarrow 
-\Delta_{\gamma} h_{ij} 
- 2 R_{i \ j}^{\ a \ b} h_{ab},
$
which appears in linearized versions of the evolution equations,
are non-negative. Here, $h_{ij}$ is a symmetric type $\binom{0}{2}$ tensor
and $R_{i \ j}^{\ a \ b}$ is the Riemann curvature tensor of $\gamma$.
In our proof of nonlinear stable blowup \cite{iRjS2014b},
terms like $2 R_{i \ j}^{\ a \ b} h_{ab}$ also appear, but we 
are able to treat them as lower-order nonlinear error terms.
That is,  we do not have to work with combinations
such as $-\Delta_{\gamma} h_{ij} - 2 R_{i \ j}^{\ a \ b} h_{ab}$;
see Subsect.\ \ref{SS:NONLINEARERRORINTEGRALS} for an overview of 
how we handle nonlinear error terms.
In \cite{lAvM2011}, the authors also proved that a rescaled version of the perturbed spatial metric
converges to an element of the moduli space
of $\gamma$. In the case $m=3$, 
the Einstein condition implies that 
$\gamma$ has constant negative sectional curvature,
and Mostow's rigidity theorem implies that 
the moduli space is trivial. Hence, the rescaled solution in fact converges to the background solution.
In contrast, in our nonlinear results \cite{iRjS2014b}
and in the linear convergence results of Theorem~\ref{T:CMCLINEARSTABILITY}, 
the family of possible end states 
(corresponding to the asymptotic behavior of the solution near the Big Bang)
is much larger. In the nonlinear problem, the family of course includes members of the Kasner family \eqref{E:KASNER}.
However, as we described below Theorem~\ref{T:ROUGHVERSIONLINEARSTABILITY}, 
even for the linear problem,
it also includes\footnote{More accurately, we do not rigorously prove
that the family includes $x$-dependent end states. 
However, we recall here the work \cite{lAaR2001} described in Subsect.\ \ref{SS:PREVIOUSWORK}, 
in which Andersson and Rendall constructed solutions with end states that are analytic in $x$ 
(with non-trivial $x$ dependence). Based on their work, 
our results here, and the results of \cite{iRjS2014b},
we expect that it might be possible to remove the analyticity assumption 
(perhaps only in the near-FLRW regime), 
which would yield new information about the set of \emph{achievable} end states; 
see also Sect.\ \ref{S:ENDSTATESCOMMENTS}.} 
a much larger family of ``$x$-dependent'' Kasner-like states.
As Andersson and Moncrief stated in \cite{lAvM2011}, 
their work is closely related to the Fisher--Moncrief work \cite{aFvM2002}, 
in which the authors carried out the linear stability analysis.
Specifically, 
in \cite{aFvM2002},
Fischer-Moncrief found a reduced Hamiltonian description of the 
Einstein-vacuum flow 
(see also the works 
\cites{aFvM1994,aFvM2000a,aFvM2000b,aFvM2001,aFvM1997,aFvM2002,vM1990,vM1989b,lAvMaT1997}
for related results)
that applied to a family of spacetimes containing CMC hypersurfaces.
Their Hamiltonian was the \emph{volume} functional of constant-time hypersurfaces $\Sigma_t$,
where, as in the present article, the $\Sigma_t$ were CMC hypersurfaces.
They showed that the Hamiltonian is monotonic along the flow 
of their reduced equations, 
that its critical points are precisely the continuously self-similar metrics
$
\displaystyle
-dt^2 + \frac{t^2}{m^2} \gamma
$ 
mentioned above
(where $\gamma$ verifies the Einstein condition),
and, crucially for the linear stability analysis 
(on which global existence result \cite{lAvM2011} relied), 
that its second variation is positive definite
when $\gamma$ is stable in the sense described above. 

The analysis in \cite{lAvM2011} has some features in common with the present work,
including its reliance on CMC foliations to reveal monotonicity
and its focus on studying the solution at the level of the metric.
Moreover, the energies for the spatial metric and second fundamental form
defined in \cite{lAvM2011}*{Section~7}
are reminiscent of the metric energies that we use
in the present article (see \eqref{E:MODELMETRICENERGY} and \eqref{E:NONLINEARMETRICENERGY}). 
However, the energy identities of \cite{lAvM2011}*{Section~7}
do not involve subtle cancellations 
of the type that we observe in deriving 
the approximate monotonicity identities
of Props.\ \ref{P:ENERGYESTIMATELAPSEANDSCALARFIELD} and \ref{P:PARABOLICENERGYESTIMATELAPSEANDSCALARFIELD}.
A related fact is that in \cite{lAvM2011}, 
Andersson and Moncrief were able to close their proof 
by bounding the lapse in terms
of the second fundamental form via standard elliptic estimates.
In contrast, to control the lapse,
we rely on the approximate monotonicity identities
and the AVTD-type estimates described in Step (4) of Subsect.\ \ref{SS:FIVESTEPS}.
Another notable difference is that unlike our work here,
the results of \cite{lAvM2011} are based on spatial harmonic coordinates; 
see Remark~\ref{R:NONEEDFORSPATIALHARMONIC} for additional comments
about those coordinates.

\section{Notation and conventions} \label{S:NOTATION}
In this section, we summarize some notation and conventions that we use throughout the article.

\subsection{Indices} \label{SS:INDICES}
Greek ``spacetime'' indices $\alpha, \beta, \cdots$ take on the values $0,1,2,3$, while Latin ``spatial'' indices $a,b,\cdots$ 
take on the values $1,2,3$. Repeated indices are summed over (from $0$ to $3$ if they are Greek, and from $1$ to $3$ if they are Latin). 
We use the same conventions for primed indices such as $a'$ as we do for their non-primed counterparts.
When working with the nonlinear equations
in CMC-transported spatial coordinates gauge or the parabolic lapse gauges,
spatial indices are lowered and raised with the Riemannian $3$-metric $g_{ij}$ and its inverse $g^{ij}$. 
When working with the linearized equations, we will always explicitly raise and lower indices 
with the background Kasner $3$-metric $\gKasner_{ij}$ and its inverse $\gKasner^{ij}$.

\subsection{Spacetime tensorfields and \texorpdfstring{$\Sigma_t$-}{constant-time hypersurface-}tangent tensorfields}
We denote spacetime tensorfields $\Tfour_{\nu_1 \cdots \nu_n}^{\ \ \ \ \ \ \mu_1 \cdots \mu_m}$ in bold font. 
In the nonlinear equations, 
we denote the $\gfour$-orthogonal projection of $\Tfour_{\nu_1 \cdots \nu_n}^{\ \ \ \ \ \ \mu_1 \cdots \mu_m}$ 
onto the constant-time hypersurfaces $\Sigma_t := \lbrace (s,x) \in \mathbb{R} \times \mathbb{T}^3 \ | \ s = t \rbrace$ 
in non-bold font: 
$T_{b_1 \cdots b_n}^{\ \ \ \ \ \ a_1 \cdots a_m}$. 
We also denote general $\Sigma_t$-tangent tensorfields in non-bold font.

\subsection{Coordinate systems and differential operators} \label{SS:COORDINATES}
We often work in a fixed standard local coordinate system $(x^1,x^2,x^3)$ on $\mathbb{T}^3$. The vectorfields 
$\partial_j := \frac{\partial}{\partial x^{j}}$ are globally well-defined even though the coordinates themselves are not. 
Hence, in a slight abuse of notation, we use $\lbrace \partial_1, \partial_2, \partial_3 \rbrace$ to denote the globally defined vectorfield frame. We denote the corresponding dual frame by $\lbrace dx^1, dx^2, dx^3 \rbrace$. 
As we described in Subsect.\ \ref{SS:INITIALVALUEANDCMC}, 
the spatial coordinates can be transported along the unit normal to $\Sigma_t$, 
thus producing a local coordinate system $(x^0,x^1,x^2,x^3)$ on manifolds-with-boundary of the form 
$(T,1] \times \mathbb{T}^3$, and we often write $t$ instead of $x^0$. 
The corresponding vectorfield frame on $(T,1] \times \mathbb{T}^3$ 
is $\lbrace \partial_0, \partial_1, \partial_2, \partial_3 \rbrace$, and the corresponding dual frame is
$\lbrace dx^0, dx^1, dx^2, dx^3 \rbrace$.  Relative to this frame, the Kasner metrics $\mathring{\gfour}$ are of the form 
\eqref{E:KASNER}. The symbol $\partial_{\mu}$ denotes the frame derivative $\frac{\partial}{\partial x^{\mu}}$, and we often write $\partial_t$ instead of $\partial_0$ and $dt$ instead of $dx^0$. Most of our 
equations and estimates are stated relative to the frame 
$\big\lbrace \partial_{\mu} \big\rbrace_{\mu = 0,1,2,3}$ and dual frame 
$\big\lbrace dx^{\mu} \big\rbrace_{\mu = 0,1,2,3}$.

We use the notation $\partial f$ to denote the \emph{spatial coordinate} gradient of the function $f$. Similarly,
if $\dlap$ is a $\Sigma_t$-tangent one-form, 
then $\partial \dlap$ denotes the $\Sigma_t$-tangent type $\binom{0}{2}$
tensorfield with components $\partial_i \dlap_j$ relative to the frame described above.

If $\vec{I} = (n_1,n_2,n_3)$ is a triple of non-negative integers, then we define the \emph{spatial} multi-indexed 
differential operator $\partial_{\vec{I}}$ by 
$\partial_{\vec{I}} := \partial_1^{n_1} \partial_2^{n_2} \partial_3^{n_3}$. The notation 
$|\vec{I}| := n_1 + n_2 + n_3$ denotes the order of $\vec{I}$. 

Throughout, $\Dfour$ denotes the Levi--Civita connection of $\gfour$.  We write 
\begin{align} \label{E:SPACETIMECOV}
	\Dfour_{\nu} \Tfour_{\nu_1 \cdots \nu_n}^{\ \ \ \ \ \ \mu_1 \cdots \mu_m} = 
		\partial_{\nu} \Tfour_{\nu_1 \cdots \nu_n}^{\ \ \ \ \ \ \mu_1 \cdots \mu_m} + 
		\sum_{r=1}^m \Chfour_{\nu \ \alpha}^{\ \mu_r} \Tfour_{\nu_1 \cdots \nu_n}^{\ \ \ \ \ \ \mu_1 \cdots \mu_{r-1} \alpha \mu_{r+1} 
			\cdots \mu_m} 
		- 
		\sum_{r=1}^n \Chfour_{\nu \ \nu_{r}}^{\ \alpha} 
		\Tfour_{\nu_1 \cdots \nu_{r-1} \alpha \nu_{r+1} \cdots \nu_n}^{\ \ \ \ \ \ \ \ \ \ \ \ \ \ \ \ \ \ \mu_1 \cdots \mu_m}
\end{align} 
to denote a component of the covariant derivative of a tensorfield $\Tfour$ 
(with components $\Tfour_{\nu_1 \cdots \nu_n}^{\ \ \ \ \ \ \mu_1 \cdots \mu_m}$) defined on
$(T,1] \times \mathbb{T}^3$. 
The Christoffel symbols of $\gfour$, which we denote by $\Chfour_{\mu \ \nu}^{\ \lambda}$, are defined by
\begin{align}
\Chfour_{\mu \ \nu}^{\ \lambda} & := 
		\frac{1}{2} (\gfour^{-1})^{\lambda \sigma} 
		\left\lbrace
			\partial_{\mu} \gfour_{\sigma \nu} 
			+ \partial_{\nu} \gfour_{\mu \sigma} 
			- \partial_{\sigma} \gfour_{\mu \nu} 
		\right\rbrace.
			\label{E:FOURCHRISTOFFEL}
\end{align}

We use similar notation to denote the covariant derivative of a $\Sigma_t$-tangent tensorfield $T$ 
(with components $T_{b_1 \cdots b_n}^{\ \ \ \ \ a_1 \cdots a_m}$) with respect to the Levi--Civita connection $\nabla$ 
of the Riemannian metric $g$. The Christoffel symbols of $g$, which we denote by $\Gamma_{j \ k}^{\ i}$,
are defined by 
\begin{align} \label{E:THREECHRISTOFFEL}
	\Gamma_{j \ k}^{\  i} 
	:= \frac{1}{2} g^{ia} 	
			\left\lbrace
				\partial_j g_{ak} 
				+ \partial_k g_{ja} 
				- \partial_a g_{jk}
			\right\rbrace.
\end{align}

\subsection{Integrals and \texorpdfstring{$L^2$}{square integral} norms} \label{SS:L2NORMS}
Throughout this subsection, $f$ denotes a scalar function defined on the hypersurface 
$\Sigma_t = \lbrace (s,x) \in \mathbb{R} \times \mathbb{T}^3 \ | \ s = t \rbrace$.
We define
\begin{align} \label{E:SIGMATINTEGRALDEF}
	\int_{\Sigma_t}
		f
	\, dx
	:= \int_{\mathbb{T}^3} 
				f(t,x^1,x^2,x^3) 
			\, dx.
\end{align}
Above, the notation $``\int_{\mathbb{T}^3} f \, dx"$ denotes the integral of $f$ over $\mathbb{T}^3$ with respect to the measure corresponding to the volume form of the \emph{standard Euclidean metric $\Euc$} on $\mathbb{T}^3$, which has the components
$\Euc_{ij} = \mbox{\upshape diag}(1,1,1)$ relative to the coordinate frame described 
in Subsect.\ \ref{SS:COORDINATES}.
All of our Sobolev norms are built out of the (spatial) $L^2$ norms of scalar quantities 
(which may be the components of a tensorfield). 
We define the standard $L^2$ norm $\| \cdot \|_{L^2}$ over $\Sigma_t$ as follows:
\begin{align} \label{E:SOBOLEVNORMDF}
	\left\| f \right\|_{L^2}
	 = \left\| f \right\|_{L^2}(t)
	 := 
	\left(
	\int_{\Sigma_t}
		f^2
	\, dx
	\right)^{1/2}.
\end{align}
For integers $N \geq 0$, we define the standard $H^N$ norm $\| \cdot \|_{H^N}$ over $\Sigma_t$ as follows:
\begin{align} \label{E:HIGHERSOBOLEVNORMDF}
	\left\| f \right\|_{H^N}
	 = \left\| f \right\|_{H^N}(t)
	 := 
		\left(
			\sum_{|\vec{I}| \leq N}
			\left\| \partial_{\vec{I}} f \right\|_{L^2}^2(t)
		\right)^{1/2}.
\end{align}

\subsection{Constants} \label{SS:RUNNINGCONSTANTS}
We use $C$ and $c$ to denote positive numerical constants that are free to vary from line to line.  
If $A$ and $B$ are two quantities, then we often write 
\begin{align}
	A \lesssim B
\end{align}
to indicate that ``there exists a constant $C > 0$ such that $A \leq C B$.'' 
We write $A = \mathcal{O}(B)$ to indicate that $|A| \leq C B$.
Some of the constants $C$ and $c$ in our estimates are allowed to depend on the parameter $N$ which, 
roughly speaking, represents the number of times that the equations have been differentiated
with spatial derivatives.

\section{The Einstein-scalar field equations in CMC-transported spatial coordinates and the linearized equations}
\label{S:CMCFORMULATIONOFEINSTEIN}
In this section, we provide a standard formulation of the Einstein-scalar field equations
relative to CMC-transported spatial coordinates. We then linearize the equations around
a Kasner solution \eqref{E:KASNER}.

\subsection{Preliminary discussion}
We begin by stating some basic facts concerning
the formulation of the equations.
The fundamental unknowns are $g,\SecondFund,n$, and $\phi$,
where $g$ and $n$ are as in \eqref{E:GFOURCMCTRANSPORTED}, and $\SecondFund$ is the second fundamental form of
the hypersurfaces $\Sigma_t$. More precisely, the $\Sigma_t$-tangent type $\binom{0}{2}$ tensorfield 
$\SecondFund$
is defined by requiring that following relation 
holds for all vectorfields $X,Y$ tangent to $\Sigma_t$:
\begin{align} \label{E:SECONDFUNDDEF}
	\gfour(\Dfour_X \Nml, Y) = - \SecondFund(X,Y),
\end{align}
where $\Dfour$ is the Levi--Civita connection of $\gfour$ and
\begin{align}
	\Nml 
	& := n^{-1} \partial_t
\end{align}
is the future-directed normal to $\Sigma_t$.
It is a standard fact that $\SecondFund$ is symmetric:
\begin{align}
	 \SecondFund(X,Y) =  \SecondFund(Y,X).
\end{align}
Let $\nabla$ denote the Levi--Civita connection of $g$.
The action of the Levi--Civita connection $\Dfour$ of $\gfour$ can be decomposed into
the action of $\nabla$ and $\SecondFund$ as follows:
\begin{align} \label{E:DDECOMP}
	\Dfour_X Y = \nabla_X Y - \SecondFund(X,Y)\Nml.
\end{align}

\begin{remark}[\textbf{The mixed form of $\SecondFund$ verifies equations with favorable structure and the meaning of} $\partial_{\alpha} \SecondFund_{\ j}^i$] 
\label{R:ALWAYSMIXED}
	When working with the components of $\SecondFund$, \emph{we will always write it in the mixed form 
	$\SecondFund_{\ j}^i := g^{ia} \SecondFund_{aj}$ 
	with the first index upstairs and the second one downstairs.} 
	The reason is that the nonlinear evolution and constraint equations verified by the components 
	$\SecondFund_{\ j}^i$ have a more favorable structure than 
	the corresponding equations verified by $\SecondFund_{ij}$.
	For this reason, throughout the article, we use the notation 
	$\partial_{\alpha} \SecondFund_{\ j}^i := \partial_{\alpha} (\SecondFund_{\ j}^i)$.
\end{remark}

\subsection{The Einstein-scalar field equations in CMC-transported spatial coordinates}
In the following proposition, we formulate the Einstein-scalar field
equations \eqref{E:EINSTEINSF}-\eqref{E:WAVEMODEL}
relative to CMC-transported spatial coordinates.

\begin{proposition}[\textbf{The Einstein-scalar field equations in CMC-transported spatial coordinates}]
\label{P:EINSTEINSFCMC}
In CMC-transported spatial coordinates normalized by 
\begin{align}\label{E:CMCNORMALIZATIONCHOICE}
	\SecondFund_{	\ a}^a(t,x) = - t^{-1},
\end{align}
the Einstein-scalar field system comprises the following equations.

The \textbf{Hamiltonian and momentum constraint equations} are respectively:
\begin{subequations}
\begin{align}
		R - \SecondFund_{\ b}^a \SecondFund_{\ a}^b + \underbrace{(\SecondFund_{\ a}^a)^2}_{t^{-2}} 
		& = \overbrace{(n^{-1} \partial_t \phi)^2 + g^{ab} \nabla_a \phi \nabla_b \phi}^{2 \Tfour(\Nml,\Nml)}, \label{E:HAMILTONIAN} \\
		\nabla_a \SecondFund_{\ i}^a - \underbrace{\nabla_i \SecondFund_{\ a}^a}_0 & = 
		\underbrace{- n^{-1} \partial_t \phi \nabla_i \phi}_{- \Tfour(\Nml,\partial_i)}, \label{E:MOMENTUM}
\end{align}
\end{subequations}
where $R$ denotes the scalar curvature of $g_{ij}$.

The \textbf{metric evolution equations} are:
\begin{subequations}
\begin{align}
	\partial_t g_{ij} & = - 2 n g_{ia}\SecondFund_{\ j}^a, \label{E:PARTIALTGCMC} \\
	\partial_t \SecondFund_{\ j}^i & = - g^{ia} \nabla_a \nabla_j n
		+ n \Big\lbrace \Ric_{\ j}^i + \underbrace{\SecondFund_{\ a}^a}_{-t^{-1}} \SecondFund_{\ j}^i 
			\underbrace{- g^{ia} \nabla_a \phi \nabla_j \phi}_{- T_{\ j}^i + (1/2)\ID_{\ j}^i \Tfour} 
			\Big\rbrace,  \label{E:PARTIALTKCMC}
\end{align}
\end{subequations}
where $\Ric_{\ j}^i$ denotes the Ricci curvature of $g_{ij}$
(see \eqref{E:LITTLEGRICCIONEUP}),
$\ID_{\ j}^i = \mbox{\upshape diag}(1,1,1)$ denotes the identity transformation,
and $\Tfour := (\gfour^{-1})^{\alpha \beta} \Tfour_{\alpha \beta}$ denotes the trace of the energy-momentum tensor \eqref{E:EMTSCALARFIELD}.

The \textbf{volume form factor} $\sqrt{\mbox{\upshape det} g}$
verifies the auxiliary equation\footnote{This equation, which we do not use in the present article, 
is implied by \eqref{E:PARTIALTGCMC} and 
the CMC condition $\SecondFund_{\ a}^a = - t^{-1}$.}
\begin{align} \label{E:VOLUMEFORMEVOLUTIONEQUATION}
	\partial_t \ln \left(t^{-1} \sqrt{\mbox{\upshape det} g} \right)
	& = \frac{n - 1}{t}.
\end{align}

The \textbf{scalar field wave equation} is:
\begin{align} \label{E:SCALARFIELDCMC}
	\overbrace{- n^{-1} \partial_t(n^{-1} \partial_t \phi)}^{- \Dfour_{\Nml} \Dfour_{\Nml} \phi} + g^{ab} \nabla_a \nabla_b \phi 
		& = \overbrace{\frac{1}{t} n^{-1} \partial_t \phi}^{- \SecondFund_{	\ a}^a \Dfour_{\Nml} \phi} 
		- n^{-1} g^{ab} \nabla_a n \nabla_b \phi.	
\end{align}

The \textbf{elliptic lapse equation}\footnote{Below, when we linearize the equations, we will view $n-1$ as a linearly small quantity.
Hence, we prefer to write \eqref{E:LAPSE} as an equation in $n-1$.} is:
\begin{align} \label{E:LAPSE}
	g^{ab} \nabla_a \nabla_b (n - 1) 
		& = (n - 1) \Big\lbrace R + \underbrace{(\SecondFund_{\ a}^a)^2}_{t^{-2}} - g^{ab} \nabla_a \phi \nabla_b \phi
		\Big\rbrace \\
	& \ \ + R - g^{ab}\nabla_a \phi \nabla_b \phi
			+ \underbrace{(\SecondFund_{\ a}^a)^2 - \partial_t (\SecondFund_{\ a}^a)}_{0}. \notag
\end{align}
The gauge condition \eqref{E:CMCNORMALIZATIONCHOICE} 
and the constraint equations \eqref{E:HAMILTONIAN}-\eqref{E:MOMENTUM}
are preserved by the flow of the remaining equations if they are verified by the data.

\end{proposition}

\begin{proof}[\textbf{Proof of Prop.\ \ref{P:EINSTEINSFCMC}}]
	It is well-known that the constraint equations \eqref{E:HAMILTONIAN}-\eqref{E:MOMENTUM} 
	follow from \eqref{E:EINSTEINSF}; see, for example, \cite{rW1984}*{Chapter~10}, and
	note that our $\SecondFund$ has the opposite sign convention of the one in \cite{rW1984}.	
	It is also well-known that equations \eqref{E:PARTIALTGCMC}-\eqref{E:SCALARFIELDCMC}
	follow from \eqref{E:EINSTEINSF}-\eqref{E:WAVEMODEL};
	see, for example, \cite{aS2010}*{Section~6.2} or \cite{mT1997III}*{Section~10 of Chapter~18}.
	To derive \eqref{E:LAPSE}, we take the trace of \eqref{E:PARTIALTKCMC} and use the CMC condition $\SecondFund_{\ a}^a = - t^{-1}$.
	The preservation of the gauge condition and constraints is a standard result that can be derived
	from a straightforward modification of the argument presented in \cite{lAvM2003}*{Theorem~4.2}.
\end{proof}

\subsection{The linearization procedure and the linearly small quantities}
\label{SS:LINEARIZEDQUANTITIES}
In our linear analysis, 
we work with the ``linearly small quantities'' defined just below in 
Def.\ \ref{D:LINEARIZEDVARIABLES}.
In the definition, $g$ denotes the 
(Riemannian) $3$-metric from Prop.\ \ref{P:EINSTEINSFCMC},
$\SecondFund_{\ j}^i$ denotes its mixed second fundamental form,
$\gKasner$ denotes the $3$-metric of the Kasner solution (see \eqref{E:KASNER}),
$\SecondFundKasner_{\ j}^i$ denotes its mixed second fundamental form
(see \eqref{E:KASNERSECONDFUNDONEUP}),
and similarly for the other quantities.
Before stating the definition of the linearly small quantities,
we first make some remarks about how one
can linearize the equations of Prop.\ \ref{P:EINSTEINSFCMC}
around a given solution. There are two ways that this can 
be achieved. Both approaches lead to the same 
system of linear PDEs but conceptually
are somewhat different.
The first way, which is manifestly invariant, 
is through the notion of
one-parameter family of
solutions to the equations,
similar to our discussion below Theorem~\ref{T:ROUGHVERSIONLINEARSTABILITY}. 
That is, 
one can consider an $\upalpha$-parameterized
family of solutions
$(n[\alpha],g[\alpha],\SecondFund[\alpha],\phi[\alpha])$
to the nonlinear equations of Prop.\ \ref{P:EINSTEINSFCMC}
such that
$(n[0],g[0],\SecondFund[0],\phi[0])$
is the background solution around which one would like to linearize.
We set $n'[\alpha] := \frac{d}{d \upalpha}n[\alpha]$ and similarly
for the other variables.
One can then differentiate the nonlinear equations
with respect to $\upalpha$ and set $\upalpha = 0$
to deduce that the variations 
$(n'[0],g'[0],\SecondFund'[0],\phi'[0])$
solve a system of linear PDEs whose coefficients
depend on $(n[0],g[0],\SecondFund[0],\phi[0])$.
The system thus obtained is the linearization of the 
Einstein-scalar field equations in CMC-transported spatial coordinates gauge
about the background solution $(n[0],g[0],\SecondFund[0],\phi[0])$.

The second way to derive the linearized system
is to perform a first-order 
Taylor expansion of the 
nonlinear equations
of Prop.\ \ref{P:EINSTEINSFCMC} 
about a given solution, in our case
a Kasner solution 
$(1,\gKasner,\SecondFundKasner,\sfKasner)$,
where $1$ is the Kasner lapse.
Equivalently, in the nonlinear equations,
one decomposes the 
nonlinear spatial metric $g_{ij}$ as
$g_{ij} = \gKasner_{ij} + \grenormalized_{ij}$
(where $\grenormalized_{ij}$ is the ``linearly small'' metric perturbation)
and similarly for the other solution variables,
and then discards all terms that are quadratic or smaller
in the perturbation variables 
(where the derivatives of the perturbation variables are also considered to be linearly small).
After one discards the quadratic-or-higher-order small terms
and accounts for the fact that the background Kasner solution 
is a solution to the nonlinear equations,
what remains is a system of linear PDEs 
whose coefficients depend on the Kasner solution.
This is the approach that we take in the proof of Prop.\ \ref{P:LINEARIZEDCMCEQUATIONS}.
Though seemingly less invariant than the first approach,
it is straightforward to see that it yields the same linear PDE system.

Having made these remarks, we now define the 
linearly small ``perturbation variables''
that play a role in our derivation of the linearized equations.

\begin{definition}[\textbf{Linearly small quantities}]
\label{D:LINEARIZEDVARIABLES}
We define 
(for $a,b,i,j = 1,2,3$)
\begin{subequations}
\begin{align}
	\grenormalized_{ij} 
	& := g_{ij} - \gKasner_{ij},
		\\
	\christrenormalizedarg{a}{i}{b}
	& := \frac{1}{2}
		\gKasner^{ic}
		\left\lbrace
			\partial_a \grenormalized_{cb} 				
			+ \partial_b \grenormalized_{ac}
			- \partial_c \grenormalized_{ab}
		\right\rbrace,
			\\
	\currenormalized
	& := - \frac{1}{2} 
				\gKasner^{ab} \gKasner^{ef} \partial_e \partial_f \grenormalized_{ab}
				+ 
			\gKasner^{ef} \partial_a \christrenormalizedarg{e}{a}{f},	
		\label{E:SCALARCURVATURERENORMALIZED} \\
	\Ricrenormalizedarg{i}{j}
	& := - \frac{1}{2} 
				\gKasner^{ia} \gKasner^{ef} \partial_e \partial_f \grenormalized_{ja}
			+ \frac{1}{2} 
				\gKasner^{ef} \partial_j \christrenormalizedarg{e}{i}{f}
			+ \frac{1}{2} 
				\gKasner^{ia}
				\gKasner_{jb}
				\gKasner^{ef} \partial_a \christrenormalizedarg{e}{b}{f},	
		\label{E:RICRENORMALIZED} \\
	\LinSecondFund_{\ j}^i
	& := \SecondFund_{\ j}^i 
			- 
			\SecondFundKasner_{\ j}^i,	
		\label{E:SECONDFUNDRENORMALZIZED} \\
	\SFRenormalized
		& := \phi - \sfKasner,
	 		\label{E:LINEARSCALARFIELD} \\
	 \LapseRenormalized
	 & := n - 1.
\end{align}
\end{subequations}
\end{definition}

\begin{remark}[\textbf{Justification of Def.\ \ref{D:LINEARIZEDVARIABLES}}]
	The main point is that for solutions
	to the nonlinear equations that are near the Kasner solution \eqref{E:KASNER},
	all of the quantities defined in Def.\ \ref{D:LINEARIZEDVARIABLES}
	are linearly small in the sense described above Def.\ \ref{D:LINEARIZEDVARIABLES}.
\end{remark}

\begin{remark} \label{R:TRACEFREEPART}
	Below and throughout, $\hat{T}$ denotes the trace-free part of the $\Sigma_t$-tangent tenor $T$.
\end{remark}

\begin{remark}
	Note that $\LinSecondFund$ is trace-free, that is,
	\begin{align} \label{E:LINEARIZEDSECONDFUNDFORMISTRACEFREE}
		\LinSecondFund 
		& = \tracefreeLinSecondFund.
	\end{align}
	\eqref{E:LINEARIZEDSECONDFUNDFORMISTRACEFREE} 
	follows from definition \eqref{E:SECONDFUNDRENORMALZIZED}, 
	the CMC condition $\SecondFund_{\ a}^a(t,x) = - t^{-1}$,
	and the fact that $\SecondFundKasner_{\ a}^a(t,x) = - t^{-1}$.
\end{remark}

\subsection{The linearized Einstein-scalar field equations in CMC-transported spatial coordinates}
\label{SS:LINEARIZEDEQNSCMC}
In the next proposition, we use the procedure described just above
Def.\ \ref{D:LINEARIZEDVARIABLES} to linearize the equations of Prop.\ \ref{P:EINSTEINSFCMC}
around a given Kasner solution \eqref{E:KASNER}.

\begin{proposition}[\textbf{The linearized Einstein-scalar field equations in CMC-transported spatial coordinates}]
\label{P:LINEARIZEDCMCEQUATIONS}
Consider the equations of Prop.\ \ref{P:EINSTEINSFCMC}
linearized around a Kasner solution \eqref{E:KASNER}.
The linearized equations in the unknowns 
$(\LapseRenormalized,\grenormalized,\LinSecondFund,\SFRenormalized)$,
which are functions of $(t,x) \in (0,\infty) \times \mathbb{T}^3$,
take the following form (see Def.\ \ref{D:LINEARIZEDVARIABLES} for the definitions of some of the quantities).

The \textbf{linearized constant mean curvature condition} is:
\begin{align} \label{E:CMCMODEL}
	\LinSecondFund_{\ a}^a & = 0.
\end{align}

The \textbf{linearized versions of the Hamiltonian and momentum constraint equations} 
\eqref{E:HAMILTONIAN}-\eqref{E:MOMENTUM} are:
\begin{subequations}
\begin{align}
	  t^2 \currenormalized
		- 2 (t \tracefreeSecondFundKasner_{\ b}^a) (t \LinSecondFund_{\ a}^b)
		- 2 A t \partial_t \SFRenormalized
		+ 2 A^2 \LapseRenormalized 
		& = 0,
		\label{E:LINEARIZEDLHAMILTONIAN} \\
	\partial_a (t \LinSecondFund_{\ i}^a)
		& = 
			- A \partial_i \SFRenormalized
			- \christrenormalizedarg{a}{a}{b} (t \tracefreeSecondFundKasner_{\ i}^b)
			+ \christrenormalizedarg{a}{b}{i} (t \tracefreeSecondFundKasner_{\ b}^a),
		\label{E:LINEARIZEDMOMENTUM}
		\\
		\gKasner^{ab} \partial_a (t \LinSecondFund_{\ b}^i)
		& = - A \gKasner^{ia} \partial_a \SFRenormalized
			- \gKasner^{ab} \christrenormalizedarg{a}{i}{c} (t \tracefreeSecondFundKasner_{\ b}^c)
			+ \gKasner^{ab} \christrenormalizedarg{a}{c}{b} (t \tracefreeSecondFundKasner_{\ c}^i),
		\label{E:LINEARIZEDSECONDMOMENTUM}
\end{align}
\end{subequations}
where the constant $0 \leq A \leq \sqrt{2/3}$ is defined by \eqref{E:KASNERHAMILTONIANCONSTRAINT}.

The \textbf{linearized version of the lapse equation} \eqref{E:LAPSE} can be expressed in either of the following two forms:
\begin{subequations}
\begin{align} 
	2 A (t \partial_t \SFRenormalized)
	+ 2 (t \tracefreeSecondFundKasner_{\ b}^a) (t \LinSecondFund_{\ a}^b)
	& = t^2 \gKasner^{ab} \partial_a \partial_b \LapseRenormalized
		+ (2 A^2 - 1) \LapseRenormalized, 
	\label{E:LINEARIZEDLAPSE} 
		\\
	t^2 \gKasner^{ab} \partial_a \partial_b \LapseRenormalized 
	- \LapseRenormalized
	& = t^2 \currenormalized.
		\label{E:LINEARIZEDLAPSELOWER} 
\end{align}
\end{subequations}
Equation \eqref{E:LINEARIZEDLHAMILTONIAN} can be used to show that \eqref{E:LINEARIZEDLAPSE} 
is equivalent to \eqref{E:LINEARIZEDLAPSELOWER}.

The \textbf{linearized versions of the metric evolution equations} 
\eqref{E:PARTIALTGCMC}-\eqref{E:PARTIALTKCMC} are:
\begin{subequations}
\begin{align}
	\partial_t \grenormalized_{ij} 
		& = -2 t^{-1} (t \SecondFundKasner_{\ j}^a) \grenormalized_{ia} 
			- 2 t^{-1} \gKasner_{ia} (t \LinSecondFund_{\ j}^a)
			- 2 t^{-1} \gKasner_{ia} (t \SecondFundKasner_{\ j}^a) \LapseRenormalized, 
		\label{E:LINEARIZEDGEVOLUTION} \\
	\partial_t (t \LinSecondFund_{\ j}^i)
		& = - t \gKasner^{ia} \partial_a \partial_j \LapseRenormalized
		- t^{-1} (t \SecondFundKasner_{\ j}^i) \LapseRenormalized 
		+ t \Ricrenormalizedarg{i}{j}.
		\label{E:LINEARIZEDKEVOLUTION}
\end{align}
\end{subequations}

The \textbf{linearized version of the scalar field wave equation} \eqref{E:SCALARFIELDCMC} is:
\begin{align} \label{E:SCALARFIELDWAVEDECOMPOSED}
	- \partial_t (t \partial_t \SFRenormalized)
	+ t \gKasner^{ab} \partial_a \partial_b \SFRenormalized 
	& = - A \partial_t \LapseRenormalized
			+ 
			A t^{-1} \LapseRenormalized.
\end{align}

\end{proposition}

\begin{remark}[\textbf{An alternate approach}]
	\label{R:ALTERNATEAPPROACH}
	One could adopt an alternate approach to the proof of our stability results
	in which the product $n^{-1} \partial_t \phi$ is treated
	as an independent quantity.
	In such an approach, 
	one would not generate terms in the equations
	that depend on the time derivative of the lapse. This
	would simplify some aspects of the analysis.
	For example, upon linearizing the equations under the alternate approach,
	one would not generate the term $\partial_t \LapseRenormalized$,
	which appears on the right-hand side
	\eqref{E:SCALARFIELDWAVEDECOMPOSED}.
	The alternate approach would not have any substantial effect
	on our main results. 
	For example, notice that $\partial_t \LapseRenormalized$
	does not appear in the approximate monotonicity identity stated in
	Theorem~\ref{T:CMCMONOTONICITYID}
	(though, under the approach of this paper, 
	$\partial_t \LapseRenormalized$ does play a role in its \emph{proof}).
	The alternate approach is closer in spirit to the approach
	that we take in \cite{iRjS2014b} in our study of 
	the Einstein-stiff fluid system, in which we avoid
	having to treat the time derivative of the lapse in the evolution equations.
\end{remark}

\begin{remark}
	Equation \eqref{E:LINEARIZEDMOMENTUM} is the linearized version of the constraint
	$\nabla_a \SecondFund_{\ i}^a = - n^{-1} \partial_t \phi \nabla_i \phi$,
	while equation \eqref{E:LINEARIZEDSECONDMOMENTUM} is the linearized version of 
	$\nabla^a \SecondFund_{\ a}^i = - n^{-1} \partial_t \phi g^{ia} \nabla_a \phi$.
	We use both of these equations when deriving estimates.
\end{remark}

\begin{remark}[\textbf{Propagation of $L^2$ regularity}]
	In deriving the equations of Prop.\ \ref{P:LINEARIZEDCMCEQUATIONS},
	we have linearized a version of the Einstein-scalar field system written relative to
	a dynamic system of coordinates that is adapted to the nonlinear flow.
	It is for this reason that our 
	approximate monotonicity identity 
	for linear solutions, 
	which we derive below in Prop.\ \ref{P:ENERGYESTIMATELAPSEANDSCALARFIELD},
	should be viewed as providing relevant information about the
	$L^2$ regularity of the nonlinear solution.
	In particular, the proof of Prop.\ \ref{P:ENERGYESTIMATELAPSEANDSCALARFIELD} 
	can be modified in a straightforward fashion 
	to yield a coercive integral identity for the nonlinear equations,
	consistent with well-posedness relative to the CMC-transported spatial coordinates gauge.
\end{remark}

\begin{proof}[Proof of Prop.\ \ref{P:LINEARIZEDCMCEQUATIONS}]
	We first note that \eqref{E:CMCMODEL} follows from \eqref{E:LINEARIZEDSECONDFUNDFORMISTRACEFREE}.
	
	We will derive three more equations in detail. The remaining
	equations can be derived using similar arguments and we omit those details.
	The overall strategy is to consider the equations of Prop.\ \ref{P:EINSTEINSFCMC}
	and to expand the Riemannian metric $g$ 
	as an order $0$ ``Kasner term'' and a perturbation term as follows: 
	$g_{ij} = \gKasner_{ij} + \grenormalized_{ij}$, and similarly for
	$(t \SecondFund_{\ j}^i,\phi,n)$. We then discard all terms that are quadratic 
	or higher-order in the perturbations, which yields the proposition. 
	Since this proof features the spatial metrics $g$ and $\gKasner$,
	to avoid confusion,
	we will denote the components of the inverse Kasner spatial metric
	by $(\ginverseKasner)^{ij}$ rather than $\gKasner^{ij}$.
	
	As our first detailed example, we derive \eqref{E:SCALARFIELDWAVEDECOMPOSED}.
	We start by expanding the scalar field wave equation
	\eqref{E:SCALARFIELDCMC} as follows:
	\begin{align} \label{E:EXPANDEDSCALARFIELDCMC}
		- \partial_t(t \partial_t \phi)
		+ 
		n^2 t g^{ab} \nabla_a \nabla_b \phi 
		& = 
			\frac{(n-1)}{t} t \partial_t \phi
			-
			\frac{(\partial_t n)}{n} t \partial_t \phi
			- 
			n t g^{ab} \nabla_a n \nabla_b \phi.
	\end{align}
	Using \eqref{E:EXPANDEDSCALARFIELDCMC}, we compute that
	\begin{align} \label{E:SECONDEXPANDEDSCALARFIELDCMC}
		&
		- 
		\partial_t(t \partial_t \phi - A)
		+
		t g^{ab} \partial_a \partial_b \phi 
		+ 
		t (n+1)(n-1) g^{ab} \nabla_a \nabla_b \phi 
			\\
		& = 
			A \frac{(n-1)}{t} 
			-
			A \partial_t n 
				\notag \\
		& \ \
			+
			n^2 t g^{ab} \Gamma_{a \ b}^{\ j} \partial_j \phi 
			+
			\frac{(n-1)}{t} (t \partial_t \phi - A)
				\notag \\
		& \ \
			-
			(\partial_t n) (t \partial_t \phi - A)
			+
			(\partial_t n) \frac{n-1}{n} t \partial_t \phi
			- 
			t g^{ab} \partial_a n \partial_b \phi
			- 
			(n-1) t g^{ab} \partial_a n \partial_b \phi.
			\notag
	\end{align}
	We now discard the quadratically small terms,
	that is, the term $t (n+1)(n-1) g^{ab} \nabla_a \nabla_b \phi$
	and the terms on the last two lines of
	\eqref{E:SECONDEXPANDEDSCALARFIELDCMC},
	which, in view of Def.\ \ref{D:LINEARIZEDVARIABLES}, 
	yields \eqref{E:SCALARFIELDWAVEDECOMPOSED}.
	
	Next, we derive equation \eqref{E:LINEARIZEDKEVOLUTION}.
	To this end, we expand the evolution equation \eqref{E:PARTIALTKCMC}
	for $\SecondFund_{\ j}^i$ as follows:
	\begin{align} \label{E:EXPANDEDPARTIALTKCMC}
	\partial_t (t \SecondFund_{\ j}^i)
		& 
		= 
		- t g^{ia} \partial_a \partial_j n
		+ t g^{ia} \Gamma_{a \ j}^{\ b} \partial_b n
		- \frac{n-1}{t} (t \SecondFund_{\ j}^i)
		+ t \Ric_{\ j}^i 
		+ t (n-1) \Ric_{\ j}^i 
		- t n g^{ia} \partial_a \phi \partial_j \phi.
	\end{align}
	From \eqref{E:EXPANDEDPARTIALTKCMC}, we compute that
	\begin{align} \label{E:SECONDEXPANDEDPARTIALTKCMC}
	\partial_t 
		\left\lbrace
			t \SecondFund_{\ j}^i - t \SecondFundKasner_{\ j}^i
		\right\rbrace
		& 
		= 
		- 
		t (\ginverseKasner)^{ia} \partial_a \partial_j n
		- 
		\frac{n-1}{t} (t \SecondFundKasner_{\ j}^i)
		+ 
		t \Ric_{\ j}^i
			\\
	& \ \
		- 
		t \left\lbrace
				g^{ia} - (\ginverseKasner)^{ia}
			\right\rbrace
			\partial_a \partial_j n
			+ 
			t g^{ia} \Gamma_{a \ j}^{\ b} \partial_b n
			\notag \\
	& \ \
		- 
		\frac{n-1}{t} (t \SecondFund_{\ j}^i - t \SecondFundKasner_{\ j}^i)
		+
		t (n-1) \Ric_{\ j}^i 
		- 
		t n g^{ia} \partial_a \phi \partial_j \phi.
		\notag
	\end{align}
	Next, we note
	that it is straightforward to see that in Def.\ \ref{D:LINEARIZEDVARIABLES},
	$\christrenormalizedarg{a}{i}{b}$ is
	the linearization of the Christoffel symbol $\Gamma_{a \ b}^{\ i}$
	(see \eqref{E:THREECHRISTOFFEL}) around the Kasner solution,
	and similarly for
	$\Ricrenormalizedarg{i}{j}$
	and
	$\currenormalized$.
	We have obtained the latter two linearizations from the standard expression 
	\begin{align} \label{E:LITTLEGRICCIONEUP}
	\Ric_{\ j}^i =
		g^{ic} \partial_a \Gamma_{c \ j}^{\ a}
		- g^{ic} \partial_c \Gamma_{j \ a}^{\ a}
  	+ g^{ic} \Gamma_{a \ b}^{\ a} \Gamma_{c \ j}^{\ b}
		- g^{ic} \Gamma_{c \ b}^{\ a} \Gamma_{a \ j}^{\ b}
	\end{align}
	for the Ricci curvature of $g$ in terms of its Christoffel symbols \eqref{E:THREECHRISTOFFEL}
	and the definition $R := \Ric_{\ a}^a$.
	From these facts and Def.\ \ref{D:LINEARIZEDVARIABLES},
	it follows that the linearly small terms in \eqref{E:SECONDEXPANDEDPARTIALTKCMC}
	are the term on the left-hand side, 
	the first two terms on the right-hand side,
	and $t \Ricrenormalizedarg{i}{j}$,
	which we obtain from linearizing the third term
	$t \Ric_{\ j}^i$ on the right-hand
	side of \eqref{E:SECONDEXPANDEDPARTIALTKCMC}.
	Discarding the remaining terms, we obtain the linearized equation
	\eqref{E:LINEARIZEDKEVOLUTION} as desired.
	
	As our final example, we derive 
	the linearized Hamiltonian constraint
	equation \eqref{E:LINEARIZEDLHAMILTONIAN}.
	We first expand equation \eqref{E:HAMILTONIAN} 
	to deduce
	\begin{align} \label{E:EXPANDEDHAMILTONIAN}
		t^2 R 
		- 
		(t \SecondFund_{\ b}^a) (t \SecondFund_{\ a}^b)
		+ 
		1
		& 
		= 
		(t \partial_t \phi)^2
		-
		2 (t \partial_t \phi)^2 (n-1) 
			\\
	& \ \
		+
		\frac{(2 n+1)}{n^2}
		(t \partial_t \phi)^2
		(n-1)^2
		+ 
		t^2 g^{ab} \partial_a \phi \partial_b \phi.
		\notag
	\end{align}
	Using \eqref{E:EXPANDEDHAMILTONIAN}, 
	the CMC condition $t \SecondFund_{\ a}^a = - 1$,
	the identity
	$(t \SecondFundKasner_{\ b}^a) (t \SecondFundKasner_{\ a}^b) 
	= \sum_{i=1}^3 q_i^2$,
	and the exponent constraints
	\eqref{E:KASNERTRACECONDITION}-\eqref{E:KASNERHAMILTONIANCONSTRAINT},
	we compute that
	\begin{align} \label{E:SECONDEXPANDEDHAMILTONIAN}
		&
		t^2 R 
		- 
		2 (t \hat{\SecondFund}_{\ b}^a - t \hat{\SecondFundKasner}_{\ b}^a) (t \SecondFundKasner_{\ a}^b)
		+ 
		\underbrace{1 - \sum_{i=1}^3 q_i^2 - A^2}_{=0}
			\\
	& 
		= 
		2 A (t \partial_t \phi - A)
		-
		2 A^2 (n-1)
		\notag \\
	& \ \
		+
		(t \partial_t \phi - A)^2
		-
		4A (n-1) (t \partial_t \phi - A) (t \partial_t \phi + A)
		-
		2 (n-1) (t \partial_t \phi - A)^2
			\notag \\
	& \ \
		+
		\frac{(2 n+1)}{n^2}
		(t \partial_t \phi)^2
		(n-1)^2
		+ 
		t^2 g^{ab} \partial_a \phi \partial_b \phi.
		\notag
	\end{align}
	With the help of Def.\ \ref{D:LINEARIZEDVARIABLES},
	we see that the linearly small terms in \eqref{E:SECONDEXPANDEDHAMILTONIAN}
	are the two terms on the first line of the right-hand side, 
	the term $2 (t \hat{\SecondFund}_{\ b}^a - t \hat{\SecondFundKasner}_{\ b}^a) (t \SecondFundKasner_{\ a}^b)$
	on the left-hand side, and the term
	$t^2 \currenormalized$
	obtained from linearizing the first term 
	$t^2 R$ on the left-hand side. 
	Discarding the remaining terms, we obtain
	the linearized equation \eqref{E:LINEARIZEDLHAMILTONIAN} as desired.
	This completes our proof of Prop.\ \ref{P:LINEARIZEDCMCEQUATIONS}.
\end{proof}

\section{Norms and energies}
\label{S:NORMSANDENERGIES}
In this short section, we define
the norms and energies that play a role in our analysis
of linear solutions.

\subsection{Pointwise norms}
We will use the following two norms.

\begin{definition}  [\textbf{Pointwise norms}] \label{D:POINTWISENORMS}
	Let $T$ be a type $\binom{m}{n}$ $\Sigma_t$-tangent tensor with components 
	$T_{b_1 \cdots b_n}^{\ \ \ \ \ a_1 \cdots a_m}$.
	Then $|T|_{Frame}$ denotes the following norm 
	(involving the \emph{components} of $T$ relative to the transported coordinate frame):
	\begin{subequations}
	\begin{align}
		|T|_{Frame}^2 :=  \sum_{a_1 = 1}^3 \cdots \sum_{a_m = 1}^3 \sum_{b_1 = 1}^3 \cdots \sum_{b_n = 1}^3  
			\left|T_{b_1 \cdots b_n}^{\ \ \ \ \ a_1 \cdots a_m} \right|^2.
	\end{align}
	
	$|T|_{\gKasner}$ denotes the $\gKasner$-norm of $T$, where $\gKasner$ is the background Kasner spatial metric from \eqref{E:KASNER}:
	\begin{align} \label{E:NEWGNORM}
		|T|_{\gKasner}^2 :=  
			\gKasner_{a_1 a_1'} \cdots \gKasner_{a_m a_m'}
			(\gKasner^{-1})^{b_1 b_1'} \cdots (\gKasner^{-1})^{b_n b_n'} 
			T_{b_1 \cdots b_n}^{\ \ \ \ \ a_1 \cdots a_m}
			T_{b_1' \cdots b_n'}^{\ \ \ \ \ a_1' \cdots a_m'}.
	\end{align}
	\end{subequations}
\end{definition}

\subsection{Sobolev and Lebesgue norms}
In our analysis, we will use
the Sobolev norms
$\| \cdot \|_{H_{Frame}^M}$ 
and the Lebesgue norm $\| \cdot \|_{L_{\gKasner}^2}$ defined below in Def.\ \ref{D:SOBOLEVNORMS}.
The norms $\| \cdot \|_{H_{Frame}^M}$ 
are ``less geometric'' than the energies of Def.\ \ref{D:ENERGIES}
because their definition involves the components of tensorfields
relative to the transported coordinate frame
rather than invariant quantities.
The norms $\| \cdot \|_{H_{Frame}^M}$ are important for the 
proof of linear stability (see Theorem~\ref{T:CMCLINEARSTABILITY}).

\begin{definition} [\textbf{Sobolev and Lebesgue norms}]  \label{D:SOBOLEVNORMS}
	Let $T$ be a type $\binom{m}{n}$
	$\Sigma_t$-tangent tensorfield with components $T_{b_1 \cdots b_n}^{\ \ \ \ \ a_1 \cdots a_m}$.
	We define
\begin{align} \label{E:SOBOLEVNORMS}
	\| T \|_{H_{Frame}^M} = \| T \|_{H_{Frame}^M}(t) & := 
		\sum_{|\vec{I}| \leq M} \left\| \left|\partial_{\vec{I}} T(t,\cdot) \right|_{Frame} \right \|_{L^2}, 
\end{align}
where $\left\| f \right\|_{L^2}$ is defined in \eqref{E:SOBOLEVNORMDF},
$\vec{I}$ denotes a spatial coordinate derivative multi-index (see Subsect.\ \ref{SS:COORDINATES}), 
and
\begin{align} \label{E:PARTIALIT}
	(\partial_{\vec{I}} T)_{b_1 \cdots b_n}^{\ \ \ \ \ a_1 \cdots a_m}
	:= \partial_{\vec{I}} (T_{b_1 \cdots b_n}^{\ \ \ \ \ a_1 \cdots a_m}).
\end{align}
We sometimes use the notation $\| T \|_{L_{Frame}^2}$ in place of $\| T \|_{H_{Frame}^0}$.

We also define the Lebesgue norm
\begin{align} \label{E:L2LEBESGUE}
	\| T \|_{L_{\gKasner}^2} = \| T \|_{L_{\gKasner}^2}(t) 
		& := 
		\left\| \left|T(t,\cdot) \right|_{\gKasner} \right \|_{L^2},
\end{align}
where $|T(t,\cdot)|_{\gKasner}$ is defined in \eqref{E:NEWGNORM}.
\end{definition}

\begin{remark}
	If $T$ is a scalar function, then we often write $|T|$ instead of 
	$|T|_{Frame}$
	or
	$|T|_{\gKasner}$,
	$\| T \|_{H^M}$ instead of $\| T \|_{H_{Frame}^M}$,
	and
	$\| T \|_{L^2}$ instead of 
	$\| T \|_{L_{Frame}^2}$
	or
	$\| T \|_{L_{\gKasner}^2}$
	since 
	for scalar functions, there is no danger
	of confusion over how to measure the size of $T$.
\end{remark}

\begin{definition}[\textbf{Solution norms}] \label{D:NORMS}
The specific norms that are most relevant for the linear solutions under study are as follows:
\begin{align}
	\highnorm{M}(t) 
	& := 
			\left\| t \LinSecondFund \right\|_{H_{Frame}^M} 
			+ \| \partial \grenormalized  \|_{H_{Frame}^M} 
			+ \left\| t \partial_t \SFRenormalized \right\|_{H_{Frame}^M} 
			+ t^{2/3} \| \partial \SFRenormalized  \|_{H_{Frame}^M} 
			+ \sum_{p=0}^2 t^{(2/3)p} \left\| \LapseRenormalized \right\|_{H^{M+p}}.
		\label{E:HIGHNORM}  
\end{align}

\end{definition}

\subsection{Energies}
Our monotonicity identities
and our energy estimates
involve the following energies for the linearized variables.

\begin{definition}[\textbf{Energies}]
\label{D:ENERGIES}
For $t \in (0,1]$, we define
$
\mathscr{E}_{(Metric)}(t) \geq 0,
$
$\cdots$,
$
\mathscr{E}_{(Total);\smallparameter}(t)
\geq 0
$
as follows:
\begin{subequations}
\begin{align}
		\mathscr{E}_{(Metric)}^2(t) 
		&:= \int_{\Sigma_t} 
					|t \LinSecondFund|_{\gKasner}^2 
					+ 
					\frac{1}{4} |t \partial \grenormalized|_{\gKasner}^2 
				\, dx,
			\label{E:MODELMETRICENERGY}
				\\
	\mathscr{E}_{(Scalar)}^2(t)
		& := \int_{\Sigma_t} 
					 (t \partial_t \SFRenormalized)^2 
					 + |t \partial \SFRenormalized|_{\gKasner}^2 
				 \, dx, 
			\label{E:MODELSCALARFIELDENERGY} \\
	\mathscr{E}_{(\partial Lapse)}^2(t)
		& := \int_{\Sigma_t} 
				|t \partial \LapseRenormalized|_{\gKasner}^2
			\, dx, 
			\label{E:LINEARIZEDPARTIALLAPSEENERGY} \\
	\mathscr{E}_{(Lapse)}^2(t)
		& := \int_{\Sigma_t} 
				\LapseRenormalized^2
			\, dx, 
			\label{E:LINEARIZEDLAPSEENERGY} 
			\\
		\mathscr{E}_{(Total);\smallparameter}^2(t) 
		& := \mathscr{E}_{(Scalar)}^2(t)
				+ 
				\mathscr{E}_{(\partial Lapse)}^2(t)
				+  
				(1 - A^2) \mathscr{E}_{(Lapse)}^2(t)
				+
				\smallparameter \mathscr{E}_{(Metric)}^2(t),
			\label{E:TOTALENERGY}
\end{align}
\end{subequations}
where the constant $0 \leq A \leq \sqrt{2/3}$ is defined by \eqref{E:KASNERHAMILTONIANCONSTRAINT}
and $\smallparameter$ is a small positive constant that we choose below when we derive
estimates for $\mathscr{E}_{(Total);\smallparameter}^2(t)$.

We will also use up-to-order $M$ energies.
Specifically, we view the energy
$\mathscr{E}_{(Total);\smallparameter}^2$
defined in \eqref{E:TOTALENERGY}
as a functional of 
$\LinSecondFund, 
	\partial \grenormalized, 
  \partial_t \SFRenormalized,
	\partial \SFRenormalized,\partial \LapseRenormalized, 
	\LapseRenormalized
$
(that is, 
$\mathscr{E}_{(Total);\smallparameter}^2 
= 
\mathscr{E}_{(Total);\smallparameter}^2
[\LinSecondFund, 
	\partial \grenormalized, 
  \partial_t \SFRenormalized,
	\partial \SFRenormalized,\partial \LapseRenormalized, 
	\LapseRenormalized]$),
and we define
\begin{align}
	\mathscr{E}_{(Total);\smallparameter;M}^2(t) 
		& :=  \sum_{|\vec{I}| \leq M}
			\mathscr{E}_{(Total);\smallparameter}^2
			[\partial_{\vec{I}} \LinSecondFund, \partial \partial_{\vec{I}} \grenormalized, 
				\partial_t \partial_{\vec{I}} \SFRenormalized,
			\partial \partial_{\vec{I}} \SFRenormalized,\partial \partial_{\vec{I}} \LapseRenormalized, 
			\partial_{\vec{I}}\LapseRenormalized](t).
			\label{E:TOPORDERENERGY}
\end{align}

\end{definition}

In Lemma~\ref{L:ENERGYNORMCOMPARISON} below, we compare the 
strength of the energies to the strength of the norms.
Its proof is straightforward and amounts to tracking powers of $t$.
We first provide the following lemma, whose simple proof we omit.

\begin{lemma}[\textbf{Basic properties of the spatial part of the Kasner metric}]
\label{L:KASNERMETRIC}
Let $\tracefreeparameter \geq 0$ be as defined in \eqref{E:TRACEFREEPARAMETER}.
The components $\gKasner_{ij}$ of the Kasner spatial metric  
(see \eqref{E:KASNER})
and the components $\gKasner^{ij}$ of its inverse verify the following estimates 
for $(t,x) \in (0,1] \times \mathbb{T}^3$,
($i,j=1,2,3$):
\begin{subequations}
\begin{align}
	|\gKasner_{ij}| & \leq t^{2/3 - 2 \tracefreeparameter},
		\label{E:GKASNERBASICESTIMATE} 
		\\
	|\gKasner^{ij}| & \leq t^{-2/3 - 2 \tracefreeparameter}.
		\label{E:GINVERSEKASNERBASICESTIMATE}
\end{align}
\end{subequations}

Furthermore, the $3 \times 3$ matrices $\gKasner_{ij}$ and $\gKasner^{ij}$
have the following positive definiteness properties:
\begin{subequations}
\begin{align}
	t^{2/3 + 2 \tracefreeparameter} \delta_{ab} X^a X^b 
	& \leq \gKasner_{ab} X^a X^b
	\leq t^{2/3 - 2 \tracefreeparameter} \delta_{ab} X^a X^b, && \forall X \in \mathbb{R}^3,
		\label{E:KASNERMETRICPOSITIVITY} \\
	t^{-2/3 + 2 \tracefreeparameter} \delta^{ab} \xi_a \xi_b
	& \leq \gKasner^{ab} \xi_a \xi_b
	\leq t^{- 2/3 - 2 \tracefreeparameter} \delta^{ab} \xi_a \xi_b, && \forall \xi \in \mathbb{R}^3,
		\label{E:INVERSEKASNERMETRICPOSITIVITY} 
\end{align}
\end{subequations}
where $\delta_{ab}$ and $\delta^{ab}$ are standard Kronecker deltas.

Furthermore, 
\begin{align} \label{E:KASNERTIMEDERIVATIVEIDENTITIES} 
	\partial_t \gKasner_{ij} & = - 2 t^{-1} \gKasner_{ia} (t \SecondFundKasner_{\ j}^a),
	&& \partial_t \gKasner^{ij} = 2 t^{-1} \gKasner^{ja} (t \SecondFundKasner_{\ a}^i),
\end{align}
where $t \SecondFundKasner_{\ j}^i = -\mbox{\upshape diag}(q_1,q_2,q_3)$
(see \eqref{E:KASNERSECONDFUNDONEUP}).
\end{lemma}

Before comparing the strength of the energies and the norms, 
we first provide the following simple elliptic estimate,
which will allow us to derive estimates for the top-order derivatives of the linearized lapse.

\begin{lemma} [\textbf{Top-order estimate for} $\LapseRenormalized$]
	\label{L:TOPORDERLAPSEINTEGRATIONBYPARTSINEQUALITY}
	If $\LapseRenormalized$ verifies equation \eqref{E:LINEARIZEDLAPSE},
	then the following\footnote{We note that 
	$\| \partial^2 \LapseRenormalized \|_{L_{\gKasner}^2}^2 =
	\int_{\Sigma_t} 
		\gKasner^{ab} \gKasner^{ef}
	 	\partial_a \partial_e \LapseRenormalized 
	 	\partial_b \partial_f \LapseRenormalized
	 	\, dx$.} 
	elliptic estimate holds:
	\begin{align} \label{E:TOPORDERLAPSEINTEGRATIONBYPARTSINEQUALITY}
		t^2 \| \partial^2 \LapseRenormalized \|_{L_{\gKasner}^2}
		& \lesssim
				|2 A^2 - 1| \| \LapseRenormalized \|_{L^2}
				+ 2 A \| t \partial_t \SFRenormalized \|_{L^2} 
				+ 2 |t \tracefreeSecondFundKasner|_{\gKasner} \| t \LinSecondFund \|_{L_{\gKasner}^2}.
	\end{align}
	
\end{lemma}

\begin{proof}
	We multiply equation \eqref{E:LINEARIZEDLAPSE}
	by $t^2 \gKasner^{ef} \partial_e \partial_f \LapseRenormalized$,
	integrate by parts over $\Sigma_t$ (relative to the Euclidean volume form on $\Sigma_t$), 
	and use Cauchy-Schwarz and Young's inequality
	as well as the simple estimate 
	$ \| \gKasner^{ef} \partial_e \partial_f \LapseRenormalized \|_{L^2} 
	\lesssim \| \partial^2 \LapseRenormalized \|_{L_{\gKasner}^2}
	.
	$
\end{proof}

\begin{lemma}[\textbf{Energy-norm comparison lemma}]
\label{L:ENERGYNORMCOMPARISON}
Let $N \geq 0$ be an integer and let
$\tracefreeparameter \geq 0$ be as defined in \eqref{E:TRACEFREEPARAMETER}.
Under the assumptions of Lemma \ref{L:TOPORDERLAPSEINTEGRATIONBYPARTSINEQUALITY},
there exist constants\footnote{As we have mentioned, $C$ and $c$ are free to vary from line to line and can depend on $N$.} 
$C > 0$ and  $c > 0$, depending on $\smallparameter$,
such that the following comparison estimates hold for 
the norm $\highnorm{N}(t)$ defined in \eqref{E:HIGHNORM}
and the total energy $\mathscr{E}_{(Total);\smallparameter;N}(t)$ 
defined in \eqref{E:TOPORDERENERGY}
on the interval $t \in (0,1]$:
\begin{subequations}
\begin{align}  \label{E:ENERGYNORMCOMPARISON}
	\mathscr{E}_{(Total);\smallparameter;N}(t)
	& \leq C t^{-c \tracefreeparameter} \highnorm{N}(t),
		\\
	\highnorm{N}(t)
	& \leq C t^{-c \tracefreeparameter}
		\mathscr{E}_{(Total);\smallparameter;N}(t).
		\label{E:NORMENERGYCOMPARISON}
\end{align}
\end{subequations}

\end{lemma}

\begin{proof}
	Lemma~\ref{L:ENERGYNORMCOMPARISON} 
	follows easily from Lemma~\ref{L:KASNERMETRIC},
	Lemma~\ref{L:TOPORDERLAPSEINTEGRATIONBYPARTSINEQUALITY}
	(which allows us to bound the top-order linearized lapse term 
	$t^{4/3} \left\| \LapseRenormalized \right\|_{H^{N+2}}$ from \eqref{E:HIGHNORM} 
	in terms of the other linear solution variables),
	and the definitions of the quantities involved.
\end{proof}

\section{The approximate monotonicity identity}
\label{S:MONOTONICITYIDENTITIES}

\subsection{Statement of the approximate monotonicity identity}
The next theorem provides the approximate monotonicity identity that
lies at the heart of the linear stability of near-FLRW Kasner solutions.
Unlike the results of Sects.\ \ref{S:ENERGYESTIMATES} and ~\ref{S:LINEARSTABILITY},
\emph{the identity is valid for all Kasner backgrounds}.

\begin{remark}[\textbf{Monotonicity-coaxing terms and error terms}]
	\label{R:MONOTONICITYANDERROR}
	The favorable ``monotonicity-coaxing terms'' are the negative definite spacetime
	integrals on the third and fourth lines of
	\eqref{E:MONOTONICITYID}. 
	The last line of
	\eqref{E:MONOTONICITYID} features unsigned error integrals that 
	compete against the negative definite integrals.
	In Theorem~\ref{T:L2MILDENERGYBLOWUPCMCGAUGE}, we show that
	for near-FLRW Kasner backgrounds,
	the unsigned integrals can be absorbed into the negative definite
	integrals, except for one error integral whose coefficient
	is controlled by the parameter $\tracefreeparameter$.
\end{remark}

\begin{theorem}[\textbf{The approximate monotonicity identity}]
	\label{T:CMCMONOTONICITYID}
	For any constant $\smallparameter > 0$,
	solutions to the linearized equations of Prop.\ \ref{P:LINEARIZEDCMCEQUATIONS}
	verify the following identity for $t \in (0,1]$:
	\begin{align} \label{E:MONOTONICITYID}
		&
		\int_{\Sigma_t} 
			(t \partial_t \SFRenormalized)^2 
			+ |
			t \partial \SFRenormalized|_{\gKasner}^2 
		\, dx
		+
		\int_{\Sigma_t} 
				|t \partial \LapseRenormalized|_{\gKasner}^2
			\, dx
		+
		(1-A^2)
		\int_{\Sigma_t} 
			\LapseRenormalized^2
		\, dx
		+
		\smallparameter
		\int_{\Sigma_t} 
			|t \LinSecondFund|_{\gKasner}^2 
			+ 
			\frac{1}{4} |t \partial \grenormalized|_{\gKasner}^2 
		\, dx
		-
		\int_{\Sigma_t} 
			\mathcal{N}_1
		\, dx
			\\
		&
		=
		\int_{\Sigma_1} 
			(t \partial_t \SFRenormalized)^2 
			+ |
			t \partial \SFRenormalized|_{\gKasner}^2 
		\, dx
		+
		\int_{\Sigma_1} 
				|t \partial \LapseRenormalized|_{\gKasner}^2
			\, dx
		+
		(1-A^2)
		\int_{\Sigma_1} 
			\LapseRenormalized^2
		\, dx
		+
		\smallparameter
		\int_{\Sigma_1} 
			|\LinSecondFund|_{\gKasner}^2 
			+ 
			\frac{1}{4} |t \partial \grenormalized|_{\gKasner}^2 
		\, dx
		-
		\int_{\Sigma_1} 
			\mathcal{N}_1
		\, dx
			\notag
				\\
		& \ \
			- 2
			\int_{s=t}^1
				s^{-1}
				\int_{\Sigma_s}  
					|s \partial \SFRenormalized|_{\gKasner}^2
				\, dx
		 	\, ds
			- 
			 \int_{s=t}^1 
				s^{-1}
				\int_{\Sigma_s}
					|s \partial \LapseRenormalized|_{\gKasner}^2
		 		\, dx 
		 	\, ds
			- 	\int_{s=t}^1 
					s^{-1}
					\int_{\Sigma_s}
						\LapseRenormalized^2
		 			\, dx 
		 		\, ds
				\notag \\
		& \ \
		-
		\frac{1}{2} \smallparameter
		\int_{s=t}^1
			s^{-1}
			\int_{\Sigma_s}
				|s \partial \grenormalized|_{\gKasner}^2
			\, dx
		\, ds
			\notag \\
		& 
			\ \
			+
			\sum_{i=1}^3
			\int_{s=t}^1
				s^{-1}
				\int_{\Sigma_s}  
					\mathcal{N}_i
				 \, dx
		 	\, ds
			+
			\smallparameter
			\sum_{i=4}^{10}
			\int_{s=t}^1
				s^{-1}
				\int_{\Sigma_s}  
					\mathcal{N}_i
				 \, dx
		 	\, ds,
			\notag
	\end{align}
where the constant $0 \leq A \leq \sqrt{2/3}$ is defined by \eqref{E:KASNERHAMILTONIANCONSTRAINT} 
and along $\Sigma_s$, we have
\begin{subequations}
\begin{align}
	\mathcal{N}_1 
	& = \mathcal{N}_1(s \tracefreeSecondFundKasner,s \LinSecondFund,\LapseRenormalized)
	:= - 2 (s \tracefreeSecondFundKasner_{\ b}^a) (s \LinSecondFund_{\ a}^b) \LapseRenormalized,
		\label{E:CUBICFORM1} \\
	\mathcal{N}_2 
	 & = \mathcal{N}_2(s \SecondFundKasner, s \partial \SFRenormalized,s \partial \SFRenormalized)
		:= -2 s^2 \gKasner^{ab} (s \SecondFundKasner_{\ b}^c) \partial_a \SFRenormalized \partial_c \SFRenormalized,
		\label{E:CUBICFORM2} \\
	\mathcal{N}_3 
	& = \mathcal{N}_3(s \partial \SFRenormalized,s \partial \LapseRenormalized)
	:= - 2 A s^2 \gKasner^{ab} \partial_a \SFRenormalized \partial_b \LapseRenormalized,
		\label{E:QUADRATICFORM1}
			\\
\mathcal{N}_4 
& = \mathcal{N}_4(s \SecondFundKasner, s \partial \grenormalized, s \partial \grenormalized)
:=   	-\frac{1}{2}
						s^2
						\gKasner^{ab} \gKasner^{ij} \gKasner^{cf}
						(s \SecondFundKasner_{\ c}^e)
						\partial_e \grenormalized_{ai}
						\partial_f \grenormalized_{bj},
						\label{E:CUBICFORM3}	\\
\mathcal{N}_5 
& = \mathcal{N}_5(s \tracefreeSecondFundKasner,s \LinSecondFund,s \LinSecondFund)
:= 2 \gKasner_{ic} \gKasner^{ab} (s \tracefreeSecondFundKasner_{\ j}^c)
				(s \LinSecondFund_{\ a}^i) (s \LinSecondFund_{\ b}^j)
			- 2 \gKasner_{ij} \gKasner^{ac} (s \tracefreeSecondFundKasner_{\ c}^b)
				(s \LinSecondFund_{\ a}^i) (s \LinSecondFund_{\ b}^j),
					\label{E:CUBICFORM4} \\
	\mathcal{N}_6 
	& = \mathcal{N}_6(s \tracefreeSecondFundKasner,s \partial \grenormalized,s \partial \grenormalized)
	:= 		s^2 
					\gKasner_{ab} \gKasner^{ef} \gKasner^{ij} 
					(s \tracefreeSecondFundKasner_{\ c}^a)
					\christrenormalizedarg{i}{c}{j} 
					\christrenormalizedarg{e}{b}{f}
				- s^2
					\gKasner_{ab} \gKasner^{ef} \gKasner^{ij} 
					(s \tracefreeSecondFundKasner_{\ j}^c)
					\christrenormalizedarg{i}{a}{c} 
					\christrenormalizedarg{e}{b}{f}
		\label{E:CUBICFORM5} 			\\
	& \ \ + s^2 \gKasner^{ef}  (s \tracefreeSecondFundKasner_{\ c}^a) 
						\christrenormalizedarg{a}{c}{b} \christrenormalizedarg{e}{b}{f}	
					- s^2 \gKasner^{ef} (s \tracefreeSecondFundKasner_{\ b}^c) 
							\christrenormalizedarg{a}{a}{c} \christrenormalizedarg{e}{b}{f},
		\notag
			\\
	\mathcal{N}_7 
	& = 
	\mathcal{N}_7(s \SecondFundKasner,s \partial \grenormalized,s \partial \LapseRenormalized) 
	:= 2 s^2 \gKasner^{ij} (s \tracefreeSecondFundKasner_{\ i}^b)  
						\christrenormalizedarg{a}{a}{b} \partial_j \LapseRenormalized
						- 2 s^2 \gKasner^{ij}  (s \tracefreeSecondFundKasner_{\ b}^a) \christrenormalizedarg{a}{b}{i} 
							\partial_j \LapseRenormalized
							\label{E:CUBICFORM6} \\
		& \ \ + s^2 \gKasner^{ij} \gKasner^{ef} (s \SecondFundKasner_{\ j}^a)
				  	\partial_e \grenormalized_{ai} \partial_f \LapseRenormalized,
				  	\notag	\\
	\mathcal{N}_8 
	& = \mathcal{N}_8(s \tracefreeSecondFundKasner,s \LinSecondFund,\LapseRenormalized)
	:= 	2 \gKasner_{ab} \gKasner^{ij} (s \tracefreeSecondFundKasner_{\ i}^a) (s \LinSecondFund_{\ j}^b) \LapseRenormalized,
				\label{E:CUBICFORM7} \\
	\mathcal{N}_9  
	& = \mathcal{N}_9(s \partial \SFRenormalized, s \partial \LapseRenormalized)
	:= 	2 A s^2\gKasner^{ij} \partial_i \SFRenormalized \partial_j \LapseRenormalized,
				\label{E:QUADRATICFORM2} \\
	\mathcal{N}_{10} 
	& = \mathcal{N}_{10}(s \partial \grenormalized, s \partial \SFRenormalized)
	:= - 2 A s^2 \gKasner^{ef} \christrenormalizedarg{e}{a}{f} \partial_a \SFRenormalized.
		\label{E:QUADRATICFORM3}
\end{align}	
\end{subequations}
\end{theorem}

\begin{proof}
	Below we independently derive
	the identities
	\eqref{E:SECONDENERGYESTIMATEMODELSCALARFIELD}
	and
	\eqref{E:METRICENERGYID}.
	To obtain \eqref{E:MONOTONICITYID},
	we simply add
	\eqref{E:SECONDENERGYESTIMATEMODELSCALARFIELD}
	to $\smallparameter$ times \eqref{E:METRICENERGYID}.
\end{proof}	

\begin{corollary}[\textbf{Approximate monotonicity identity for the solution's higher derivatives}]
	\label{C:MONOTONICITY}
	For any spatial derivative multi-index $\vec{I}$ (as defined in Subsect.\ \ref{SS:COORDINATES}),
	the identity \eqref{E:MONOTONICITYID} holds with
	$\LinSecondFund, \partial \grenormalized, \SFRenormalized, \LapseRenormalized$
	replaced with, respectively,
	$\partial_{\vec{I}} \LinSecondFund, \partial \partial_{\vec{I}} \grenormalized, 
	\partial_{\vec{I}} \SFRenormalized, \partial_{\vec{I}} \LapseRenormalized$.
\end{corollary}

\begin{proof}
	Since the Kasner background metric is spatially homogeneous 
	(that is, independent of $x \in \mathbb{T}^3$),
	the operators $\partial_{\vec{I}}$ commute through the linear equations of
	Prop.\ \ref{P:LINEARIZEDCMCEQUATIONS}. Put differently,
	the differentiated quantities
	$\partial_{\vec{I}} \LinSecondFund, \partial \partial_{\vec{I}} \grenormalized, 
	\partial_{\vec{I}} \SFRenormalized, \partial_{\vec{I}} \LapseRenormalized$
	verify the same equations
	satisfied by
	$\LinSecondFund, \partial \grenormalized, \SFRenormalized, \LapseRenormalized$.
	Hence, Theorem~\ref{E:MONOTONICITYID}
	applies to the differentiated quantities as well.
\end{proof}

\subsection{The key integral identity for the linearized lapse and scalar field}
The most important ingredient in the proof of Theorem~\ref{T:CMCMONOTONICITYID}
is the following proposition, which provides 
an integral identity for the linearized scalar field and the linearized lapse.
The proof of the proposition essentially involves combining 
several integration by parts-type identities 
in a manner that replaces dangerous error integrals with favorable ones.

\begin{proposition}[\textbf{The key integral 
	identity for the linearized scalar field and the linearized lapse}]
	\label{P:ENERGYESTIMATELAPSEANDSCALARFIELD}
	Solutions to the linearized equations of Prop.\ \ref{P:LINEARIZEDCMCEQUATIONS}
	verify the following identity for $t \in (0,1]$:
	\begin{align}
	& \int_{\Sigma_t} 
		(t \partial_t \SFRenormalized)^2
		+ 
		|t \partial \SFRenormalized|_{\gKasner}^2
	\, dx  
	+	\int_{\Sigma_t} 
			|t \partial \LapseRenormalized|_{\gKasner}^2
		\, dx
	+   (1 - A^2)
			\int_{\Sigma_t} 
				\LapseRenormalized^2
			\, dx 
	-
	\int_{\Sigma_t} 
		\mathcal{N}_1
	\, dx
			\label{E:SECONDENERGYESTIMATEMODELSCALARFIELD}  
				\\
	& = 
		\int_{\Sigma_1} 
			(\partial_t \SFRenormalized)^2
			+ |\partial \SFRenormalized|_{\gKasner}^2
		\, dx  
	+	\int_{\Sigma_1} 
				|\partial \LapseRenormalized|_{\gKasner}^2
			\, dx
	+ 	(1 - A^2)
			\int_{\Sigma_1} 
				\LapseRenormalized^2
			\, dx 
	-  \int_{\Sigma_1} 
				\mathcal{N}_1
			 \, dx		
			\notag \\
	& \ \ 
			- 2
			\int_{s=t}^1
				s^{-1}
				\int_{\Sigma_s}  
					|s \partial \SFRenormalized|_{\gKasner}^2
				\, dx
		 	\, ds
			- 
			 \int_{s=t}^1 
				s^{-1}
				\int_{\Sigma_s}
					|s \partial \LapseRenormalized|_{\gKasner}^2
		 		\, dx 
		 	\, ds
			- 	\int_{s=t}^1 
					s^{-1}
					\int_{\Sigma_s}
						\LapseRenormalized^2
		 			\, dx 
		 		\, ds
				\notag	\\
	 & \ \
			+
			\sum_{i=1}^3
			\int_{s=t}^1
				s^{-1}
				\int_{\Sigma_s}  
					\mathcal{N}_i
				 \, dx
		 	\, ds,
		 \notag
\end{align}
where the constant $0 \leq A \leq \sqrt{2/3}$ is defined by \eqref{E:KASNERHAMILTONIANCONSTRAINT} and
$\mathcal{N}_1$, $\mathcal{N}_2$, and $\mathcal{N}_3$ are defined in
\eqref{E:CUBICFORM1}-\eqref{E:QUADRATICFORM1}.

\end{proposition}

\begin{remark}
	\emph{The surprising aspect of the identity \eqref{E:SECONDENERGYESTIMATEMODELSCALARFIELD}
	is the presence of the spacetime integrals that are negative definite in 
	$\LapseRenormalized$ and $\partial \LapseRenormalized$.}
 	In Sect.\ \ref{S:PARABOLICMONOTONICITY}, we show that
	a version of \eqref{E:SECONDENERGYESTIMATEMODELSCALARFIELD} 
	also holds when the CMC gauge is replaced with 
	a parabolic lapse gauge.
\end{remark}

\begin{proof}[Proof of Prop.\ \ref{P:ENERGYESTIMATELAPSEANDSCALARFIELD}]
The proof involves combining three integration by parts identities.
Throughout, we silently use the identities in \eqref{E:KASNERTIMEDERIVATIVEIDENTITIES}.
To obtain the first identity, we multiply both sides of the linearized lapse equation 
\eqref{E:LINEARIZEDLAPSE} by
$\LapseRenormalized$ and integrate by parts over $\Sigma_t$ to deduce that
\begin{align}  \label{E:KEYIDENTITY}
		2 A \int_{\Sigma_t}
	 			(t \partial_t \SFRenormalized) \LapseRenormalized 
	 		 \, dx
	& = - 
			\int_{\Sigma_t} 
				|t \partial \LapseRenormalized|_{\gKasner}^2
			\, dx
		+ (2 A^2 - 1)
			\int_{\Sigma_t} 
				\LapseRenormalized^2
			\, dx 
	  - 	2 
	  		\int_{\Sigma_t} 
					(t \tracefreeSecondFundKasner_{\ b}^a) (t \LinSecondFund_{\ a}^b) \LapseRenormalized 
			 	\, dx.
\end{align}

The second identity is an energy identity for the linearized scalar field wave equation.
Specifically, we replace $t$ with the integration variable $s$ in equation \eqref{E:SCALARFIELDWAVEDECOMPOSED},
multiply by $- 2s \partial_t \SFRenormalized$, 
and integrate by parts 
over $(s,x) \in [t,1] \times \mathbb{T}^3$ (we stress that $t \leq 1$) to deduce that the following identity holds for
$t \in (0,1]$:
\begin{align}
	\int_{\Sigma_t} 
		(t \partial_t \SFRenormalized)^2
		+ |t \partial \SFRenormalized|_{\gKasner}^2
	\, dx  
	& = \int_{\Sigma_1} 
				(\partial_t \SFRenormalized)^2 
				+ |\partial \SFRenormalized|_{\gKasner}^2
			\, dx 
		\label{E:FIRSTENERGYESTIMATEMODELSCALARFIELD} \\
	& \ \ 
		- 2
			\int_{s=t}^1
				s^{-1}
				\int_{\Sigma_s}  
					 |s \partial \SFRenormalized|_{\gKasner}^2
				 	 + s^2 \gKasner^{ab} (s \SecondFundKasner_{\ b}^c)
							\partial_a \SFRenormalized \partial_c \SFRenormalized
		 	 	\, dx
		 	\, ds
		\notag  \\
	& \ \ 
			- 2 A
			\int_{s=t}^1 
				\int_{\Sigma_s}
					(s \partial_t \SFRenormalized) \partial_t \LapseRenormalized
		 		\, dx 
		 	\, ds
			+ 2 A
			\int_{s=t}^1 
				s^{-1}
				\int_{\Sigma_s}
					(s \partial_t \SFRenormalized) \LapseRenormalized
		 		\, dx 
		 	\, ds.
		 \notag
\end{align}

Next, we multiply equation \eqref{E:SCALARFIELDWAVEDECOMPOSED} by $\LapseRenormalized$ 
to obtain the following identity:
\begin{align} \label{E:LINEARIZEDSCALARFIELDLAPSEMIXEDEQUATION}
	(t \partial_t \SFRenormalized) \partial_t \LapseRenormalized
	& = \partial_t (t \partial_t \SFRenormalized \LapseRenormalized)
		- \frac{1}{2} A \partial_t (\LapseRenormalized^2)
	 	- t \LapseRenormalized \gKasner^{ab} \partial_a \partial_b \SFRenormalized 
	 	+ A t^{-1} \LapseRenormalized^2.
\end{align}
To obtain the third identity, we now replace $t$ with the integration variable $s$ in equation \eqref{E:LINEARIZEDSCALARFIELDLAPSEMIXEDEQUATION},
multiply by $2A$,
and integrate by parts 
over $(s,x) \in [t,1] \times \mathbb{T}^3$
to deduce that 
\begin{align}  \label{E:LINEARIZEDSCALARFIELDLAPSEMIXEDEQUATIONIBPIDENTITY}
		- 2 A
			\int_{s=t}^1 
				\int_{\Sigma_s}
					(s \partial_t \SFRenormalized) \partial_t \LapseRenormalized
		 		\, dx 
		 	\, ds
		 	& = 
		 		- 2A
		 			\int_{\Sigma_1}
		 				\partial_t \SFRenormalized \LapseRenormalized
		 			\, dx
		 		+ A^2 
		 			\int_{\Sigma_1}
		 				\LapseRenormalized^2
		 			\, dx
		 			\\
		 & \ \ 
		 		+ 2A
		 			\int_{\Sigma_t}
		 				(t \partial_t \SFRenormalized) \LapseRenormalized
		 			\, dx
		 		- A^2 
		 			\int_{\Sigma_t}
		 				\LapseRenormalized^2
		 			\, dx
		 		\notag		\\
		& \ \
			- 2A
			 \int_{s=t}^1 
				s^{-1}
				\int_{\Sigma_s}
					s^2 \gKasner^{ab} \partial_a \SFRenormalized \partial_b \LapseRenormalized
		 		\, dx 
		 	\, ds
		 - 2 A^2 
		 		\int_{s=t}^1 
					s^{-1}
					\int_{\Sigma_s}
						\LapseRenormalized^2
		 			\, dx 
		 		\, ds
		 	\notag \\
		& = 
					\int_{\Sigma_1} 
						|\partial \LapseRenormalized|_{\gKasner}^2
					\, dx
				+ 
				(1- A^2)
		 			\int_{\Sigma_1}
		 				\LapseRenormalized^2
		 			\, dx
		 + 2 
	  		\int_{\Sigma_1} 
					\tracefreeSecondFundKasner_{\ b}^a \LinSecondFund_{\ a}^b \LapseRenormalized 
			 \, dx
			  \notag \\
		 & \ \ 
				- \int_{\Sigma_t} 
						|t \partial \LapseRenormalized|_{\gKasner}^2
					\, dx
				- 
				(1- A^2)
		 			\int_{\Sigma_t}
		 				\LapseRenormalized^2
		 			\, dx
		 - 2 
	  		\int_{\Sigma_t} 
					(t \tracefreeSecondFundKasner_{\ b}^a) (t \LinSecondFund_{\ a}^b) \LapseRenormalized 
			 \, dx
			 \notag \\
		& \ \
			- 2A
			 \int_{s=t}^1 
			  s^{-1}
				\int_{\Sigma_s}
					s^2 \gKasner^{ab} \partial_a \SFRenormalized \partial_b \LapseRenormalized
		 		\, dx 
		 	\, ds
		 - 2 A^2 
		 		\int_{s=t}^1 
					s^{-1}
					\int_{\Sigma_s}
						\LapseRenormalized^2
		 			\, dx 
		 		\, ds, 
		 		\notag
\end{align}
where to obtain the second equality, we substituted the right-hand side
of \eqref{E:KEYIDENTITY} for the integrals
$2A
		 			\int_{\Sigma_1}
		 				\partial_t \SFRenormalized \LapseRenormalized
		 			\, dx$
and		 			
$2A
		 			\int_{\Sigma_t}
		 				(t \partial_t \SFRenormalized) \LapseRenormalized
		 			\, dx$.
We now use the identity \eqref{E:KEYIDENTITY}
with $t$ replaced by $s$ 
to substitute for the integral 
$2A \int_{\Sigma_s}
			(s \partial_t \SFRenormalized) \LapseRenormalized
\, dx$ in the last spacetime integral on the right-hand side
\eqref{E:FIRSTENERGYESTIMATEMODELSCALARFIELD}.
Finally, we substitute the right-hand side of
\eqref{E:LINEARIZEDSCALARFIELDLAPSEMIXEDEQUATIONIBPIDENTITY}
for the next-to-last spacetime integral on the right-hand
side of \eqref{E:FIRSTENERGYESTIMATEMODELSCALARFIELD}.
In total, these steps lead to the identity \eqref{E:SECONDENERGYESTIMATEMODELSCALARFIELD}.

\end{proof}

\subsection{An energy identity for the linearized metric variables}
In the next proposition, we derive an energy identity for
the linearized metric solution variables.

\begin{proposition}[\textbf{Energy identity for the linearized metric variables}]
\label{PCMC:LINEARIZEDMETRICENERGYESTIMATE}
Solutions to the linearized equations of Prop.\ \ref{P:LINEARIZEDCMCEQUATIONS}
verify the following identity for $t \in (0,1]$:
\begin{align} \label{E:METRICENERGYID}
	\int_{\Sigma_t}
		|t \LinSecondFund|_{\gKasner}^2
		+ 
		\frac{1}{4} |t \partial \grenormalized|_{\gKasner}^2
	\, dx
	& = 
		\int_{\Sigma_1}
			|\LinSecondFund|_{\gKasner}^2
			+ 
			\frac{1}{4} |\partial \grenormalized|_{\gKasner}^2
		\, dx 
			\\
	& \ \
		- \frac{1}{2}
		\int_{s=t}^1
			s^{-1}
			\int_{\Sigma_s}
				 		|s \partial \grenormalized|_{\gKasner}^2
		\, dx
		\, ds
			\notag	\\
	& \ \
			+
			\sum_{i=4}^{10}
			\int_{s=t}^1
				s^{-1}
				\int_{\Sigma_s}  
					\mathcal{N}_i
				 \, dx
		 	\, ds,
			\notag
\end{align}			
where
$\mathcal{N}_4$, $\cdots$, $\mathcal{N}_{10}$
are defined in \eqref{E:CUBICFORM3}-\eqref{E:QUADRATICFORM3}.
\end{proposition}

\begin{remark}[\textbf{No need for spatial harmonic coordinates}]
\label{R:NONEEDFORSPATIALHARMONIC}
Prop.\ \ref{PCMC:LINEARIZEDMETRICENERGYESTIMATE} shows in particular that we can
derive energy estimates for solutions to Einstein's equations\footnote{Although the proposition 
addresses only the linearized equations, essentially the same argument can be used to derive
a similar energy identity for the nonlinear equations.} 
directly in CMC-transported spatial coordinates.
Remarkably, we have not seen this observation made in the literature.
Previous authors (see, for example, \cites{lAvM2003,lAvM2011}) 
have instead chosen to 
impose the spatial harmonic coordinate condition $g^{ab} \nabla_a \nabla_b x^i = 0$
to ``reduce'' the Ricci tensor $R_{ij}$ of $g$ to an elliptic operator acting on
the components $g_{ij}$. That is, in spatial harmonic coordinates, we have
$R_{ij} = - \frac{1}{2} g^{ab} \partial_a \partial_b g_{ij} + f_{ij}(g,\partial g)$,
which eliminates the last two products on the right-hand side of \eqref{E:RICRENORMALIZED}
and leads to a simpler proof of a basic $L^2$-type energy identity.
In the proof of Prop.\ \ref{PCMC:LINEARIZEDMETRICENERGYESTIMATE}, we handle
these two products through a procedure involving integration by parts and the constraint equations;
see equations \eqref{E:SECONDLINEARIZEDMOMENTUMDIFFBYPARTS} and \eqref{E:FIRSTLINEARIZEDMOMENTUMDIFFBYPARTS}.
The spatial harmonic coordinate condition, though it might have advantages in certain contexts,
introduces additional complications into the analysis. 
The complications arise from the necessity of including a non-zero ``shift vector'' $X^i$
in the spacetime metric $\gfour$: $\gfour = - n^2 dt^2 + g_{ab}(dx^a + X^a dt)(dx^b + X^b dt)$. To 
enforce the spatial harmonic coordinate condition, the components
$X^i$ must verify a system of elliptic PDEs that are coupled to the other solution variables.
\end{remark}

\begin{proof}[Proof of Prop.\ \ref{PCMC:LINEARIZEDMETRICENERGYESTIMATE}]
The proof involves combining a collection of integration by parts identities.
Throughout, we silently use the identities in \eqref{E:KASNERTIMEDERIVATIVEIDENTITIES}.
To begin, we use the evolution equation \eqref{E:LINEARIZEDKEVOLUTION} to deduce that
\begin{align} \label{E:FIRSTPARTIALTRENORMALIZEDKSQUARED}
	\partial_t (|t \LinSecondFund|_{\gKasner}^2)
	& = - 2 t^{-1} \gKasner_{ic} \gKasner^{ab} (t \SecondFundKasner_{\ j}^c) 
				(t \LinSecondFund_{\ a}^i) (t \LinSecondFund_{\ b}^j)
			+ 2 t^{-1} \gKasner_{ij} \gKasner^{ac} (t \SecondFundKasner_{\ c}^b) 
				(t \LinSecondFund_{\ a}^i) (t \LinSecondFund_{\ b}^j)
				\\
	& \ \ 
		+ 2 \gKasner_{ab} \gKasner^{ij}(t \LinSecondFund_{\ i}^a)
			\left\lbrace
				- t \gKasner^{bc} \partial_c \partial_j \LapseRenormalized
				- t^{-1} \LapseRenormalized (t \SecondFundKasner_{\ j}^b)
				+ t \Ricrenormalizedarg{b}{j}
			\right\rbrace.
			\notag
\end{align}
Note that we can express the 
first line of the right-hand side of 
\eqref{E:FIRSTPARTIALTRENORMALIZEDKSQUARED}
as
\begin{align} \label{E:TRACEFREEREPLACEMENT}
	- 2 t^{-1} \gKasner_{ic} \gKasner^{ab} (t \tracefreeSecondFundKasner_{\ j}^c)
					(t \LinSecondFund_{\ a}^i) (t \LinSecondFund_{\ b}^j)
			+ 2 t^{-1} \gKasner_{ij} \gKasner^{ac} (t \tracefreeSecondFundKasner_{\ c}^b) 
				(t \LinSecondFund_{\ a}^i) (t \LinSecondFund_{\ b}^j)
\end{align}
because the terms corresponding to the pure trace part of
$\SecondFundKasner$ cancel.
Furthermore, 
since equation \eqref{E:CMCMODEL} implies that
$\LinSecondFund = \tracefreeLinSecondFund$,
we can express the second product on the 
second line of the right-hand side of \eqref{E:FIRSTPARTIALTRENORMALIZEDKSQUARED} as
\begin{align} \label{E:ANOTHERTRACEFREEREPLACEMENT}
	- 2 t^{-1} 2 \gKasner_{ab} \gKasner^{ij}(t \LinSecondFund_{\ i}^a) (t \SecondFundKasner_{\ j}^b) \LapseRenormalized 
	& = - 2 \gKasner_{ab} \gKasner^{ij} (s \tracefreeSecondFundKasner_{\ i}^a) (s \LinSecondFund_{\ j}^b) \LapseRenormalized.
\end{align}

Similarly, using the evolution equation \eqref{E:LINEARIZEDGEVOLUTION}, we deduce that\footnote{We recall that 
$
|\partial \grenormalized|_{\gKasner}^2 
= \gKasner^{ab} \gKasner^{ij} \gKasner^{ef} \partial_e \grenormalized_{ai} \partial_f \grenormalized_{bj}.
$}
\begin{align} \label{E:FIRSTPARTIALTTSQUAREDRENORMALIZEDGSQUARED}
	\frac{1}{4} \partial_t (t^2 |\partial \grenormalized|_{\gKasner}^2)
	& = \frac{1}{2} t |\partial \grenormalized|_{\gKasner}^2
		+
		t \gKasner^{bc} \gKasner^{ij} \gKasner^{ef}
				(t \SecondFundKasner_{\ c}^a)
				\partial_e \grenormalized_{ai}
				\partial_f \grenormalized_{bj}
		+ \frac{1}{2} 
				t	
				\gKasner^{ab} \gKasner^{ij} \gKasner^{cf}
				(t \SecondFundKasner_{\ c}^e)
				\partial_e \grenormalized_{ai}
				\partial_f \grenormalized_{bj}
			\\
	& \ \
		+
		\frac{1}{2} \gKasner^{ab} \gKasner^{ij} \gKasner^{ef}	
		\partial_e \grenormalized_{ai}
		\partial_f
		\left\lbrace
			- 2 t \grenormalized_{bc} (t \SecondFundKasner_{\ j}^c)
			- 2 t \gKasner_{bc} (t \LinSecondFund_{\ j}^c)
			- 2 t \gKasner_{bc} (t \SecondFundKasner_{\ j}^c) \LapseRenormalized 
		\right\rbrace.
			\notag
\end{align}

For convenience, in the remainder of this proof, we denote terms that
can be expressed as perfect spatial derivatives by ``$\cdots$."
These terms will vanish when we integrate the identities 
over $\mathbb{T}^3$.
We now use equation \eqref{E:LINEARIZEDMOMENTUM} and differentiation by parts
to express the first product on the 
second line of the right-hand side of \eqref{E:FIRSTPARTIALTRENORMALIZEDKSQUARED} as
\begin{align} \label{E:FIRSTDIFFBYPARTS}
	- 2t \gKasner_{ab} \gKasner^{bc} \gKasner^{ij}(t \LinSecondFund_{\ i}^a)
		\partial_c \partial_j \LapseRenormalized
		& = 2t \gKasner^{ij} \partial_a(t \LinSecondFund_{\ i}^a) \partial_j \LapseRenormalized
			+ \cdots \\
		& = - 2A t \gKasner^{ij} \partial_i \SFRenormalized \partial_j \LapseRenormalized
			\notag \\
		& \ \ 
			- 2 t \gKasner^{ij} \christrenormalizedarg{a}{a}{b} (t \tracefreeSecondFundKasner_{\ i}^b) \partial_j \LapseRenormalized
			+ 2 t \gKasner^{ij} \christrenormalizedarg{a}{b}{i} (t \tracefreeSecondFundKasner_{\ b}^a) \partial_j \LapseRenormalized
				\notag \\
		& \ \
			+ \cdots.
				\notag
\end{align}

Next, we use equation \eqref{E:RICRENORMALIZED} 
to express the third product on the 
second line of the right-hand side of \eqref{E:FIRSTPARTIALTRENORMALIZEDKSQUARED} as
\begin{align} \label{E:LINEARIZEDKEVOLUTIONRICTERMMADEEXPLICIT}
	2 t \gKasner_{ab} \gKasner^{ij}(t \LinSecondFund_{\ i}^a) \Ricrenormalizedarg{b}{j}
	& = - t \gKasner^{ij} \gKasner^{ef} 
				(t \LinSecondFund_{\ i}^a) \partial_e \partial_f \grenormalized_{ja}
				\\
	& \ \ + t \gKasner_{ab} \gKasner^{ij} \gKasner^{ef} 
				(t \LinSecondFund_{\ i}^a)
				\partial_j \christrenormalizedarg{e}{b}{f}
			+ t (t \LinSecondFund_{\ b}^a)
				\gKasner^{ef} \partial_a \christrenormalizedarg{e}{b}{f}.
				\notag
\end{align}

Next, we use differentiation by parts
to express the first product on the right-hand side of 
\eqref{E:LINEARIZEDKEVOLUTIONRICTERMMADEEXPLICIT}
as
\begin{align} \label{E:FIRSTRICCIPRODUCT}
	- t \gKasner^{ij} \gKasner^{ef} 
				(t \LinSecondFund_{\ i}^a) \partial_e \partial_f \grenormalized_{ja}
	& = t \gKasner^{ij} \gKasner^{ef} 
				\partial_e(t \LinSecondFund_{\ i}^a) \partial_f \grenormalized_{ja}
				+ \cdots.
\end{align}

Next, we use equation \eqref{E:LINEARIZEDSECONDMOMENTUM}
and differentiation by parts
to express the second product on the right-hand side of 
\eqref{E:LINEARIZEDKEVOLUTIONRICTERMMADEEXPLICIT}
as
\begin{align} \label{E:SECONDLINEARIZEDMOMENTUMDIFFBYPARTS}
	t \gKasner_{ab} \gKasner^{ij} \gKasner^{ef} 
		(t \LinSecondFund_{\ i}^a)
		\partial_j \christrenormalizedarg{e}{b}{f}
		& = - t \gKasner_{ab} \gKasner^{ij} \gKasner^{ef} 
				\partial_j(t \LinSecondFund_{\ i}^a)
				\christrenormalizedarg{e}{b}{f}
				+ \cdots
				\\
		& = A t \gKasner^{ef} \partial_a \SFRenormalized \christrenormalizedarg{e}{a}{f}
			\notag \\
		& \ \ 
			+ t \gKasner_{ab} \gKasner^{ef} \gKasner^{ij} (t \tracefreeSecondFundKasner_{\ j}^c)
				\christrenormalizedarg{i}{a}{c} 
				\christrenormalizedarg{e}{b}{f}
			- t \gKasner_{ab} \gKasner^{ef} \gKasner^{ij} 
				(t \tracefreeSecondFundKasner_{\ c}^a)
				\christrenormalizedarg{i}{c}{j} 
				\christrenormalizedarg{e}{b}{f}
				\notag \\
		& \ \ 
				+ \cdots. 
				\notag 
\end{align}

Next, we use equation \eqref{E:LINEARIZEDMOMENTUM}
and differentiation by parts
to express the third product on the right-hand side of 
\eqref{E:LINEARIZEDKEVOLUTIONRICTERMMADEEXPLICIT}
as
\begin{align} \label{E:FIRSTLINEARIZEDMOMENTUMDIFFBYPARTS}
	t (t \LinSecondFund_{\ b}^a)
		\gKasner^{ef} \partial_a \christrenormalizedarg{e}{b}{f}
	& = - t \gKasner^{ef} 
			 \partial_a (t \LinSecondFund_{\ b}^a)
			 \christrenormalizedarg{e}{b}{f} 					
			 + \cdots  \\
  & = A t \gKasner^{ef} \partial_b \SFRenormalized \christrenormalizedarg{e}{b}{f}
		+ t \gKasner^{ef} (t \tracefreeSecondFundKasner_{\ b}^c) \christrenormalizedarg{a}{a}{c} \christrenormalizedarg{e}{b}{f}
		- t \gKasner^{ef}  (t \tracefreeSecondFundKasner_{\ c}^a) \christrenormalizedarg{a}{c}{b} \christrenormalizedarg{e}{b}{f}
		+ \cdots. \notag
\end{align}

Combining \eqref{E:FIRSTPARTIALTRENORMALIZEDKSQUARED}-\eqref{E:FIRSTLINEARIZEDMOMENTUMDIFFBYPARTS}
and carrying out straightforward computations,
we deduce that
\begin{align} \label{E:METRICENERGYIDDIFFERNTIALFORM}
	\partial_t (|t \LinSecondFund|_{\gKasner}^2)
	+ \frac{1}{4} \partial_t (t^2 |\partial \grenormalized|_{\gKasner}^2)
	& = \frac{1}{2} t |\partial \grenormalized|_{\gKasner}^2
			+ \frac{1}{2} 
				t \gKasner^{ab} \gKasner^{ij} \gKasner^{cf}
				(t \SecondFundKasner_{\ c}^e)
				\partial_e \grenormalized_{ai}
				\partial_f \grenormalized_{bj}
				\\
	& \ \ 
		- 2 t^{-1} \gKasner_{ic} \gKasner^{ab} (t \tracefreeSecondFundKasner_{\ j}^c) 
				(t \LinSecondFund_{\ a}^i) (t \LinSecondFund_{\ b}^j)
		+ 2 t^{-1} \gKasner_{ij} \gKasner^{ac} (t \tracefreeSecondFundKasner_{\ c}^b) 
				(t \LinSecondFund_{\ a}^i) (t \LinSecondFund_{\ b}^j)
			\notag \\
	& \ \ 
			+ t \gKasner_{ab} \gKasner^{ef} \gKasner^{ij} (t \tracefreeSecondFundKasner_{\ j}^c)
				\christrenormalizedarg{i}{a}{c} 
				\christrenormalizedarg{e}{b}{f}
			- t \gKasner_{ab} \gKasner^{ef} \gKasner^{ij} 
				(t \tracefreeSecondFundKasner_{\ c}^a)
				\christrenormalizedarg{i}{c}{j} 
				\christrenormalizedarg{e}{b}{f}
			\notag \\
	& \ \ 
		+ t \gKasner^{ef} (t \tracefreeSecondFundKasner_{\ b}^c) 
				\christrenormalizedarg{a}{a}{c} \christrenormalizedarg{e}{b}{f}
		- t \gKasner^{ef}  (t \tracefreeSecondFundKasner_{\ c}^a) 
				\christrenormalizedarg{a}{c}{b} \christrenormalizedarg{e}{b}{f}	
		\notag \\
	& \ \	
			- 2 t \gKasner^{ij} (t \tracefreeSecondFundKasner_{\ i}^b)  \christrenormalizedarg{a}{a}{b} \partial_j \LapseRenormalized
			+ 2 t \gKasner^{ij}  (t \tracefreeSecondFundKasner_{\ b}^a) \christrenormalizedarg{a}{b}{i} \partial_j \LapseRenormalized
			- t \gKasner^{ij} \gKasner^{ef} (t \SecondFundKasner_{\ j}^a)
				  \partial_e \grenormalized_{ai} \partial_f \LapseRenormalized
			\notag \\	
& \ \ - 2 A t \gKasner^{ij} \partial_i \SFRenormalized \partial_j \LapseRenormalized 
			- 2 \gKasner_{ab} \gKasner^{ij} (s \tracefreeSecondFundKasner_{\ i}^a) (s \LinSecondFund_{\ j}^b) \LapseRenormalized
			+ 2 A t \gKasner^{ef} \partial_a \SFRenormalized \christrenormalizedarg{e}{a}{f}
			\notag \\	
& \ \
	+ \cdots.
	\notag
\end{align}
To conclude \eqref{E:METRICENERGYID},
we have only to replace $t$ with the integration variable $s$ in the identity \eqref{E:METRICENERGYIDDIFFERNTIALFORM}
and to integrate by parts over $(s,x) \in [t,1] \times \mathbb{T}^3$, where we stress that $t \leq 1$.
\end{proof}

\section{Mildly singular energy estimates without derivative loss for the linearized equations}
\label{S:ENERGYESTIMATES}
In the next result, Theorem~\ref{T:L2MILDENERGYBLOWUPCMCGAUGE}, we use
the approximate monotonicity identity of Theorem~\ref{T:CMCMONOTONICITYID}
to derive energy estimates for solutions to the linearized
equations. A central aspect  
of the estimates is that the energies
can blow up as $t \downarrow 0$.
Consequently, the energy estimates by themselves
do not yield a proof of linear stability.
However, for near FLRW backgrounds, the 
blowup-rate is mild (see \eqref{E:ORDERNTOTALENERGYGRONWALLED}),
which is a key ingredient in our subsequent proof of linear stability
(see Theorem~\ref{T:CMCLINEARSTABILITY}).
We stress that if our proof of Theorem~\ref{T:L2MILDENERGYBLOWUPCMCGAUGE} 
had relied on more standard energy identities 
rather than the approximate monotonicity identity of Theorem~\ref{T:CMCMONOTONICITYID},
then the outcome would have been a much worse energy blowup-rate, which
in turn would have obstructed our proof of linear stability.
In Theorem~\ref{T:L2MILDENERGYBLOWUPCMCGAUGE}, 
we consider only the case of near-FLRW Kasner backgrounds,
though it is possible to derive 
(perhaps very singular) 
energy estimates in the case of a general Kasner background.

\begin{theorem}[\textbf{Mildly singular energy estimates without derivative loss for solutions to the linearized equations}]
\label{T:L2MILDENERGYBLOWUPCMCGAUGE}
Consider a solution to the linear equations of Prop.\ \ref{P:LINEARIZEDCMCEQUATIONS}
corresponding to the data 
$\left(\LinSecondFund(1), \grenormalized(1), \partial_t \SFRenormalized(1),
\partial \SFRenormalized(1) \right)$
(given on $\Sigma_1 = \lbrace 1 \rbrace \times \mathbb{T}^3$),
where $\LapseRenormalized(1)$ is
determined by the elliptic PDEs \eqref{E:LINEARIZEDLAPSE}-\eqref{E:LINEARIZEDLAPSELOWER}.
There exist a small constant $\smallparameter_* > 0$
and constants $C > 0$ and $c > 0$
such that if
$\tracefreeparameter \geq 0$ is sufficiently small (see definition \ref{E:TRACEFREEPARAMETER})
and $\highnorm{0}(1) < \infty$ (see definition \eqref{E:HIGHNORM}),
then the base-level total energy 
$\mathscr{E}_{(Total);\smallparameter_*}(t)$
defined in \eqref{E:TOTALENERGY}
verifies the following inequality\footnote{The explicit numerical 
constants on the right-hand side of \eqref{E:SECONDMODELENERGYGRONWALLREADY} 
are not sharp, but that is not important when $\tracefreeparameter$ is small.} 
for $t \in (0,1]$:
\begin{align} \label{E:SECONDMODELENERGYGRONWALLREADY}
	\mathscr{E}_{(Total);\smallparameter_*}^2(t) 
	& \leq 
			C \mathscr{E}_{(Total);\smallparameter_*}^2(1) 
		\\
	& \ \ \underbrace{- \frac{1}{6} \smallparameter_* 
					\int_{s=t}^1 s^{-1} \int_{\Sigma_s} |s \partial \grenormalized|_{\gKasner}^2 \, dx \, ds}_{
	\mbox{\upshape Past-favorable sign}}
		\underbrace{- \frac{1}{6} \int_{s=t}^1 s^{-1} \int_{\Sigma_s} |s \partial \SFRenormalized|_{\gKasner}^2 \, dx \, ds}_{
		\mbox{\upshape Past-favorable sign}}
		 \notag \\
	& \ \ \underbrace{- \frac{1}{6} \int_{s=t}^1 s^{-1} \int_{\Sigma_s} |s \partial \LapseRenormalized|_{\gKasner}^2 \, dx \, ds}_{
		\mbox{\upshape Past-favorable sign}}
	\underbrace{- \frac{1}{2} \int_{s=t}^1 s^{-1} \int_{\Sigma_s} \LapseRenormalized^2 \, dx \, ds}_{
		\mbox{\upshape Past-favorable sign}}
		\notag \\
	& \ \ + \underbrace{c \tracefreeparameter \int_{s=t}^1 s^{-1} 
				\mathscr{E}_{(Total);\smallparameter_*}^2(s) \, ds}_{\mbox{\upshape 
				Error integral that can create blowup}}. 
		\notag 
\end{align}	

In addition, if $N \geq 0$ is an integer and the solution norm 
$\highnorm{N}(t)$
defined in \eqref{E:HIGHNORM}
verifies $\highnorm{N}(1) < \infty$,
then the up-to-order $N$ energy 
$\mathscr{E}_{(Total);\smallparameter_*;N}(t)$
defined in \eqref{E:TOPORDERENERGY}
verifies the following inequality for $t \in (0,1]$:
\begin{align} \label{E:ORDERNTOTALENERGYGRONWALLED}
	\mathscr{E}_{(Total);\smallparameter_*;N}(t)
	& \leq 
		C	
		\mathscr{E}_{(Total);\smallparameter_*;N}(1)
		t^{- c \tracefreeparameter}.
\end{align}

Furthermore, if $N \geq 0$ is an integer and $\highnorm{N}(1) < \infty$,
then there exist constants $C > 0$ and $c > 0$ 
such that the following inequality holds for $t \in (0,1]$:
\begin{align} \label{E:HIGHNORMBOUND}
	\highnorm{N}(t) 
	& \leq C \highnorm{N}(1) t^{-c \tracefreeparameter}.
\end{align}	

\end{theorem}

\begin{remark}
	Theorem~\ref{T:L2MILDENERGYBLOWUPCMCGAUGE} should be viewed as 
	being relevant for estimating the solution's
	high-order derivatives in the nonlinear problem,
	while Theorem~\ref{T:CMCLINEARSTABILITY} below should be viewed as 
	being relevant for 
	estimating its low-order derivatives; see Sect.\ \ref{S:NONLINEARPROBLEM} for further discussion.
\end{remark}

\begin{remark}
	The proof of Theorem~\ref{T:L2MILDENERGYBLOWUPCMCGAUGE} essentially amounts to combining 
	an intricate collection of integration by parts identities in the right way
	(this step was carried out in Theorem~\ref{T:CMCMONOTONICITYID})
	and absorbing various integrals into favorably signed integrals. 
	Certain aspects of our proof somewhat remind us of
	arguments used in \cite{rB1986}, in which Bartnik gave a new proof of the positive mass theorem
	of Schoen--Yau \cites{rSstY1979,rSstY1981} and Witten \cite{eW1981}.
	His proof was simpler than the previous proofs but was valid 
	only under the assumption that the metric is near-Euclidean
	and required the use of spatial harmonic coordinates. 
	Like our proof, Bartnik's involved expressing the scalar curvature of the Riemannian 
	$3$-metric in terms of Christoffel symbols, integrating with respect to the
	measure corresponding to the Euclidean metric, 
	and absorbing all of the unsigned quadratic terms into favorably signed quadratic terms
	(whose coefficients happened to be sufficiently large).
\end{remark}

\begin{proof}[Proof of Theorem~\ref{T:L2MILDENERGYBLOWUPCMCGAUGE}]
	The prove the theorem, we will use the following pointwise estimates
	for the integrand terms
	$\mathcal{N}_i$, $i=1,2,\cdots,10$
	defined in \eqref{E:CUBICFORM1}-\eqref{E:QUADRATICFORM3}:
	\begin{align}
			|\mathcal{N}_1|
			& \leq 
				2 
				|s \tracefreeSecondFundKasner|_{\gKasner}
				|s \LinSecondFund|_{\gKasner}
				|\LapseRenormalized|
			\leq 
				\tracefreeparameter
				\smallparameter
				|s \LinSecondFund|_{\gKasner}^2
				+
				\frac{\tracefreeparameter}{\smallparameter}
				\LapseRenormalized^2,
					\label{E:N1POINTWISE} \\
			 \mathcal{N}_2
			& \leq 
				\left(\frac{2}{3} + 2 \tracefreeparameter \right) 
				|\partial \SFRenormalized|_{\gKasner}^2,
					\label{E:N2POINTWISE} \\
			|\mathcal{N}_3|
			& \leq  
				|s \partial \SFRenormalized|_{\gKasner}^2
				+ 
				A^2 |s \partial \LapseRenormalized|_{\gKasner}^2
				\leq  
				|s \partial \SFRenormalized|_{\gKasner}^2
				+ 
				\frac{2}{3}|s \partial \LapseRenormalized|_{\gKasner}^2,
					\label{E:N3POINTWISE} \\
			\smallparameter \mathcal{N}_4
			& \leq 
				\left(\frac{1}{6} + \frac{1}{2} \tracefreeparameter \right) 
				\smallparameter
				|s \partial \grenormalized|_{\gKasner}^2,
					\label{E:N4POINTWISE} \\
			\smallparameter |\mathcal{N}_5|
			& \leq 
				C \tracefreeparameter \smallparameter |s \LinSecondFund|_{\gKasner}^2,
					\label{E:N5POINTWISE} \\
			\smallparameter |\mathcal{N}_6|
			& \leq 
				C \tracefreeparameter \smallparameter |s \partial \grenormalized|_{\gKasner}^2,
					\label{E:N6POINTWISE} \\
			\smallparameter |\mathcal{N}_7|
			& \leq 
				\frac{1}{12}
				\smallparameter
				|s \partial \grenormalized|_{\gKasner}^2
				+
				C \smallparameter
				|s \partial \LapseRenormalized|_{\gKasner}^2,
					\label{E:N7POINTWISE} \\
			\smallparameter |\mathcal{N}_8|
			& \leq 
				\tracefreeparameter 
				\smallparameter
				|s \LinSecondFund|_{\gKasner}^2
				+
				C 
				\tracefreeparameter 
				\smallparameter
				\LapseRenormalized^2,
					\label{E:N8POINTWISE} \\		
			\smallparameter |\mathcal{N}_9|
			& \leq 
				C 
				\smallparameter
				|s \partial \SFRenormalized|_{\gKasner}^2
				+
				C 
				\smallparameter
				|s \partial \LapseRenormalized|_{\gKasner}^2,
					\label{E:N9POINTWISE} \\
			\smallparameter |\mathcal{N}_{10}|
			& \leq 
				\frac{1}{12}
				\smallparameter
				|s \partial \grenormalized|_{\gKasner}
				+
				C
				\smallparameter
				|s \partial \SFRenormalized|^2.
					\label{E:N10POINTWISE} 
	\end{align}
	All of the above estimates
	except for \eqref{E:N2POINTWISE}-\eqref{E:N4POINTWISE}
	are straightforward consequences of the Cauchy-Schwarz
	inequality relative to the metric $\gKasner$
	and simple estimates of the form $|ab| \leq (1/2)\delta a^2 + (1/2)\delta^{-1} b^2$,
	for appropriately chosen constants $\delta > 0$.
	To derive \eqref{E:N2POINTWISE}
	and \eqref{E:N4POINTWISE},
	we use the fact that
	the eigenvalues of $t \SecondFundKasner_{\ j}^i$ are 
	$\geq - q_{Max} \geq - \left\lbrace \frac{1}{3} + \tracefreeparameter \right\rbrace$,
	where $q_{Max}$ is defined in \eqref{E:BIGGESTKASNEREXPONENT}.
	To derive the second inequality in \eqref{E:N3POINTWISE},
	we use the simple inequality $A \leq \sqrt{\frac{2}{3}}$.

	We now claim that there exist constants $C > 0$ and $c > 0$ such that
	the following estimate holds
	when $\smallparameter > 0$:
	\begin{align}
		&
		\int_{\Sigma_t} 
			(t \partial_t \SFRenormalized)^2 
			+ 
			|t \partial \SFRenormalized|_{\gKasner}^2 
		\, dx
		+
		\int_{\Sigma_t} 
			|t \partial \LapseRenormalized|_{\gKasner}^2
		\, dx
		+
		\left\lbrace
			1 - A^2 - \frac{\tracefreeparameter}{\smallparameter}
		\right\rbrace
		\int_{\Sigma_t} 
			\LapseRenormalized^2
		\, dx
			\label{E:ALMOSTTHERESECONDMODELENERGYGRONWALLREADY}
				\\
	& \ \
		+
		\smallparameter
		\left\lbrace
			1 - \tracefreeparameter
		\right\rbrace
		\int_{\Sigma_t} 
			|t \LinSecondFund|_{\gKasner}^2 
		\, dx
			+ 
		\frac{1}{4} 
		\smallparameter
		\int_{\Sigma_t}
			|t \partial \grenormalized|_{\gKasner}^2 
		\, dx
			\notag \\
		& 
		\leq
		\int_{\Sigma_1} 
			(\partial_t \SFRenormalized)^2 
			+ 
			|\partial \SFRenormalized|_{\gKasner}^2 
		\, dx
		+
		\int_{\Sigma_1} 
			|\partial \LapseRenormalized|_{\gKasner}^2
		\, dx
		+
		\left\lbrace
			1 - A^2 + \frac{\tracefreeparameter}{\smallparameter}
		\right\rbrace
		\int_{\Sigma_1} 
			\LapseRenormalized^2
		\, dx
			\notag \\
	& \ \
		+
		\smallparameter
		\left\lbrace
			1 + \tracefreeparameter
		\right\rbrace
		\int_{\Sigma_1} 
			|\LinSecondFund|_{\gKasner}^2 
		\, dx
			+ 
		\frac{1}{4} 
		\smallparameter
		\int_{\Sigma_1}
			|\partial \grenormalized|_{\gKasner}^2 
		\, dx
			\notag \\
	& \ \
			- 
			\left\lbrace
				\frac{1}{3} - C \tracefreeparameter - C \smallparameter 
			\right\rbrace
			\int_{s=t}^1
				s^{-1}
				\int_{\Sigma_s}  
					|s \partial \SFRenormalized|_{\gKasner}^2
				\, dx
		 	\, ds
			- 
			\left\lbrace
				\frac{1}{3} 
				-
				C \smallparameter 
			\right\rbrace	
			 \int_{s=t}^1 
				s^{-1}
				\int_{\Sigma_s}
					|s \partial \LapseRenormalized|_{\gKasner}^2
		 		\, dx 
		 	\, ds
				 \notag \\
		& \ \
				- 
					\left\lbrace
						1 
						-
						C \frac{\tracefreeparameter}{\smallparameter}
						-
						C \tracefreeparameter \smallparameter
					\right\rbrace
					\int_{s=t}^1 
					s^{-1}
					\int_{\Sigma_s}
						\LapseRenormalized^2
		 			\, dx 
		 		\, ds
		- 
		\smallparameter
		\left\lbrace
			\frac{1}{6} - C \tracefreeparameter 
		\right\rbrace
		\int_{s=t}^1
			s^{-1}
			\int_{\Sigma_s}
				|s \partial \grenormalized|_{\gKasner}^2
			\, dx
		\, ds
			\notag 
				\\
		& \ \
			+ 
		c
		\tracefreeparameter 
		\smallparameter
		\int_{s=t}^1 
			s^{-1} 
			\int_{\Sigma_s}
				|s \LinSecondFund|_{\gKasner}^2
			\, dx
			\, ds.
		\notag
	\end{align}
	To obtain \eqref{E:ALMOSTTHERESECONDMODELENERGYGRONWALLREADY},
	we simply substitute
	the estimates \eqref{E:N1POINTWISE}-\eqref{E:N10POINTWISE}
	into the approximate monotonicity identity \eqref{E:MONOTONICITYID}
	and keep careful track of the coefficients.
	For example, the 
	coefficient 
	$- \left\lbrace \frac{1}{3} - C \tracefreeparameter - C \smallparameter \right\rbrace$
	found in front
	of the integral
	$
		\int_{s=t}^1
				s^{-1}
				\int_{\Sigma_s}  
					|s \partial \SFRenormalized|_{\gKasner}^2
				\, dx
		 	\, ds
	$
	on the right-hand side of \eqref{E:ALMOSTTHERESECONDMODELENERGYGRONWALLREADY}
	comes from adding the coefficient $-2$ on the third line of \eqref{E:MONOTONICITYID}
	to the coefficient
	$\frac{2}{3} + 2 \tracefreeparameter$
	from \eqref{E:N2POINTWISE},
	the coefficient $1$ from \eqref{E:N3POINTWISE},
	and the coefficients $C \smallparameter$
	from \eqref{E:N9POINTWISE}-\eqref{E:N10POINTWISE}.
	Note that for $i=4,\cdots,10$, the terms $\mathcal{N}_i$ from \eqref{E:MONOTONICITYID}
	are multiplied by $\smallparameter$.
	
	Next, from definition \eqref{E:TOTALENERGY},
	we deduce the following simple bound 
	for the last integral in \eqref{E:ALMOSTTHERESECONDMODELENERGYGRONWALLREADY}:
	\begin{align}  \label{E:BADTERM}
		c
		\tracefreeparameter 
		\smallparameter
		\int_{s=t}^1 
			s^{-1} 
			\int_{\Sigma_s}
				|s \LinSecondFund|_{\gKasner}^2
			\, dx
		 \, ds
		& \leq 
				c \tracefreeparameter 
				\int_{s=t}^1 
					s^{-1} 
					\mathscr{E}_{(Total);\smallparameter}^2(s) 
				\, ds.
	\end{align}
	The desired inequality 
	\eqref{E:SECONDMODELENERGYGRONWALLREADY}
	now follows from definition \eqref{E:TOTALENERGY},
	the identity \eqref{E:TRACEFREEPARAMETER},
	and the estimates \eqref{E:ALMOSTTHERESECONDMODELENERGYGRONWALLREADY} and \eqref{E:BADTERM},
	and from first choosing $\smallparameter := \smallparameter_*$
	to be sufficiently small and then choosing 
	$\tracefreeparameter$ to be sufficiently small in a manner that
	depends on the fixed choice of $\smallparameter_*$.
	We stress that the estimate \eqref{E:BADTERM}
	is precisely what generates the last error integral on the right-hand side of 
	\eqref{E:SECONDMODELENERGYGRONWALLREADY}.
	
	To deduce inequality \eqref{E:ORDERNTOTALENERGYGRONWALLED}, 
	we first use \eqref{E:SECONDMODELENERGYGRONWALLREADY}
	and Gronwall's inequality to deduce
	\begin{align} \label{E:ENERGYGRONWALLED}
	\mathscr{E}_{(Total);\smallparameter_*}^2(t)
	\leq C	
		\mathscr{E}_{(Total);\smallparameter_*}^2(1) t^{- c \tracefreeparameter}.
	\end{align}
	Next, we recall the following fact noted in 
	the proof of Cor.\ \ref{C:MONOTONICITY}: the $\partial_{\vec{I}}$-differentiated linear solution variables
	solve the same equations as the non-differentiated linearized
	solution variables. Thus, the energy of the 
	$\partial_{\vec{I}}$-differentiated linear solution variables
	verifies an analog of the estimate
	\eqref{E:ENERGYGRONWALLED}.
	Summing these estimates 
	for $|\vec{I}| \leq N$ 
	and appealing to definition \eqref{E:TOPORDERENERGY},
	we arrive at \eqref{E:ORDERNTOTALENERGYGRONWALLED}.
	
	Inequality \eqref{E:HIGHNORMBOUND} 
	then follows from \eqref{E:ORDERNTOTALENERGYGRONWALLED}
	and Lemma~\ref{L:ENERGYNORMCOMPARISON}.
\end{proof}

\section{Linear stability for near-FLRW Kasner backgrounds}
\label{S:LINEARSTABILITY}
In this section, we state and prove Theorem~\ref{T:CMCLINEARSTABILITY},
which is our main linear stability result.
The theorem shows \textbf{i)} that for nearly spatially isotropic Kasner backgrounds,
the lower-order derivatives
of the linear solution enjoy improved 
estimates with respect to $t$ (i.e., involving less singular powers of $t$) 
compared to the energy estimates of Theorem~\ref{T:L2MILDENERGYBLOWUPCMCGAUGE}
and \textbf{ii)} that various time-rescaled components of the solution variables
converge as $t \downarrow 0$.
As we outline in Sect.\ \ref{S:NONLINEARPROBLEM}, 
the improved behavior is essential for proving the 
nonlinear stable blow-up results of \cite{iRjS2014b}.
The proof of the theorem is 
essentially based on revisiting the linearized equations 
and treating them as transport equations
\emph{with derivative-losing error terms}
that we control with the energy estimates of Theorem~\ref{T:L2MILDENERGYBLOWUPCMCGAUGE}.
Elliptic estimates for the lapse also play a role.
The main difficulty is finding a suitable order in which to prove the estimates.
In essence, this amounts to finding effective dynamic decoupling.

\begin{theorem} [\textbf{Linear stability}]
\label{T:CMCLINEARSTABILITY}
Assume the hypotheses and conclusions of Theorem~\ref{T:L2MILDENERGYBLOWUPCMCGAUGE}.
Let $N \geq 2$ be an integer
and assume that
$\left\| \grenormalized \right\|_{L_{Frame}^2}(1) < \infty$
and
$\highnorm{N}(1) < \infty$
(see definition \eqref{E:HIGHNORM}).
There exist constants $C > 0$ and $c > 0$ such that
if $\tracefreeparameter > 0$ is sufficiently small (see definition \ref{E:TRACEFREEPARAMETER}),
then the linear solution 
to the equations of Prop.\ \ref{P:LINEARIZEDCMCEQUATIONS}
verifies the following estimates for $t \in (0,1]$:
\begin{subequations}
	\begin{align}
		\left\|
			\partial_t (t \LinSecondFund)
		\right\|_{H_{Frame}^{N-1}}
	& \leq C \highnorm{N}(1) t^{-1/3 - c \tracefreeparameter},
		\label{E:PARTIALTSECONDFUNDHNMINUSONE}
		\\
		\left\|
			t \LinSecondFund
		\right\|_{H_{Frame}^{N-1}}
	& \leq C \highnorm{N}(1),
		\label{E:SECONDFUNDHNMINUSONE} \\
	\left\|
		\grenormalized 
	\right\|_{L_{Frame}^2}
	& \leq C \left\lbrace
							\left\|
								\grenormalized 
							\right\|_{L_{Frame}^2}(1)
							+ 
							\frac{1}{\tracefreeparameter}
							\highnorm{N}(1)
					\right\rbrace
		t^{2/3 - c \tracefreeparameter},
		\label{E:METRICRENORMALIZEDL2} \\
	\left\|
		\partial \grenormalized 
	\right\|_{H_{Frame}^{N-1}}
	& \leq 
				\frac{C}{\tracefreeparameter}
				\highnorm{N}(1)
				t^{2/3 - c \tracefreeparameter},
		\label{E:PARTIALMETRICRENORMALIZEDHNMINUSTWO} \\
	\left\|
		\partial_t (t \partial_t \SFRenormalized - A \LapseRenormalized)
	\right\|_{H^{N-1}}
	& \leq C \highnorm{N}(1)
		t^{- 1/3 - c \tracefreeparameter},
		\label{E:PARTIALTTPARTIALTPHIMINUSLAPSEHNMINUSONE} \\
	\left\|
		t \partial_t \SFRenormalized  
	\right\|_{H^{N-1}}
	& \leq C \highnorm{N}(1),
		\label{E:PARTIALTPHIHNMINUSONE} \\
	\left\|
		\partial \SFRenormalized 
	\right\|_{H_{Frame}^{N-2}}
	& \leq C \highnorm{N}(1)
				\left\lbrace
					1 + |\ln(t)|
				\right\rbrace,
		\label{E:PARTIALPHIHNMINUSTWO} \\
	\left\|
	 	\LapseRenormalized 
	\right\|_{H^N}
	& \leq C \highnorm{N}(1)
		t^{- c \tracefreeparameter},
			\label{E:PARTIALLAPSEHNMINUSONEESTIMATE}	\\
	\left\|
			\LapseRenormalized 
	\right\|_{H^{N-1}}
	& \leq C \highnorm{N}(1)
		t^{2/3 - c \tracefreeparameter},
		  \label{E:LAPSEHNMINUSONEESTIMATE} \\
	\left\|
		\LapseRenormalized 
	\right\|_{H^{N-2}}
	& \leq \frac{C}{\tracefreeparameter} \highnorm{N}(1)
		t^{4/3 - c \tracefreeparameter}.
		\label{E:LAPSEHNMINUSTWOESTIMATE}
\end{align}
\end{subequations}

\noindent{\underline{\textbf{Convergence}}.}
There exist a symmetric type $\binom{0}{2}$ tensorfield
$h_{Regular} \in H_{Frame}^{N-1}(\mathbb{T}^3)$,
a type $\binom{1}{1}$ tensorfield
$K_{Bang} \in H_{Frame}^{N-1}(\mathbb{T}^3)$
verifying $(K_{Bang})_{\ a}^a = 0$,
and a function $\Psi_{Bang} \in H^{N-1}(\mathbb{T}^3)$
such that the following estimates hold\footnote{On the left-hand sides of 
	\eqref{E:METRICCONVERGESQISQJ}-\eqref{E:METRICCONVERGESQIISNOTQJ}, we do not sum over $i$ or $j$.} 
for $t \in (0,1]$:
\begin{subequations}
	\begin{align}
		\left\|
			t^{-2 q_j} \grenormalized_{ij} + 2 \ln(t) (K_{Bang})_{\ j}^i 
			- (h_{Regular})_{ij}
		\right\|_{H^{N-1}}
		& \leq C \highnorm{N}(1) t^{ 2/3 - c \tracefreeparameter},
		&& (\mbox{\upshape if} \ q_i = q_j),
		\label{E:METRICCONVERGESQISQJ}	\\
		\left\|
			t^{-2 q_j} \grenormalized_{ij} + \frac{1}{q_i - q_j} t^{2(q_i - q_j)} (K_{Bang})_{\ j}^i 
			- (h_{Regular})_{ij}
		\right\|_{H^{N-1}}
		& \leq C \highnorm{N}(1) t^{2/3 - c \tracefreeparameter},
		&& (\mbox{\upshape if} \ q_i \neq q_j),
		\label{E:METRICCONVERGESQIISNOTQJ}	\\
		\left\|
			t \LinSecondFund
			- K_{Bang}
		\right\|_{H_{Frame}^{N-1}}
		& \leq C \highnorm{N}(1) t^{2/3 - c \tracefreeparameter},
		\label{E:SECONDFUNDCONVERGENES}	\\
		\left\|
			t \partial_t \SFRenormalized - \Psi_{Bang}
		\right\|_{H^{N-1}}
	& \leq C \highnorm{N}(1)
		t^{2/3 - c \tracefreeparameter},
			\label{E:PARTIALTPHICONVERGES} \\
	  \left\|
			\partial \SFRenormalized -  \ln(t) \partial \Psi_{Bang}
		\right\|_{H_{Frame}^{N-2}}
	& \leq C \highnorm{N}(1),
			\label{E:PARTIALPHICONVERGES}
	\end{align}
\end{subequations}
and
\begin{subequations}
	\begin{align}
		\left\|
			h_{Regular}
			- \grenormalized(1)
		\right\|_{H_{Frame}^{N-1}}
		& \leq C \highnorm{N}(1),
			\label{E:METRICLIMITNEARDATA} \\
		\left\|
			K_{Bang}
			- \LinSecondFund(1)
		\right\|_{H_{Frame}^{N-1}}
		& \leq C \highnorm{N}(1),
			\label{E:SECONDFUNDLIMITNEARDATA} \\
		\left\|
			\Psi_{Bang}
			- \partial_t \SFRenormalized(1)
		\right\|_{H^{N-1}}
	& \leq C \highnorm{N}(1).
			\label{E:PARTIALTPHILIMITNEARDATA} 
	\end{align}
\end{subequations}

In addition, the same estimates hold in the case
$\tracefreeparameter = 0$ with all factors of
$
\displaystyle
\frac{1}{\tracefreeparameter}
$ replaced by $1 + \ln t$.

\end{theorem}

Before proving the theorem, we first make some remarks.

\begin{itemize}
	\item Just below the ``rough'' Theorem~\ref{T:ROUGHVERSIONLINEARSTABILITY}
		(which is a recap of Theorem~\ref{T:CMCLINEARSTABILITY}),
		we gave a detailed explanation of why the 
		convergence results of Theorem~\ref{T:CMCLINEARSTABILITY}
		are natural. Furthermore, we highlighted the connection between
		the convergence results stated in the theorem 
		and the heuristic statements made in
		\cites{vBiK1972,jB1978} concerning the asymptotic behavior of
		solutions to the nonlinear equations near singularities.
	\item The improved behavior in $t$ 
		provided by \eqref{E:PARTIALTSECONDFUNDHNMINUSONE}-\eqref{E:LAPSEHNMINUSTWOESTIMATE}
		is of critical importance in closing the nonlinear problem;
		see Sect.\ \ref{S:NONLINEARPROBLEM}.
\end{itemize}

\begin{proof}[Proof of Theorem~\ref{T:CMCLINEARSTABILITY}]
We give the proof only in the case $\tracefreeparameter > 0$.
The case $\tracefreeparameter = 0$ can be handled by straightforward modifications of the case
$\tracefreeparameter > 0$.
Throughout the proof, we silently use Lemma~\ref{L:KASNERMETRIC},
Lemma~\ref{L:ENERGYNORMCOMPARISON},
and the $t$-weights inherent in Def.\ \ref{D:NORMS}. \\

\noindent{\textbf{Proof of \eqref{E:PARTIALLAPSEHNMINUSONEESTIMATE} and \eqref{E:LAPSEHNMINUSONEESTIMATE}:}}
We commute equation \eqref{E:LINEARIZEDLAPSELOWER} 
with $\partial_{\vec{I}}$, multiply by $\partial_{\vec{I}} \LapseRenormalized$,
and integrate by parts over $\Sigma_t$ to deduce the elliptic estimate
\begin{align} \label{E:PARTIALVECILAPSEFIRSTIMPROVED}
	t \| \partial \partial_{\vec{I}} \LapseRenormalized \|_{L_{\gKasner}^2}
	+ \| \partial_{\vec{I}} \LapseRenormalized \|_{L^2}
	& \leq C t^2 \| \partial_{\vec{I}} \currenormalized \|_{L^2}.
\end{align}
From \eqref{E:SCALARCURVATURERENORMALIZED} and \eqref{E:HIGHNORMBOUND},
we deduce that whenever $|\vec{I}| \leq N-1$, we have
$ \| \partial_{\vec{I}} \currenormalized \|_{L^2} \leq 
C t^{-4/3 - c \tracefreeparameter} \highnorm{N}(1)$.
The estimates \eqref{E:PARTIALLAPSEHNMINUSONEESTIMATE} and \eqref{E:LAPSEHNMINUSONEESTIMATE}
now readily follow. \\

\noindent{\textbf{Proof of \eqref{E:PARTIALTSECONDFUNDHNMINUSONE}:}}
We first deduce from equation \eqref{E:LINEARIZEDKEVOLUTION} that
\begin{align} \label{E:PARTIALTRENORMALIZEDKFIRSTIMPROVED}
	\| \partial_t (t \LinSecondFund) \|_{H_{Frame}^{N-1}} 
		& \leq C t^{1/3 - c \tracefreeparameter} \| \LapseRenormalized \|_{H^{N+1}} 
		+ C t^{-1} \| \LapseRenormalized \|_{H^{N-1}} 
		+ C t \| \Ricrenormalized \|_{H_{Frame}^{N-1}}.
\end{align}
From \eqref{E:RICRENORMALIZED},
\eqref{E:HIGHNORMBOUND},
and \eqref{E:LAPSEHNMINUSONEESTIMATE},
we conclude that the right-hand side of \eqref{E:PARTIALTRENORMALIZEDKFIRSTIMPROVED}
is $\leq C \highnorm{N}(1) t^{-1/3 - c \tracefreeparameter}$ as desired. \\

\noindent{\textbf{Proof of \eqref{E:SECONDFUNDHNMINUSONE},
\eqref{E:SECONDFUNDCONVERGENES},
and
\eqref{E:SECONDFUNDLIMITNEARDATA}:}}
We set
$
f(t)
:=
t \LinSecondFund_{\ j}^i(t,\cdot)
$,
where we are viewing $f$ as a scalar
$H^{N-1}(\mathbb{T}^3)$-valued function of $t$.
From \eqref{E:PARTIALTSECONDFUNDHNMINUSONE},
we deduce that for $0 < s \leq t \leq 1$, 
we have
$
\left\|
	f(t) - f(s)
\right\|_{H^{N-1}}
\leq
C 
\left(
	t^{2/3 - c \tracefreeparameter}
	-
	s^{2/3 - c \tracefreeparameter}
\right)
\highnorm{N}(1)
$.
From this bound and the completeness of $H^{N-1}(\mathbb{T}^3)$,
it follows that if $\tracefreeparameter$ is sufficiently small,
then $\lim_{t \downarrow 0} f(t)$
exists as an element of $H^{N-1}(\mathbb{T}^3)$.
We denote the limit
by $(K_{Bang})_{\ j}^i := f(0)$.
Moreover, the previous estimate yields
$
\left\|
	f(t) - f(0)
\right\|_{H^{N-1}}
\leq
C \highnorm{N}(1) t^{2/3 - c\tracefreeparameter}
$.
The estimates
\eqref{E:SECONDFUNDCONVERGENES}
and
\eqref{E:SECONDFUNDLIMITNEARDATA}
follow from this bound,
while
\eqref{E:SECONDFUNDHNMINUSONE}
follows from
\eqref{E:SECONDFUNDCONVERGENES},
\eqref{E:SECONDFUNDLIMITNEARDATA},
and the bound
$
f(1) \leq \highnorm{N}(1)
$.
\\

\noindent{\textbf{Proof of \eqref{E:METRICRENORMALIZEDL2} and \eqref{E:PARTIALMETRICRENORMALIZEDHNMINUSTWO}:}}
We give the details only for \eqref{E:METRICRENORMALIZEDL2} since the proof of
\eqref{E:PARTIALMETRICRENORMALIZEDHNMINUSTWO} is essentially the same.
To proceed, we first split
$t \SecondFundKasner_{\ j}^a$ into its pure trace and trace-free parts
and use equation \eqref{E:LINEARIZEDGEVOLUTION} to deduce that
\begin{align}
	\partial_t (t^{-2/3} \grenormalized_{ij})
		& = -2 t^{-1} (t^{-2/3} \grenormalized_{ia}) (t \tracefreeSecondFundKasner_{\ j}^a)
			- 2 t^{-5/3} \gKasner_{ia} (t \LinSecondFund_{\ j}^a)
			- 2 t^{-5/3} \gKasner_{ia} (t \SecondFundKasner_{\ j}^a) \LapseRenormalized.
		\label{E:LINEARIZEDGEVOLUTIONREWRITTEN} 
\end{align}
From equation \eqref{E:LINEARIZEDGEVOLUTIONREWRITTEN}, we deduce that
\begin{align} \label{E:LINEARIZEDGEVOLUTIONHNMINSONENORM}
	\| \partial_t (t^{-2/3} \grenormalized) \|_{L_{Frame}^2}
	& \leq C t^{-1} |t \tracefreeSecondFundKasner|_{Frame} \| t^{-2/3} \grenormalized \|_{L_{Frame}^2}
			+ C t^{-5/3} |\gKasner|_{Frame}  \| t \LinSecondFund \|_{L_{Frame}^2}
				\\
	& \ \ + C t^{-5/3} |\gKasner|_{Frame} |t \SecondFundKasner|_{Frame} \| \LapseRenormalized \|_{L^2}.
		\notag
\end{align}
From inequality \eqref{E:HIGHNORMBOUND},
we deduce that the right-hand side of \eqref{E:LINEARIZEDGEVOLUTIONHNMINSONENORM} is 
\[
	\leq c \tracefreeparameter t^{-1} \| t^{-2/3} \grenormalized \|_{L_{Frame}^2} + C \highnorm{N}(1) t^{-1 - c \tracefreeparameter}.
\]
Using this estimate and integrating \eqref{E:LINEARIZEDGEVOLUTIONHNMINSONENORM} in time, we deduce that
\begin{align} \label{E:METRICGRONWALLREADY}
	t^{-2/3} \| \grenormalized \|_{L_{Frame}^2}(t)
	& \leq \| \grenormalized \|_{L_{Frame}^2}(1)
		+	\frac{C}{\tracefreeparameter} \highnorm{N}(1) t^{- c \tracefreeparameter}
		+ c \tracefreeparameter
			\int_{s=t}^1
				s^{-1}
				\left\lbrace
					s^{-2/3} \| \grenormalized \|_{L_{Frame}^2}(s)
				\right\rbrace
			\, ds.
\end{align}
From \eqref{E:METRICGRONWALLREADY} and Gronwall's inequality in the quantity
$t^{-2/3} \| \grenormalized \|_{L_{Frame}^2}(t)$, 
we conclude the desired inequality \eqref{E:METRICRENORMALIZEDL2}. \\

\noindent{\textbf{Proof of \eqref{E:LAPSEHNMINUSTWOESTIMATE}:}}
We need only to revisit the proof of
\eqref{E:LAPSEHNMINUSONEESTIMATE}
and use the fact that the improved estimate \eqref{E:PARTIALMETRICRENORMALIZEDHNMINUSTWO}
allows us to deduce that whenever $|\vec{I}| \leq N-2$, we have
$ 
\displaystyle
\| \partial_{\vec{I}} \currenormalized \|_{L^2} \leq 
\frac{C}{\tracefreeparameter}  \highnorm{N}(1) t^{-2/3 - c \tracefreeparameter}
$. \\

\noindent{\textbf{Proof of \eqref{E:PARTIALTTPARTIALTPHIMINUSLAPSEHNMINUSONE}:}}
We first deduce from equation \eqref{E:SCALARFIELDWAVEDECOMPOSED} that
\begin{align} \label{E:SCALARFIELDWAVEFIRSTEST}
	\| \partial_t (t \partial_t \SFRenormalized - A \LapseRenormalized) \|_{H^{N-1}} 
	& \leq C t \| \gKasner^{ab} \partial_a \partial_b \SFRenormalized \|_{H^{N-1}} 
		+ 
		C t^{-1} \| \LapseRenormalized \|_{H^{N-1}}.
\end{align}
From \eqref{E:HIGHNORMBOUND}
and \eqref{E:LAPSEHNMINUSONEESTIMATE},
we deduce that the right-hand side of \eqref{E:SCALARFIELDWAVEFIRSTEST}
is $\leq$ the right-hand side of \eqref{E:PARTIALTTPARTIALTPHIMINUSLAPSEHNMINUSONE} as
desired. \\

\noindent{\textbf{Proof of \eqref{E:PARTIALTPHIHNMINUSONE},
\eqref{E:PARTIALTPHICONVERGES},
and
\eqref{E:PARTIALTPHILIMITNEARDATA}:}}
The
existence of the limiting tensorfield $\Psi_{Bang}$
and the three estimates under consideration
follow from
inequalities 
\eqref{E:PARTIALTTPARTIALTPHIMINUSLAPSEHNMINUSONE}
and 
\eqref{E:LAPSEHNMINUSONEESTIMATE}
by the same reasoning we used to prove
\eqref{E:SECONDFUNDHNMINUSONE},
\eqref{E:SECONDFUNDCONVERGENES},
and
\eqref{E:SECONDFUNDLIMITNEARDATA}.
\\

\noindent{\textbf{Proof of \eqref{E:PARTIALPHIHNMINUSTWO} and \eqref{E:PARTIALPHICONVERGES}:}}
From \eqref{E:PARTIALTPHICONVERGES},
we deduce
$
\left\|
		\partial_t 
		\left\lbrace
			\partial \SFRenormalized - \ln t \partial \Psi_{Bang}
		\right\rbrace
\right\|_{H_{Frame}^{N-2}}
\leq 
C \highnorm{N}(1)
t^{-1/3 - c \tracefreeparameter}
$.
Integrating from time $t$ to time $1$,
we find that
$
\left\|
		\partial \SFRenormalized - \ln t \partial \Psi_{Bang}
\right\|_{H_{Frame}^{N-2}}
\leq 
C \highnorm{N}(1)
t^{2/3 - c \tracefreeparameter}
+
\left\|
		\partial \SFRenormalized
\right\|_{H_{Frame}^{N-2}}
(1)
\leq 
C \highnorm{N}(1)
$,
which yields
\eqref{E:PARTIALPHICONVERGES}.
\eqref{E:PARTIALPHIHNMINUSTWO} then follows from 
\eqref{E:PARTIALTPHIHNMINUSONE},
\eqref{E:PARTIALPHICONVERGES},
and \eqref{E:PARTIALTPHILIMITNEARDATA}.
\\

\noindent{\textbf{Proof of \eqref{E:METRICCONVERGESQISQJ}, 
\eqref{E:METRICCONVERGESQIISNOTQJ},
and \eqref{E:METRICLIMITNEARDATA}:}}
Throughout this paragraph, we do not use Einstein's summation convention for $i$ or $j$.
Recall that $\gKasner_{ii} = t^{2 q_i}$, that $t (\SecondFundKasner_{\ i}^i) = - q_i$,
and that the off-diagonal components of these tensorfields are $0$.
Multiplying equation \eqref{E:LINEARIZEDGEVOLUTION} by $t^{-2q_j}$,
we deduce the equation
$\partial_t (t^{-2 q_j} \grenormalized_{ij})
	=  - 2 t^{-1 + 2(q_i - q_j)} (t \LinSecondFund_{\ j}^i)
		+ 2 q_i \delta_{ij} t^{-1} \LapseRenormalized.
$
From this equation,
the estimates
\eqref{E:LAPSEHNMINUSONEESTIMATE} and
\eqref{E:SECONDFUNDCONVERGENES},
and the simple estimate
$|q_i - q_j| \leq 2 \tracefreeparameter$
(see \eqref{E:TRACEFREEPARAMETER}), 
we deduce that
for $i,j=1,2,3$,
we have 
\begin{align} \label{E:TIMEDERIVATIVEOFMETRICINEQUALITYSUGGESTSCONVERGENCE}
	\left\|
		\partial_t 
		\left\lbrace
			t^{-2 q_j} \grenormalized_{ij} - 2 \left(\int_{s=t}^1 s^{-1 + 2(q_i - q_j)} \, ds \right)(K_{Bang})_{\ j}^i 
		\right\rbrace
	\right\|_{H^{N-1}}
	& \leq C \highnorm{N}(1) t^{-1/3 - c \tracefreeparameter}.
\end{align}
When $\tracefreeparameter$ is sufficiently small,
the existence of the limiting tensorfield components $(h_{Regular})_{ij}$
and the estimates 
\eqref{E:METRICCONVERGESQISQJ}, 
\eqref{E:METRICCONVERGESQIISNOTQJ},
and \eqref{E:METRICLIMITNEARDATA}
follow from \eqref{E:TIMEDERIVATIVEOFMETRICINEQUALITYSUGGESTSCONVERGENCE}
and \eqref{E:SECONDFUNDLIMITNEARDATA}
by the same reasoning we used to prove
\eqref{E:SECONDFUNDHNMINUSONE},
\eqref{E:SECONDFUNDCONVERGENES},
and
\eqref{E:SECONDFUNDLIMITNEARDATA}.
\end{proof}

\section{Summary of the proof of the nonlinear stability of the FLRW Big Bang singularity}
\label{S:NONLINEARPROBLEM}
Recall that the FLRW metric is $\gfour_{FLRW} = - dt^2 + t^{2/3} \sum_{i=1}^3 (dx^i)^2$
(see \eqref{E:FLRWBACKGROUNDSOLUTION}). 
In this section, we outline the proof of Theorem~\ref{T:STABILITYOFFLRW},
which yields the nonlinear stability of FLRW solution's Big Bang
singularity. Our discussion will provide a detailed overview of the central role
that the monotonicity identities 
and linear stability results play 
in the nonlinear problem.
For complete details in the context of the Einstein-stiff fluid system,
we refer the reader to \cite{iRjS2014b}.

\begin{remark}
	\label{R:TIMERESCALEDVARS}
	In \cite{iRjS2014b}, we formulated the equations and estimates
	in terms of time-rescaled solution variables.
	Here, to keep the discussion short, we do not introduce
	such time-rescaled variables.
	This changes the appearance of various 
	equations and estimates
	compared to \cite{iRjS2014b},
	but not their content.
\end{remark}

\subsection{Norms and energies}
\label{SS:NONLINEARNORMSANDENERGIES}
We start by introducing the norms and energies that we use to control the nonlinear solution.
In our analysis, we view the unknowns to be the solution variables
$n,g_{ij},k_{\ j}^i,\phi$ appearing in the nonlinear equations of Prop.\ \ref{P:EINSTEINSFCMC}. 
Note that the pure trace part of
$k_{\ j}^i$ is controlled by the CMC condition 
$k_{\ a}^a = - t^{-1}$ 
and thus we only need to derive estimates for its trace-free part $\hat{k}_{\ j}^i$.

\begin{definition}[\textbf{The pointwise norm} $|\cdot|_g$]
	\label{D:DYNAMICGNORM}
	Throughout this section, we use the pointwise norm $|\cdot|_g$,
	which is defined by replacing the background Kasner metric
	$\gKasner$ with the metric $g$ on both sides of
	\eqref{E:NEWGNORM}.
\end{definition}

\begin{definition}[\textbf{Solution norms}] 
\label{D:NONLINEARFRAMENORM}
To control the nonlinear solution, we rely 
on norms\footnote{More precisely, in 
\cite{iRjS2014b}, our high-order solution norms 
do not directly control
$\left\|
	g - g_{FLRW}
\right\|_{L_{Frame}^2}$
or
$
\left\|
	g^{-1} - g_{FLRW}^{-1}
\right\|_{L_{Frame}^2}
$,
but that detail is not important for the ensuing discussion.} 
belonging to the following family:
\begin{align}
	\highnorm{M}(t)
	& := 
			\left\| t  \hat{\SecondFund} \right\|_{H_{Frame}^M} 
			+ 
			\| \partial g \|_{H_{Frame}^M} 
			\label{E:NONLINEARHIGHNORM}  \\
	& \ \
			+
			t^{-2/3}
			\left\|
				g - g_{FLRW}
			\right\|_{H_{Frame}^M}
			+
			t^{2/3}
			\left\|
				g^{-1} - g_{FLRW}^{-1}
			\right\|_{H_{Frame}^M}
			\notag
			\\
	& \ \	
			+ 
			\left\| t \partial_t \phi \right\|_{H_{Frame}^M} 
			+ 
			t^{2/3} \| \partial \phi \|_{H_{Frame}^M} 
			+ 
			\sum_{p=0}^2 t^{(2/3)p} \left\| n - 1 \right\|_{H^{M+p}}.
		\notag 
\end{align}

\end{definition}

To control the norms \eqref{E:NONLINEARHIGHNORM},
we will use the energies provided by the next definition.
The energies for the nonlinear solution are tied to 
approximate monotonicity identities 
for the nonlinear solution
in the same way that
the energies of Def.\ \ref{D:ENERGIES}
for the linear solution are
tied to the approximate monotonicity identity 
of Theorem~\ref{T:CMCMONOTONICITYID}.

\begin{definition}[\textbf{Energies}]
\label{D:NONLINEARENERGIES}
Let $\vec{I}$ is a spatial derivative multi-index (as defined in Subsect.\ \ref{SS:COORDINATES}),
let $M \geq 0$ be an integer,
and let $\smallparameter > 0$ be a constant.
We define the energies
$
\mathscr{E}_{(Metric);\vec{I}}(t) \geq 0,
$
$\cdots$,
$
\mathscr{E}_{(Total);\smallparameter;M}(t) \geq 0 
$
as follows:
\begin{subequations}
\begin{align}
	\mathcal{E}_{(Metric);\vec{I}}^2(t)
	&:=
		\int_{\Sigma_t} 
				\left|
					t \partial_{\vec{I}} \hat{\SecondFund}
				\right|_g^2 
				+ 
				\frac{1}{4} 
				\left|
					t \partial \partial_{\vec{I}} g
				\right|_g^2 
			\, dx,
			\label{E:NONLINEARMETRICENERGY}
				\\
	\mathscr{E}_{(Scalar);\vec{I}}^2(t)
		& := \int_{\Sigma_t} 
					(t \partial_t \partial_{\vec{I}} \phi)^2 
					+ 
					n^2 |t \partial \partial_{\vec{I}} \phi|_g^2 
				 \, dx, 
			\label{E:NONLINEARSCALARFIELDENERGY} \\
	\mathscr{E}_{(\partial Lapse);\vec{I}}^2(t)
		& := \int_{\Sigma_t} 
						|t \partial \partial_{\vec{I}} n|_g^2
					\, dx, 
			\label{E:NONLINEARPARTIALLAPSEENERGY} \\
	\mathscr{E}_{(Lapse)}^2(t)
		& := \int_{\Sigma_t} 
						|\partial_{\vec{I}} (n-1)|^2
				 \, dx, 
			\label{E:NONLINEARLAPSEENERGY} 
			\\
		\mathscr{E}_{(Total);\smallparameter;\vec{I}}^2(t) 
		& := \mathscr{E}_{(Scalar);\vec{I}}^2(t)
					+ 
					\mathscr{E}_{(\partial Lapse);\vec{I}}^2(t)
					+	  
					\frac{1}{3} \mathscr{E}_{(Lapse);\vec{I}}^2(t)
					+
					\smallparameter \mathscr{E}_{(Metric);\vec{I}}^2(t),
			\label{E:NONILINEARTOTALENERGY}
				\\
	 \mathscr{E}_{(Total);\smallparameter;M}^2(t) 
		& := \sum_{|\vec{I}| \leq M} \mathscr{E}_{(Scalar);\vec{I}}^2(t).
			\label{E:SUMMEDNONILINEARTOTALENERGY}
\end{align}
\end{subequations}
We clarify that on the right-hand side of \eqref{E:NONLINEARMETRICENERGY},
$\left|
	t \partial \partial_{\vec{I}} g
\right|_g^2 
=
t^2
g^{ab} g^{ij} g^{ef}
\partial_e \partial_{\vec{I}} g_{ai}
\partial_f \partial_{\vec{I}} g_{bj}
$.
\end{definition}

\begin{remark}
	\label{R:ENERGIESDONOTCONTROLMETRIC}
	Note that our energies do not directly control the terms
	$
	t^{2/3}
	\left\|
		g - g_{FLRW}
	\right\|_{H_{Frame}^M}
	$
	or
	$
	t^{-2/3}
	\left\|
		g^{-1} - g_{FLRW}^{-1}
	\right\|_{H_{Frame}^M}
	$,
	which are featured in the norm \eqref{E:NONLINEARHIGHNORM}.
	Therefore, control of these terms 
	does not directly follow from the energy estimates described below
	and instead requires a separate argument
	based on the metric evolution equation \eqref{E:PARTIALTGCMC};
	we will avoid further discussion of this issue here.
\end{remark}

As in our proof of Theorem~\ref{T:L2MILDENERGYBLOWUPCMCGAUGE}, 
in order to close the energy estimates for the nonlinear solutions,
we have to make a suitable choice of $\smallparameter > 0$.
Here we note that the same choice of 
\begin{align} \label{E:SMALLPARAMETERCHOSEN}
	\smallparameter := \smallparameter_*
\end{align}
that we made in the proof of Theorem~\ref{T:L2MILDENERGYBLOWUPCMCGAUGE}
is also sufficient in our study of nonlinear solutions.
The reason is that $\smallparameter$
needs to be adapted only to handle 
various integrals in the approximate monotonicity identity 
that are generated by \emph{linear} terms in the equations;
quadratically small nonlinear terms
do not affect the viability of the choice $\smallparameter := \smallparameter_*$,
but rather generate error integrals that
we explain how to control in Subsect.\ \ref{SS:NONLINEARERRORINTEGRALS}.

\subsection{The nonlinear stability of the FLRW Big Bang singularity}
\label{SS:STATEMENTOFFLROWSATBILITY}
In this subsection, 
we state our main nonlinear result, namely Theorem~\ref{T:STABILITYOFFLRW}.
In the rest of Sect.\ \ref{S:NONLINEARPROBLEM},
we will explain the main ideas behind the proof of the theorem. 
We refer the reader to \cite{iRjS2014b} for complete details in the
case of the Einstein-stiff fluid system.
We note that in Remarks~\ref{R:ALTERNATEAPPROACH} and \ref{R:TIMERESCALEDVARS},
we pointed out some minor differences
between the approach outlined here and the approach taken in
\cite{iRjS2014b}.

\begin{theorem}[\textbf{Stable Big Bang Formation for near-FLRW solutions}]
	\label{T:STABILITYOFFLRW}
	Consider initial data 
	(as described in Subsect.\ \ref{SS:INITIALVALUEANDCMC})
	for the Einstein-scalar field system
	given on the Cauchy hypersurface
	$\Sigma_1 = \lbrace 1 \rbrace \times \mathbb{T}^3$
	that verify the CMC condition
	$\SecondFundupdatadarg{a}{a} = - 1$,
	where $\SecondFundupdatadarg{i}{j} := (\SecondFund_{\ j}^i)|_{\Sigma_1}$
	is the mixed second fundamental form of $\Sigma_1$.
	Assume that $N \geq 8$ and that\footnote{Note that $\highnorm{N}(1)$ is assumed to be quadratically small compared
	to the amplitude $\varepsilon$ featured in the estimate \eqref{E:NORMBOUND}.
	This assumption is non-optimal and could be improved with further effort;
	we have aimed for a clean presentation rather than
	for optimizing powers of $\varepsilon$.
	\label{FN:NONOPTIMALSMALLNESS}} 
	$\highnorm{N}(1) \leq \varepsilon^2$,
	where $\highnorm{N}$ is defined by
	\eqref{E:NONLINEARHIGHNORM} (see also Remark~\ref{R:NOTGEOMETRICSMALLNESS}).
	There exist constants $C > 0$ and $c > 0$ such that if $\varepsilon$ is sufficiently small,
	then the perturbed solution
	to the Einstein-scalar field system
	in CMC-transported-spatial-coordinates gauge
	(that is, to equations \eqref{E:HAMILTONIAN}-\eqref{E:LAPSE}) 
	exists for $(t,x) \in (0,1] \times \mathbb{T}^3$
	and verifies the norm bound
	\begin{align} \label{E:NORMBOUND}
		\highnorm{N}(t)
		& \leq \frac{1}{2} \varepsilon t^{- c \sqrt{\varepsilon}}.
	\end{align}
	Moreover, the Kretschmann scalar verifies the pointwise bound
	\begin{align} \label{E:NONLINEARKRETSCHMANNBLOWUP}
		\left|
			t^4 \Riemfour^{\alpha \beta \gamma \delta} \Riemfour_{\alpha \beta \gamma \delta}(t,x)
			-
			\frac{20}{27}
		\right|
		& \leq C \varepsilon.
	\end{align}
	In particular, 
	$\Riemfour^{\alpha \beta \gamma \delta} \Riemfour_{\alpha \beta \gamma \delta}$
	blows up like $t^{-4}$ as $t \downarrow 0$.
	Moreover, there exists a type $\binom{1}{1}$ tensorfield
	$K_{Bang} \in H^{N-1}(\mathbb{T}^3)$ such that the following convergence results 
	for components holds for $t \in (0,1]$,
	$(i,j=1,2,3)$:
	\begin{subequations}
	\begin{align}
		\left\|
			n-1
		\right\|_{H^{N-2}}
		& \leq C \varepsilon t^{4/3- c \sqrt{\varepsilon}},
				\label{E:LAPSECONVERGENCES} \\
		\left\|
			t k_{\ j}^i
			- 
			(K_{Bang})_{\ j}^i
		\right\|_{H^{N-1}}
		& \leq C \varepsilon t^{2/3- c \sqrt{\varepsilon}},
		\label{E:NONLINEARSECONDFUNDCONVERGENES}
			\\
		\left\|
			(K_{Bang})_{\ j}^i
			+
			\frac{1}{3} \ID_{\ j}^i
		\right\|_{H^{N-1}}
		& \leq C \varepsilon, 
		\label{E:LIMITINGSECFUNDFORMNEARID}
	\end{align}
	\end{subequations}
	where $\ID_{\ j}^i := \mbox{\upshape diag}(1,1,1)$ is the identity transformation.
	Similar convergence results hold for other solution variables,
	in analogy with the convergence results 
	for the linear solution proved in Theorem~\ref{T:CMCLINEARSTABILITY}; 
	see \cite{iRjS2014b} for precise statements
	in the context of the Einstein-stiff fluid system.
\end{theorem}

\begin{remark}
	\label{R:NOTGEOMETRICSMALLNESS}
	In Theorem~\ref{T:STABILITYOFFLRW},
	we formulated our near-FLRW data assumption
	as a smallness condition on the norm
	$\highnorm{N}(1)$.
	We could have instead formulated
	a ``more geometric'' near-FLRW assumption
	by making assumptions only on the ``geometric data''
	from Subsect.\ \ref{SS:INITIALVALUEANDCMC},
	which does not include the lapse.
	We could have then derived
	the smallness of $\highnorm{N}(1)$ as a consequence of the assumptions 
	on the geometric data 
	(essentially by deriving elliptic estimates for the lapse along $\Sigma_1$
	by using equation \eqref{E:LAPSE});
	for convenience, we have avoided doing this.
\end{remark}

\begin{remark}
	As stated, Theorem~\ref{T:STABILITYOFFLRW} applies only to data with 
	constant mean curvature. However, this restriction is not necessary:
	in \cite{iRjS2014b}, we show that
	for perturbations of the FLRW solution,
	it is always possible to find
	a CMC hypersurface $\Sigma_1'$ near $\Sigma_1$.
	One can then use CMC-transported coordinates gauge
	starting from the ``data'' induced on $\Sigma_1'$.
	Alternatively, one could employ the 
	parabolic lapse gauges
	described in Sect.\ \ref{S:PARABOLICMONOTONICITY}
	starting from near-FLRW data on $\Sigma_1$; 
	these gauges do not require the initial Cauchy hypersurface
	to have constant mean curvature.
\end{remark}

\subsection{Outline of the proof}
\label{SS:NONLINEARBLOWUPOUTLINE}
We now outline the main steps in proof of Theorem~\ref{T:STABILITYOFFLRW}.
In the remainder of Sect.\ \ref{S:NONLINEARPROBLEM},
we will provide additional details about the most important aspects
of the proof. 

\begin{enumerate}
\item \textbf{(Big picture)}
	The main step in the proof of the theorem
	is to derive the a priori estimate \eqref{E:NORMBOUND}
	for the ``high-norm'' $\highnorm{N}(t)$, which shows in particular
	that it remains finite
	for $t \in (0,1]$, even though it can blow up
	as $t \downarrow 0$.
	It then follows as a standard result for 
	elliptic-hyperbolic systems (see \cite{lAvM2003})
	that, as a consequence of the a priori norm estimate, 
	the solution must exist
	for $(t,x) \in (0,1] \times \mathbb{T}^3$.
\item \textbf{(High-norm bootstrap assumption)}
	We use a bootstrap argument to obtain the desired estimates for
	the norm $\highnorm{N}$. To this end, we let 
	$(T,1]$ be any time interval on which the solution
	exists, where $0 < T < 1$. 
	We make a bootstrap
	assumption for $\highnorm{N}(t)$ for $t \in (T,1]$; see Subsect.\ \ref{SS:NORMBOOTSTRAP}.
	The bootstrap assumption
	weakly captures the fact
	that the perturbed solution 
	is near-FLRW.
	By the remarks made in Step (1), to prove the existence result of
	Theorem~\ref{T:STABILITYOFFLRW}, it suffices
	to derive the a priori estimate \eqref{E:NORMBOUND} for
	$\highnorm{N}(t)$ for $t \in (T,1]$,
	which is a strict improvement of the bootstrap assumption; 
	by a standard continuity argument,
	this justifies the bootstrap assumption,
	shows that the solution
	exists for $(t,x) \in (0,1] \times \mathbb{T}^3$,
	and shows that in fact, the norm estimate \eqref{E:NORMBOUND} holds for $t \in (0,1]$.
	To obtain the desired a priori norm estimate, 
	we will derive energy estimates
	via a nonlinear analog of
	Theorem~\ref{T:L2MILDENERGYBLOWUPCMCGAUGE}, that is,
	a result showing that appropriately defined nonlinear energies 
	can blow up at most in a very mild fashion as $t \downarrow 0$.
	We carry this out in Step (7). 
	The intermediate steps stated below are mostly in service of Step (7).
\item \textbf{(``Strong'' low-norm bootstrap assumptions)}	
	We make stronger bootstrap assumptions at the low-order derivative levels
	for $t \in (T,1]$,
	that is, bootstrap assumptions that
	involve less singular behavior in $t$ than what is afforded by the bootstrap 
	assumptions for the high-norm $\highnorm{N}(t)$.
	These stronger bootstrap assumptions are 
	key ingredients for controlling error terms 
	in the energy estimates, 
	for exhibiting the AVTD nature of the solution
	(that is, that the spatial derivative terms in the equations are negligible near the singularity),
	and for proving the convergence results such as
	\eqref{E:LAPSECONVERGENCES}-\eqref{E:LIMITINGSECFUNDFORMNEARID}.
	Although we do not explicitly state such stronger bootstrap assumptions in this paper,
	we note that they are essentially nonlinear analogs of the estimates that we 
	proved in the linear stability results of
	Theorem~\ref{T:CMCLINEARSTABILITY}.
	The existence of a $T \in (0,1)$ such that the solution exists and verifies the 
	high-norm and low-norm
	bootstrap assumptions on $(T,1]$
	follows from standard local well-posedness for elliptic-hyperbolic
	systems; see \cite{lAvM2003}.
	\item \textbf{(Improvements of the low-norm bootstrap assumptions)}
		As an intermediate step,
		we derive improvements of the 
		strong low-norm bootstrap assumptions 
		from the previous step,
		thereby closing this portion of the bootstrap argument.
		By improvements, we mean estimates that
		are strictly stronger than the estimates afforded by the low-norm bootstrap assumptions.
		This step is tantamount to justifying the AVTD nature of the solution.
		We state several of the resulting estimates in
		Subsect.\ \ref{SS:NONLINEARIMPROVEDESTIMATESATLOWERDERIVATIVELEVELS}.
		In this paper, we do not provide details behind this step since
		the desired estimates can be obtained by
		using arguments similar to the ones that we used
		in proving the linear stability results of Theorem~\ref{T:CMCLINEARSTABILITY},
		but with the added complication that one must control the nonlinear error terms. 
		We will, however, explain how to bound some representative 
		nonlinear error terms that arise in the energy estimates;
		see Steps (6)-(7).
		As in the proof of Theorem~\ref{T:CMCLINEARSTABILITY},
		the proofs in this step incur a loss of derivatives.
	\item \textbf{(Approximate monotonicity identity)}
		To obtain the desired energy estimates,
		the key starting point is an approximate monotonicity identity,
		that is, a nonlinear analog of Theorem~\ref{T:CMCMONOTONICITYID};
		recall that for linear solutions,
		the approximate monotonicity identity provided by Theorem~\ref{T:CMCMONOTONICITYID}
		is the main ingredient that we use to derive the mildly singular
		energy estimates of Theorem~\ref{T:L2MILDENERGYBLOWUPCMCGAUGE}.
		In this article, we do not derive an approximate monotonicity identity 
		for the nonlinear equations because the derivation would 
		be very similar to the proof of
		Theorem~\ref{T:CMCMONOTONICITYID}
		but would be rather lengthy due to the presence of 
		many nonlinear error integrals. 
		It turns out that these nonlinear error integrals
		have only a small effect on the dynamics in the sense
		that their presence is compatible with 
		the proof of a mild blowup-rate for the nonlinear energies,
		similar to the (at most) mild blowup of the linear solution's energies
		guaranteed by Theorem~\ref{T:L2MILDENERGYBLOWUPCMCGAUGE}.
		In the next two steps, we highlight some key representative
		nonlinear error integrals and 
		overview how we can handle them.
	\item \textbf{(Bounds for nonlinear error integrals)} 
	In Subsects.\ \ref{SS:REPRESENTATIVENONLINEARERRORTERMS}
	and \ref{SS:NONLINEARERRORINTEGRALS},
	we highlight three representative nonlinear error integrals,
	which appear in the approximate monotonicity identity described in
	the previous step,
	and bound the error integrals in terms of the energies.
	The improved estimates at the low derivative levels
	from Step (4) are crucial for this.
	\item  \textbf{(A priori energy estimates)}
		Recall that by using the approximate monotonicity identity for linear solutions,
		we were able to show that they
		verify the estimate \eqref{E:SECONDMODELENERGYGRONWALLREADY},
		which is the integral inequality for linear solutions' energies
		that we used to establish the mild energy blowup rate 
		\eqref{E:ORDERNTOTALENERGYGRONWALLED}.
		In the present nonlinear context, an analog of
		the integral inequality \eqref{E:SECONDMODELENERGYGRONWALLREADY}
		also holds, but the right-hand side features
		all of the nonlinear error integrals
		generated by the previous two steps.
		In Subsect.\ \ref{SS:MAINAPRIORIENERGYNONLINEAR},
		we outline the derivation of the nonlinear energy integral 
		inequalities that result
		from accounting for the nonlinear error integrals.
		We then use Gronwall's inequality to
		obtain the desired a priori energy estimates
		on the bootstrap interval $(T,1]$
		and sketch a proof of how the energy
		estimates allow one to derive strict improvements
		of the bootstrap assumption for the norm $\highnorm{N}(t)$ made in Step (2).
		In particular, this yields the desired a priori estimate \eqref{E:NORMBOUND}.
	\item \textbf{(Additional information)}
		Having derived improvements of both the low-norm and high-norm bootstrap assumptions,
		to complete the proof of the theorem,
		we need only to derive
		the curvature blowup result
		\eqref{E:NONLINEARKRETSCHMANNBLOWUP}
		and convergence results such as
		\eqref{E:LAPSECONVERGENCES}-\eqref{E:LIMITINGSECFUNDFORMNEARID}.
		We omit these details since
		they can be essentially be proved as part of Step (4), 
		that is, by using derivative-losing arguments similar to the ones we
		gave in the proof of the linear stability results of
		Theorem~\ref{T:CMCLINEARSTABILITY}.
	\end{enumerate}

\subsection{Bootstrap assumptions}
\label{SS:NORMBOOTSTRAP}
Let $T \in (0,1)$ be a ``bootstrap time'' such that the solution classically exists on
$(T,1] \times \mathbb{T}^3$ and obeys the following
bootstrap assumption, where the norm $\highnorm{N}(t)$ is defined in \eqref{E:NONLINEARHIGHNORM}:
\begin{align} \label{E:NORMBOOT}
	\highnorm{N}(t)
	& \leq \varepsilon t^{- \upsigma},
	&& t \in (T,1].
\end{align}
In \eqref{E:NORMBOOT}, $\varepsilon$ and $\upsigma$ are two small positive
bootstrap parameters that are constrained in particular by
$0 < \sqrt{\varepsilon} \leq \upsigma < 1$.
We will adjust the allowable smallness of $\varepsilon$ and $\upsigma$ 
throughout the course of the analysis. In particular, we will later impose
a condition of the form $c \sqrt{\varepsilon} < \upsigma$ 
for a large constant $c$
(see just below inequality \eqref{E:MAINAPRIORINORMESTIMATE}).
One can think of $\upsigma$ 
as a rough bound for the maximum possible size of $t \hat{\SecondFund}$, 
in analogy with the role that the parameter 
$\tracefreeparameter$ 
played in driving the energy blowup rates 
of Theorem~\ref{T:L2MILDENERGYBLOWUPCMCGAUGE}
(recall that $\tracefreeparameter$  
is equal to $t$ times the norm of the trace-free part of the background Kasner metric's 
second fundamental form).
Note that our smallness assumption for $\upsigma$ is reasonable 
in the sense that
$t \hat{\SecondFund}$  
is small for perturbations of the FLRW metric.
Note also that \eqref{E:NORMBOOT} allows 
for the possibility that
$\highnorm{N}(t)$ blows up as $t \downarrow 0$,
consistent with the estimates for the linear solutions
that we derived in Theorem~\ref{T:L2MILDENERGYBLOWUPCMCGAUGE}.
In Cor.\ \ref{C:MAINAPRIORIENERGY},
we sketch a proof that 
for near-FLRW data, the following bound holds:
\begin{align} \label{E:IMPROVEDNORMBOOT}
	\highnorm{N}(t)
	& \leq 
		\frac{1}{2} \varepsilon t^{- c \sqrt{\varepsilon}},
	&& t \in (T,1],
\end{align}
which is a strict improvement of the bootstrap assumption
\eqref{E:NORMBOOT} for $\varepsilon$ sufficiently small.
Deriving \eqref{E:IMPROVEDNORMBOOT} is the main
technical step in the proof of Theorem~\ref{T:STABILITYOFFLRW}.

\begin{remark}
	Recall that in Theorem~\ref{T:STABILITYOFFLRW}, we are assuming that $N \geq 8$.
	In \cite{iRjS2014b}, we show that this is sufficient 
	to allow for closure of all nonlinear estimates at all orders.
\end{remark}

\subsection{Statement of the main a priori energy and norm estimates}
\label{SS:MAINAPRIORIENERGYNONLINEAR}
In Prop.\ \ref{P:NONLINEARINTEGRALINEQUALITIESFORENERGIES},
we state the integral inequalities
verified by the energies.
In Cor.\ \ref{C:MAINAPRIORIENERGY},
we use use the integral inequalities 
to derive a Gronwall estimate for the energies,
which leads to the improvement
\eqref{E:IMPROVEDNORMBOOT}
of the norm bootstrap assumption
and completes the main step in the proof of Theorem~\ref{T:STABILITYOFFLRW}.
Following this, we devote the rest of Sect.\ \ref{S:NONLINEARPROBLEM}
to sketching the main ideas behind the proof 
of Prop.\ \ref{P:NONLINEARINTEGRALINEQUALITIESFORENERGIES}.

\begin{proposition}[\textbf{Integral inequalities verified by the energies}]
\label{P:NONLINEARINTEGRALINEQUALITIESFORENERGIES}
There exist constants $C > 0$ and $c > 0$ such that
if the bootstrap assumption \eqref{E:NORMBOOT}
holds for $t \in (T,1]$ and if $\varepsilon$ and $\upsigma$ are sufficiently small,
then the following analog of the linear energy inequality 
\eqref{E:SECONDMODELENERGYGRONWALLREADY}
holds for $0 \leq M \leq N$ and $t \in (T,1]$:
\begin{align} \label{E:NONLINEARENERGYINEQUALITY}
	\mathcal{E}_{(Total);\smallparameter_*;M}^2(t)
	& \leq \mathcal{E}_{(Total);\smallparameter_*;M}^2(1)
		\\
	& \ \
		+ \underbrace{
			c \varepsilon
			\int_{s=t}^1
				s^{-1} \mathcal{E}_{(Total);\smallparameter_*;M}^2(s)
			\, ds
			}_{\mbox{\upshape Borderline term}}
			\notag \\
	& \ \
			+ \underbrace{
			C
			\int_{s=t}^1
				s^{-1/3 - c \sqrt{\varepsilon}} \mathcal{E}_{(Total);\smallparameter_*;M}^2(s)
			\, ds
			}_{\mbox{\upshape Non-borderline term}}
				\notag
				\\
	& \ \
		+
		\mbox{\upshape Similar error integrals not treated here}
		\notag
		\\
	& \ \
		+
		\mbox{\upshape Negative definite spacetime integrals, similar to those in
		\eqref{E:SECONDMODELENERGYGRONWALLREADY}}.
		\notag
\end{align}
\end{proposition}

\begin{corollary}[\textbf{Main a priori energy estimates}]
\label{C:MAINAPRIORIENERGY}
Assume that $N \geq 8$.
Consider the energy $\mathcal{E}_{(Total);\smallparameter_*;N}(t)$
defined in \eqref{E:SUMMEDNONILINEARTOTALENERGY} 
and the norm $\highnorm{N}(t)$
defined in \eqref{E:NONLINEARHIGHNORM}.
Assume that
$\highnorm{N}(1) \leq \varepsilon^2$
(see Footnote~\ref{FN:NONOPTIMALSMALLNESS} regarding this assumption).
There exists a constant $C > 0$ such that 
under the bootstrap assumption \eqref{E:NORMBOOT},
the following a priori estimate holds for 
$t \in (T,1]$ whenever $\varepsilon$ and $\upsigma$ are sufficiently small:
\begin{align} \label{E:MAINAPRIORIENERGYESTIMATE}
	\mathcal{E}_{(Total);\smallparameter_*;N}(t)
	& 
	\leq
	C \varepsilon^2
	t^{- c \sqrt{\varepsilon}}.
\end{align}
Moreover, the following estimate holds
for $t \in (T,1]$:
\begin{align} \label{E:MAINAPRIORINORMESTIMATE}
	\highnorm{N}(t)
	& \leq \frac{1}{2} \varepsilon t^{- c \sqrt{\varepsilon}},
\end{align}
which is an improvement of the bootstrap assumption \eqref{E:NORMBOOT} whenever $c \sqrt{\varepsilon} < \upsigma$.

\end{corollary}	

\begin{proof}[Discussion of the proof]
	We refer readers to \cite{iRjS2014b}*{Section~13} 
	for the complete details of the proof. Here we only sketch the main ideas.
	First, we note that
	it is straightforward to establish comparison estimates
	in the spirit of Lemma~\ref{L:ENERGYNORMCOMPARISON}.
	The comparison estimates show in particular that \eqref{E:MAINAPRIORINORMESTIMATE}
	follows from combining the energy estimate
	\eqref{E:MAINAPRIORIENERGYESTIMATE}
	with estimates for the terms
	$
	t^{2/3}
	\left\|
		g - g_{FLRW}
	\right\|_{H_{Frame}^N}
	$
	and
	$
	t^{-2/3}
	\left\|
		g^{-1} - g_{FLRW}^{-1}
	\right\|_{H_{Frame}^N}
	$,
	which are featured in the nonlinear norm \eqref{E:NONLINEARHIGHNORM}
	but which we do not discuss here.\footnote{Note that we did not include such terms in the norms \eqref{E:HIGHNORM}
	for linear solutions.}
	These additional
	terms are also the reason that the amplitude
	on the right-hand side of \eqref{E:MAINAPRIORINORMESTIMATE}
	is $\mathcal{O}(\varepsilon)$ rather than $\mathcal{O}(\varepsilon^2)$.
	The proofs of the comparison estimates
	rely on estimates for the coordinate components $g_{ij}$ and $g^{ij}$.
	Specifically, they rely on the estimates
	\eqref{E:GANDGINVERSEESTIMATES}
	stated below, which do not follow 
	directly from the bootstrap assumption \eqref{E:NORMBOOT}
	and thus require an independent proof
	(along the lines of the proof of the estimate \eqref{E:METRICRENORMALIZEDL2} for 
	linear solutions).
	
	To derive the energy estimates stated in \eqref{E:MAINAPRIORIENERGYESTIMATE},
	one can use inequality \eqref{E:NONLINEARENERGYINEQUALITY}
	to establish the following estimates by 
	a straightforward argument based on Gronwall's inequality and
	induction in $M$ for $0 \leq M \leq N$:
	\begin{align} \label{E:INDUCTIONMAINAPRIORIENERGYESTIMATE}
	\mathcal{E}_{(Total);\smallparameter_*;M}^2(t)
	& 
	\leq
	C \varepsilon^4
	t^{- c \sqrt{\varepsilon}},
	&&
	t \in (T,1].
\end{align}
We now further comment on two aspects of the estimate 
\eqref{E:INDUCTIONMAINAPRIORIENERGYESTIMATE}.
First, it is only the first ``Borderline'' term on the right-hand side
of \eqref{E:NONLINEARENERGYINEQUALITY} that can cause
$\mathcal{E}_{(Total);\smallparameter_*;M}(t)$
to blow up as $t \downarrow 0$;
the ``Non-borderline'' term on the right-hand side
of \eqref{E:NONLINEARENERGYINEQUALITY}
is harmless in the sense that
the function $s^{-1/3 - c \sqrt{\varepsilon}}$ is integrable
over the interval $s \in (0,1]$ 
whenever $\varepsilon$ is sufficiently small.
Second, we note that the exponent  
on the right-hand side of \eqref{E:INDUCTIONMAINAPRIORIENERGYESTIMATE}
is $t^{- c \sqrt{\varepsilon}}$ due to  
some terms that we have not discussed here, that is,
the terms ``$\mbox{\upshape Similar error integrals not treated here}$''
from \eqref{E:NONLINEARENERGYINEQUALITY}; 
if not for the omitted terms, 
the exponent could be improved to $t^{- c \varepsilon}$
(this is a minor remark that has no substantial bearing on the main results).
\end{proof}

\subsection{A convenient frame and dual frame}
\label{SS:NONLINEARFRAME}
In the ensuing discussion,
we will find it convenient to perform some computations 
relative to the frame\footnote{In
\eqref{E:ALMOSTORTHFRAMEANDCOFRAME} and the remainder of Sect.\ \ref{S:NONLINEARPROBLEM},
$\partial_A := \frac{\partial}{\partial x^A}$, with $\lbrace x^A \rbrace_{A=1,2,3}$ denoting the transported spatial coordinates.
Moreover, $\partial_{\vec{I}}$ is still the coordinate partial derivative multi-indexed operator defined in Subsect.\ \ref{SS:COORDINATES}.}
$\lbrace e_{(A)}' \rbrace_{A=1}^3$
and dual frame $\lbrace \theta^{'(A)} \rbrace_{A=1}^3$,
whose elements are defined as follows:
\begin{align} \label{E:ALMOSTORTHFRAMEANDCOFRAME}
	e_{(A)}' & := t^{-1/3} \partial_A, & \theta^{'(A)} & := t^{1/3} dx^A.
\end{align}
The appeal of the frame $\lbrace e_{(A)}' \rbrace_{A=1}^3$ is that it is orthonormal as measured by the background spatial metric $g_{FLRW} := t^{2/3} \sum_{i=1}^3 (dx^i)^2$, and, as we explain
in Subsect.\ \ref{SS:NONLINEARIMPROVEDESTIMATESATLOWERDERIVATIVELEVELS}, 
it is approximately orthonormal 
for the perturbed metric $g$
(in a sense that we make precise 
via the estimate \eqref{E:FRAMEESTIMATESGANDGINVERSE}).
The perturbed metric and its inverse can respectively be expanded\footnote{Throughout this subsection
and the next one, we use Einstein's summation convention for uppercase Latin indices.}
relative to the dual frame and frame as follows:
\begin{align}
	g & = g_{AB} \theta^{'(A)} \otimes \theta^{'(B)},
	& & g^{-1} = g^{AB} e_{(A)}' \otimes e_{(B)}',
\end{align}
where $g_{AB} := g(e_{(A)}',e_{(B)}')$ and $g^{AB} := g^{-1}(\theta^{'(A)},\theta^{'(B)})$.
We remark that in \cite{iRjS2014b}, 
instead of working with the ``time-rescaled'' frame and dual frame 
\eqref{E:ALMOSTORTHFRAMEANDCOFRAME}, we instead work with solution
variables that are rescaled with respect to powers of $t$; see Remark~\ref{R:TIMERESCALEDVARS}.

The connection coefficients $\gamma_{A C B}$
of the frame relative to $g$ are 
determined by the equation\footnote{Recall that $\nabla$ denotes the Levi--Civita connection of $g$.}
\begin{align}
	\nabla_{e_{(A)}'} e_{(B)}'
	& = g^{CD} \gamma_{A D B} e_{(C)}',
\end{align}
where, since the vectorfield commutators $[e_{(A)}',e_{(B)}']$ vanish, we have\footnote{We are using here the
standard notation 
$X (f)$ to denote the derivative of the scalar function $f$
in the direction of the vectorfield $X$.}
\begin{align} \label{E:LITTLEGAMMALOWER}
	\gamma_{A C B}
	& = \frac{1}{2}
		\left\lbrace
			e_{(A)}'(g_{CB})
			+ e_{(B)}'(g_{AC})
			- e_{(C)}'(g_{AB})
		\right\rbrace.
\end{align}

For use below, we note the following standard expression for the Ricci curvature of $g$
(in type $\binom{1}{1}$ form):
\[
\Ric = \Ric_{\ B}^A e_{(A)}' \otimes \theta^{'(B)},
\]
where
\begin{align} \label{E:RICCIFRAMEEXPANSION}
	\Ric_{\ B}^A 
	& =
		g^{AE} g^{CF} e_{(C)}' (\gamma_{E F B})
	- g^{AE} g^{CF} e_{(E)}' (\gamma_{B C F})
		\\
  & \ \ + g^{AE} g^{CF} g^{DH} \gamma_{C F D} \gamma_{E H B}
		- g^{AE} g^{CF} g^{DH} \gamma_{E F D} \gamma_{C H B}.
		\notag
\end{align}

\subsection{Improved AVTD-type estimates at the lower derivative levels}
\label{SS:NONLINEARIMPROVEDESTIMATESATLOWERDERIVATIVELEVELS}
As we mentioned in Steps (3) and (4) of Subsect.\ \ref{SS:NONLINEARBLOWUPOUTLINE},
to prove Prop.\ \ref{P:NONLINEARINTEGRALINEQUALITIESFORENERGIES},
we need to derive improved estimates at the lower derivative levels.
By improved, we mean that they are less singular
as $t \downarrow 0$
compared to the estimates afforded by the 
bootstrap assumption \eqref{E:NORMBOOT}.
In the context of the linear problem, we derived such estimates
in Theorem~\ref{T:CMCLINEARSTABILITY}.
For brevity, we will take for granted 
here that we can derive similar estimates
for the nonlinear solution, effectively postponing
the discussion of the nonlinear error terms until
Subsect.\ \ref{SS:NONLINEARERRORINTEGRALS},
when we discuss them in the context of energy estimates.
Specifically,
we will 
take for granted
that the following pointwise coordinate component estimates hold for
$t \in (T,1]$ whenever $|\vec{I}| \leq N - 3$ in \eqref{E:GANDGINVERSEESTIMATES}-\eqref{E:PARTIALTPHIESTIMATES},
$1 \leq |\vec{I}| \leq N - 3$ in \eqref{E:SPATIALDERIVATIVESOFPHIESTIMATES},
and $i,j=1,2,3$:
\begin{subequations}
\begin{align} \label{E:GANDGINVERSEESTIMATES}
	\left|
		\partial_{\vec{I}}
		\left\lbrace
			g_{ij} - (g_{FLRW})_{ij}
		\right\rbrace
	\right|
	& \lesssim \sqrt{\varepsilon} t^{2/3 - c \sqrt{\varepsilon}},
	&& 
	\left|
		\partial_{\vec{I}}
		\left\lbrace
			g^{ij} - (g_{FLRW}^{-1})^{ij}
		\right\rbrace
	\right|
	\lesssim \sqrt{\varepsilon} t^{-2/3 - c \sqrt{\varepsilon}},
		\\
	\left|
		\partial_{\vec{I}}
		\left\lbrace
			t \partial_t \phi
			-
			\sqrt{\frac{2}{3}}
		\right\rbrace
	\right|
	& \lesssim \varepsilon,
		\label{E:PARTIALTPHIESTIMATES} \\
	\left|
		\partial_{\vec{I}} \phi
	\right|
	& \lesssim 
			\sqrt{\varepsilon} t^{- c \sqrt{\varepsilon}}.
	\label{E:SPATIALDERIVATIVESOFPHIESTIMATES}
\end{align}
\end{subequations}
Note, for example, 
that \eqref{E:PARTIALTPHIESTIMATES}
is an improvement over the bootstrap assumption \eqref{E:NORMBOOT}
in that \eqref{E:NORMBOOT} and Sobolev embedding
would yield only the bound
$
\left|
		\partial_{\vec{I}}
		\left\lbrace
			t \partial_t \phi
			-
			\sqrt{\frac{2}{3}}
		\right\rbrace
	\right|
	\lesssim \varepsilon t^{- c \upsigma}
$,
which, due to the singular behavior of the right-hand side as $t \downarrow 0$, 
is inadequate for treating the 
borderline integral
that we control in \eqref{E:ANOTHERIFRSTBOUNDSCALARFIELDENERGYESTIMATENONLINEARERRORINTEGRAL}.
Similarly, for $\varepsilon$ sufficiently small, the
factors of
$t^{2/3 - c \sqrt{\varepsilon}}$ and $t^{-2/3 - c \sqrt{\varepsilon}}$
in \eqref{E:GANDGINVERSEESTIMATES}
are improvements\footnote{Note that the amplitude factors of $\sqrt{\varepsilon}$
in \eqref{E:GANDGINVERSEESTIMATES} are worse than the amplitude factor of
$\varepsilon$ that would follow from \eqref{E:NORMBOOT} and Sobolev embedding.
This is an artifact of some inefficiencies in our proof 
and is not important for our main results; the $t$-dependent factors of
$t^{2/3 - c \sqrt{\varepsilon}}$ and $t^{-2/3 - c \sqrt{\varepsilon}}$ in 
\eqref{E:GANDGINVERSEESTIMATES} are what matters.} 
over the factors of
$t^{2/3 - c \upsigma}$ and $t^{-2/3 - c \upsigma}$
that would follow from
\eqref{E:NORMBOOT} and Sobolev embedding.
We refer readers to \cite{iRjS2014b}
for proofs of analogs of
\eqref{E:GANDGINVERSEESTIMATES}-\eqref{E:SPATIALDERIVATIVESOFPHIESTIMATES}
in the context of the Einstein-stiff fluid system.
The estimates stated in \eqref{E:GANDGINVERSEESTIMATES}
are analogs\footnote{Note that the estimates stated in 
\eqref{E:GANDGINVERSEESTIMATES}-\eqref{E:SPATIALDERIVATIVESOFPHIESTIMATES}
are of pointwise type while the estimates of Theorem~\ref{T:CMCLINEARSTABILITY}
are in terms of Sobolev norms. This is a minor point in the sense that
we can obtain pointwise estimates from Sobolev estimates via Sobolev embedding
(at the cost of a few derivatives). } 
of the estimates 
\eqref{E:METRICRENORMALIZEDL2} and \eqref{E:PARTIALMETRICRENORMALIZEDHNMINUSTWO}
from the linear problem 
while the estimates
\eqref{E:PARTIALTPHIESTIMATES} and
\eqref{E:SPATIALDERIVATIVESOFPHIESTIMATES}
are respectively analogs of
\eqref{E:PARTIALTPHIHNMINUSONE}  
and \eqref{E:PARTIALPHIHNMINUSTWO}.

Contracting inequalities \eqref{E:GANDGINVERSEESTIMATES} against the frame/dual frame, we find that
they are approximately orthonormal 
relative to the metric $g$ in the following weak sense
(for $t \in (T,1]$ and $|\vec{I}| \leq N - 3$):
\begin{align} \label{E:FRAMEESTIMATESGANDGINVERSE}
	\left|
		\partial_{\vec{I}}
		\left\lbrace
			g_{AB} - \delta_{AB}
		\right\rbrace
	\right|
	& \lesssim t^{- c \sqrt{\varepsilon}},
	&& 
	\left|
		\partial_{\vec{I}}
		\left\lbrace
			g^{AB} - \delta^{AB}
		\right\rbrace
	\right|
	\lesssim t^{- c \sqrt{\varepsilon}},
\end{align}
where $\delta_{AB}$ and $\delta^{AB}$
are standard Kronecker deltas.

In Subsect.\ \ref{SS:NONLINEARERRORINTEGRALS}, when
we bound some representative energy error integrals,
we will use the following simple consequences of the
above estimates:
\begin{align} \label{E:LOWGAMMTERMLINFTYBOUND}
	\left\|
		|\gamma|_g
	\right\|_{L^{\infty}}(t)
	& \lesssim t^{-1/3 - c \sqrt{\varepsilon}},
		\\
	\left\|
		|\partial \phi|_g
	\right\|_{L^{\infty}}(t)
	& \lesssim t^{-1/3 - c \sqrt{\varepsilon}}.
	\label{E:LOWSCALARFIELDLINFTYBOUND}
\end{align}
To prove \eqref{E:LOWGAMMTERMLINFTYBOUND},
we first note that \eqref{E:GANDGINVERSEESTIMATES} with $|\vec{I}|=1$ implies that
\begin{align} \label{E:LITTLEGAMALOWERFRAMECOMPONENTESTIMATE}
	\left|
		\gamma_{A C E}
	\right|
	& \lesssim t^{-1/3 - c \sqrt{\varepsilon}},
\end{align}
where in deriving \eqref{E:LITTLEGAMALOWERFRAMECOMPONENTESTIMATE},
we have incurred three factors of 
$t^{-1/3}$ relative to the estimate
\eqref{E:GANDGINVERSEESTIMATES},
one for each contraction against a frame vector belonging to $\lbrace e_{(A)}' \rbrace_{A=1}^3$.
We therefore deduce from 
\eqref{E:FRAMEESTIMATESGANDGINVERSE} and
\eqref{E:LITTLEGAMALOWERFRAMECOMPONENTESTIMATE} that
\begin{align}
	\left|
		\gamma
	\right|_g^2
	& = g^{AB} g^{CD} (g^{-1})^{EF} \gamma_{A C E}  \gamma_{B D F}
	\lesssim t^{-2/3-c \sqrt{\varepsilon}},
\end{align}
which yields \eqref{E:LOWGAMMTERMLINFTYBOUND}.
To obtain \eqref{E:LOWSCALARFIELDLINFTYBOUND},
we note that
\eqref{E:SPATIALDERIVATIVESOFPHIESTIMATES}
with $|\vec{I}|=1$ implies that
\begin{align} \label{E:SCALARFIELDFRAMECOMPONENTESTIMATE}
	\left|
		e_{(A)}' \phi
	\right|
	& \lesssim t^{-1/3 - c \sqrt{\varepsilon}},
\end{align}
where in deriving \eqref{E:SCALARFIELDFRAMECOMPONENTESTIMATE},
we have incurred a factor of 
$t^{-1/3}$ relative to the estimate
\eqref{E:SPATIALDERIVATIVESOFPHIESTIMATES}
due to the contraction against the frame vector 
belonging to $\lbrace e_{(A)}' \rbrace_{A=1}^3$.
We therefore deduce from 
\eqref{E:FRAMEESTIMATESGANDGINVERSE} and
\eqref{E:SCALARFIELDFRAMECOMPONENTESTIMATE} that
\begin{align}
	\left|
		\partial \phi
	\right|_g^2
	& = g^{AB} (e_{(A)}' \phi) e_{(B)}' \phi
	\lesssim t^{-2/3-c \sqrt{\varepsilon}},
\end{align}
which yields \eqref{E:LOWSCALARFIELDLINFTYBOUND}.

For future use, we also note the following
relations, which follow in a straightforward fashion from the definitions
of the quantities involved:
\begin{align} \label{E:GAMMAGDERIVATIVESRELATIONSHIP}
	\left|
			g^{AB} g^{CD} g^{EF} 
			(s \partial_{\vec{I}} \gamma_{ACE}) 
			(s \partial_{\vec{I}} \gamma_{BDF})
	\right|
	& 
	=
	\left|
		s \partial_{\vec{I}} \gamma
	\right|_g^2
	\leq
	C
	\left|
		s \partial \partial_{\vec{I}} g
	\right|_g^2.
\end{align}

\subsection{Identifying some representative nonlinear error terms}
\label{SS:REPRESENTATIVENONLINEARERRORTERMS}
In Subsect.\ \ref{SS:NONLINEARERRORINTEGRALS}, 
we will bound three representative nonlinear error integrals
and explain how they contribute to the
terms on the right-hand side of the energy integral inequality
\eqref{E:NONLINEARENERGYINEQUALITY}.
In the present subsection, 
as a preliminary step,
we commute some of the
nonlinear Einstein-scalar field
equations with the spatial derivative
operator $\partial_{\vec{I}}$ 
(as defined in Subsect.\ \ref{SS:COORDINATES})
and identify
the representative nonlinear terms that 
lead to the error integrals.

First, we commute the evolution equation \eqref{E:PARTIALTKCMC}
with $\partial_{\vec{I}}$. Using \eqref{E:RICCIFRAMEEXPANSION},
we see that relative to the frame/dual frame \eqref{E:ALMOSTORTHFRAMEANDCOFRAME},
the commuted equation takes the form
\begin{align} \label{E:COMMUTEDKEVOLUTIONRELATIVETOFRAME}
	\partial_t (t \partial_{\vec{I}} \SecondFund_{\ B}^A)
	& =
		g^{AE} g^{CF} g^{DH} \gamma_{C F D} \partial_{\vec{I}} \gamma_{E H B}
  	+ \cdots,
\end{align}
where, for illustration, 
we have kept only one representative nonlinear product generated by the right-hand side of \eqref{E:RICCIFRAMEEXPANSION}. 

Similarly, we commute 
the scalar field wave equation \eqref{E:SECONDEXPANDEDSCALARFIELDCMC}
with $\partial_{\vec{I}}$ (as defined in Subsect.\ \ref{SS:COORDINATES})
and, for illustration, 
retain two products generated by terms on the last 
two lines of \eqref{E:SECONDEXPANDEDSCALARFIELDCMC},
which yields
\begin{align}
	\partial_t(t \partial_t \partial_{\vec{I}} \phi) 
	+
	n^2 t g^{ab} \partial_a \partial_b \partial_{\vec{I}} \phi
	& = 
	 \left(
		t \partial_t \phi - \sqrt{\frac{2}{3}}
	 \right)	
	 \frac{(\partial_{\vec{I}} n)}{t} 
	- 
	t g^{ab} (\partial_a \partial_{\vec{I}} n) \partial_b \phi
	+
	\cdots.
			\label{E:COMMUTEDSCALARFIELDEVOLUTIONRELATIVETOFRAME}
\end{align}
Note that in writing down
\eqref{E:COMMUTEDKEVOLUTIONRELATIVETOFRAME}-\eqref{E:COMMUTEDSCALARFIELDEVOLUTIONRELATIVETOFRAME},
we have ignored various \emph{linearly} small products in the equations.
Those terms are of crucial importance
for deriving an analog of the approximate monotonicity
identity from Theorem~\ref{T:CMCMONOTONICITYID}
and for this reason, 
they are not part of the nonlinear error term analysis that
we are currently conducting.

\subsection{Bounds for some representative nonlinear error integrals and a proof sketch of 
Prop.\ \ref{P:NONLINEARINTEGRALINEQUALITIESFORENERGIES}}
\label{SS:NONLINEARERRORINTEGRALS}
Recall that in Theorem~\ref{T:CMCMONOTONICITYID},
we derived an approximate monotonicity identity
for linear solutions,
which was the main step in deriving
the energy integral inequality for linear solutions
stated in \eqref{E:SECONDMODELENERGYGRONWALLREADY}.
In the nonlinear problem,
the analog of inequality \eqref{E:SECONDMODELENERGYGRONWALLREADY}
is the energy integral inequality \eqref{E:NONLINEARENERGYINEQUALITY}
provided by Prop.\ \ref{P:NONLINEARINTEGRALINEQUALITIESFORENERGIES}.
The main difference between the 
linear estimate
\eqref{E:SECONDMODELENERGYGRONWALLREADY}
and the nonlinear estimate
\eqref{E:NONLINEARENERGYINEQUALITY}
is, of course, the presence of
nonlinear error integrals, which
arise in the nonlinear analog of
the approximate monotonicity identity.
Ultimately, the nonlinear error integrals
generate terms that appear on the right-hand side of
the nonlinear energy integral inequality
\eqref{E:NONLINEARENERGYINEQUALITY}.
In this subsection, 
to keep the discussion short, 
we consider only three representative error integrals
generated by the quadratic nonlinear terms highlighted 
in Subsect.\ \ref{SS:REPRESENTATIVENONLINEARERRORTERMS}. 
Our main goal is to show that the corresponding error integrals
(which are cubically\footnote{Some of the error integrals that we treat here
are similar to other error integrals that are generated by
integration by parts. For example, cubic error integrals similar to the one in \eqref{E:ANOTHERCUBICENERGYTERM}
arise from the nonlinear analog of \eqref{E:FIRSTRICCIPRODUCT}.}  
small)
are bounded by the right-hand side of
\eqref{E:NONLINEARENERGYINEQUALITY}.
In view of the above remarks,
it follows that the discussion in this subsection constitutes
a proof sketch of Prop.\ \ref{P:NONLINEARINTEGRALINEQUALITIESFORENERGIES}.
We note that the improved estimates 
at the lower derivative levels 
from Subsect.\ \ref{SS:NONLINEARIMPROVEDESTIMATESATLOWERDERIVATIVELEVELS}
are essential for controlling the error integrals,
especially the borderline one 
that we control in \eqref{E:ANOTHERIFRSTBOUNDSCALARFIELDENERGYESTIMATENONLINEARERRORINTEGRAL}.

\subsubsection{A non-borderline error integral involving the scalar field}
\label{SSS:NONBORDERLINESCALARFIELD}
We start by explaining how the error term
$t g^{ab} (\partial_a \partial_{\vec{I}} n) \partial_b \phi$
on the right-hand side of 
\eqref{E:COMMUTEDSCALARFIELDEVOLUTIONRELATIVETOFRAME}
contributes to the right-hand side of \eqref{E:NONLINEARENERGYINEQUALITY}.
Revisiting the \emph{proof} of Prop.\ \ref{P:ENERGYESTIMATELAPSEANDSCALARFIELD},
we see that in the analog of the integral identity 
\eqref{E:SECONDENERGYESTIMATEMODELSCALARFIELD}, the error term
generates the following spacetime integral
(where we are assuming that $1 \leq |\vec{I}| \leq M \leq N$):
\begin{align} \label{E:SCALARFIELDENERGYESTIMATENONLINEARERRORINTEGRAL}
	\int_{s=t}^1
		\int_{\Sigma_s}
			s g^{ab} (\partial_a \partial_{\vec{I}} n) (\partial_b \phi)
			(s \partial_t \partial_{\vec{I}} \phi)
		\, dx
	\, ds.
\end{align}
Using \eqref{E:LOWSCALARFIELDLINFTYBOUND},
Def.\ \ref{D:NONLINEARENERGIES},
and Cauchy-Schwarz relative to $g$,
we bound the magnitude of 
the integral in \eqref{E:SCALARFIELDENERGYESTIMATENONLINEARERRORINTEGRAL} 
as follows:
\begin{align} \label{E:FIRSTBOUNDSCALARFIELDENERGYESTIMATENONLINEARERRORINTEGRAL}
	&
	\leq
	\int_{s=t}^1
			\left\|
				|\partial \phi|_g
			\right\|_{L^{\infty}}(s)
		\int_{\Sigma_s}
			|s \partial \partial_{\vec{I}} n|_g
			|s \partial_t \partial_{\vec{I}} \phi|
		\, dx
	\, ds
		\\
	& 
	\lesssim
	\int_{s=t}^1
			s^{-1/3 - c \sqrt{\varepsilon}}
		\mathscr{E}_{(Total);\smallparameter_*;\vec{I}}^2(s)
	\, ds.
	\notag
\end{align}
We now simply observe that the right-hand side of \eqref{E:FIRSTBOUNDSCALARFIELDENERGYESTIMATENONLINEARERRORINTEGRAL}
is bounded by the non-borderline error integral
on the right-hand side of \eqref{E:NONLINEARENERGYINEQUALITY},
as desired.

\subsubsection{A borderline error integral involving the scalar field}
\label{SSS:BORDERLINESCALARFIELD}
We now explain how the error term
$\left(
			t \partial_t \phi - \sqrt{\frac{2}{3}}
		\right)	
		\frac{(\partial_{\vec{I}} n)}{t}$
on the right-hand side of 
\eqref{E:COMMUTEDSCALARFIELDEVOLUTIONRELATIVETOFRAME}
contributes to the right-hand side of \eqref{E:NONLINEARENERGYINEQUALITY}.
For the same reasons given in Subsubsect.\ \ref{SSS:NONBORDERLINESCALARFIELD},
this error term generates the error integral
\begin{align} \label{E:ANOTHERSCALARFIELDENERGYESTIMATENONLINEARERRORINTEGRAL}
	\int_{s=t}^1
		s^{-1}
		\int_{\Sigma_s}
			\left(
				s \partial_t \phi - \sqrt{\frac{2}{3}}
			\right)	
			(\partial_{\vec{I}} n)
			(s \partial_t \partial_{\vec{I}} \phi)
		\, dx
	\, ds.
\end{align}
Using \eqref{E:PARTIALTPHIESTIMATES}
and Def.\ \ref{D:NONLINEARENERGIES},
we bound the magnitude of 
the integral in \eqref{E:ANOTHERSCALARFIELDENERGYESTIMATENONLINEARERRORINTEGRAL} 
as follows
(where we are again assuming that $1 \leq |\vec{I}| \leq M \leq N$):
\begin{align} \label{E:ANOTHERIFRSTBOUNDSCALARFIELDENERGYESTIMATENONLINEARERRORINTEGRAL}
	&
	\leq
	\int_{s=t}^1
			s^{-1}
			\left\|
				s \partial_t \phi - \sqrt{\frac{2}{3}}
			\right\|_{L^{\infty}}(s)
		\int_{\Sigma_s}
			|\partial_{\vec{I}} n|
			|s \partial_t \partial_{\vec{I}} \phi|
		\, dx
	\, ds
		\\
	& 
	\leq
	c \varepsilon
	\int_{s=t}^1
		s^{-1}
		\mathscr{E}_{(Total);\smallparameter_*;\vec{I}}^2(s)
	\, ds.
	\notag
\end{align}
Note that the right-hand
side of \eqref{E:FIRSTBOUNDSCALARFIELDENERGYESTIMATENONLINEARERRORINTEGRAL}
is bounded by the borderline error integral
on the right-hand side of \eqref{E:NONLINEARENERGYINEQUALITY},
as desired. We stress that the availability of the small coefficient $\varepsilon$
is crucial since, in the Gronwall estimate for 
$\mathscr{E}_{(Total);\smallparameter_*;\vec{I}}^2$,
the right-hand side of \eqref{E:ANOTHERIFRSTBOUNDSCALARFIELDENERGYESTIMATENONLINEARERRORINTEGRAL}
can cause $\mathscr{E}_{(Total);\smallparameter_*;M}^2(t)$ to blowup 
like $t^{-c \varepsilon}$ as $t \downarrow 0$.
Note also that for this argument, it is crucial that
the singular integrand factor 
on the right-hand side of
\eqref{E:ANOTHERIFRSTBOUNDSCALARFIELDENERGYESTIMATENONLINEARERRORINTEGRAL}
is not worse than $s^{-1}$; a slightly worse factor of type $s^{-1 - C \varepsilon}$
would radically alter the Gronwall estimate and would 
prevent us from deriving an improvement of the norm bootstrap assumption. For this reason,
the ``lossless'' AVTD-type estimate \eqref{E:PARTIALTPHIESTIMATES}
is critically important for the proof of nonlinear stability.

\subsubsection{A non-borderline error integral involving the metric}
\label{SSS:NONBORDERLINEMETRIC}
Finally, we will consider the effect of the error term
$g^{AE} g^{CF} g^{DH} 
	\gamma_{C F D} \partial_{\vec{I}} \gamma_{E H B}$
on the right-hand side of \eqref{E:COMMUTEDKEVOLUTIONRELATIVETOFRAME}.
Revisiting the proof of Prop.\ \ref{PCMC:LINEARIZEDMETRICENERGYESTIMATE},
we see that in the analog of the metric energy identity 
\eqref{E:METRICENERGYID}, 
the error term generates the following spacetime integral:
\begin{align} \label{E:ANOTHERCUBICENERGYTERM}
	\int_{s=t}^1
		\int_{\Sigma_s}
			g^{AB} g^{CF} g^{DH} 
			\gamma_{C F D} (s \partial_{\vec{I}} \gamma_{E H B}) (s \partial_{\vec{I}} \hat{\SecondFund}_{\ A}^E)
		\, dx
		\, ds.
\end{align}
Using
\eqref{E:LOWGAMMTERMLINFTYBOUND},
\eqref{E:GAMMAGDERIVATIVESRELATIONSHIP},
Def.\ \ref{D:NONLINEARENERGIES},
and Cauchy-Schwarz relative to $g$,
we bound the magnitude of
the integral in \eqref{E:ANOTHERCUBICENERGYTERM}
as follows
(where we are again assuming that $1 \leq |\vec{I}| \leq M \leq N$):
\begin{align} \label{E:BOUNDFORANOTHERCUBICENERGYTERM}
	& \lesssim
		\int_{s=t}^1
			\int_{\Sigma_s}
			\left|
				\gamma
			\right|_g
			\left|
				s \partial_{\vec{I}} \gamma
			\right|_g	
			\left|
				s \partial_{\vec{I}} \hat{\SecondFund}
			\right|_g
			\, dx
		\, ds
		\\
	& \lesssim 
		\int_{s=t}^1
			\left\|
				|\gamma|_g
			\right\|_{L^{\infty}}(s)
			\mathcal{E}_{(Total);\smallparameter_*;\vec{I}}^2(s)
		\, ds
		\notag
			\\
	& 
	\lesssim
	\int_{s=t}^1
			s^{-1/3 - c \sqrt{\varepsilon}}
		\mathscr{E}_{(Total);\smallparameter_*;\vec{I}}^2(s)
	\, ds.
	\notag
\end{align}
Like the right-hand side of \eqref{E:FIRSTBOUNDSCALARFIELDENERGYESTIMATENONLINEARERRORINTEGRAL},
the right-hand side of \eqref{E:BOUNDFORANOTHERCUBICENERGYTERM}
is bounded by the non-borderline error integral
on the right-hand side of \eqref{E:NONLINEARENERGYINEQUALITY},
as desired.
This completes our discussion of the three
representative nonlinear error integrals and
finishes our proof sketch of
Prop.\ \ref{P:NONLINEARINTEGRALINEQUALITIESFORENERGIES}.

\section{Comments on realizing ``end states"}
\label{S:ENDSTATESCOMMENTS}
The linear stability results of Theorem~\ref{T:CMCLINEARSTABILITY} 
show that for some time-rescaled versions of the linear solution variables,
there is a well-defined map from their ``initial state'' along the data hypersurface $\Sigma_1$ to their
``end state'' along $\Sigma_0$. For example, the estimate \eqref{E:PARTIALTPHICONVERGES} exhibits this fact
for $t \partial_t \SFRenormalized$, in which case the end state is $\Psi_{Bang}$
and the map is from $H^N$ to $H^{N-1}$.
It is natural to inquire whether or not one can realize a given end state
(more precisely, one in which time derivative terms in the equations are dominant)
by finding suitable initial data that lead to it. 
Although we do not give a proof that one can 
``realize all end states in which time derivative terms dominate''
in solutions to the linearized equations of Prop.\ \ref{P:LINEARIZEDCMCEQUATIONS},
we do point to some evidence in this direction by discussing some relevant results
in a simplified context. Our discussion here is closely connected to the 
work described in Subsect.\ \ref{SS:PREVIOUSWORK} in which authors used Fuchsian methods
to construct singular solutions to various Einstein-matter systems under 
symmetry or analyticity assumptions.
In this section, we consider a model equation in $1+1$ dimensions, obtained from the linearized scalar field equation
\eqref{E:SCALARFIELDWAVEDECOMPOSED}
in the case $\gKasner = g_{FLRW} = t^{2/3} \sum_{i=1}^3 (dx^i)^2$
by dropping the linearized lapse terms and making the symmetry assumption that the solution depends
only on $t$ and a single spatial variable $x^1 \in \mathbb{T}$.
We have made the symmetry assumption only to shorten the presentation; 
the arguments we sketch below remain valid without it.
For convenience, in the rest of this section, we will write $x$ instead of $x^1$.
We caution that ignoring the lapse and its elliptic PDE 
is tantamount to sidestepping new difficulties
not found in the standard Fuchsian framework,
which applies to hyperbolic equations.
Specifically, our model equation in $\SFRenormalized = \SFRenormalized(t,x)$ 
on the domain $(t,x) \in (0,1] \times \mathbb{T}$ is
\begin{align} \label{E:SCALARFIELDWAVESIMPLIFIED}
	- \partial_t (t \partial_t \SFRenormalized)
	+ t^{1/3} \partial_x^2 \SFRenormalized 
	& = 0.
\end{align}
 The methods of \cites{fBpL2010a,fBpL2010c} (see also the many other related works cited in Subsect.\ \ref{SS:PREVIOUSWORK}), 
can be used to show that given an asymptotic expansion for the end state
of the form $\ln t \Psi_1(x) + \Psi_2(x)$
(where the $\Psi_i$ have sufficient Sobolev regularity),
one can construct a solution $\SFRenormalized$ to \eqref{E:SCALARFIELDWAVESIMPLIFIED}
existing on a slab of the form $(0,1] \times \mathbb{T}$ such that
\begin{align} \label{E:SFASYMPTOTICEXPANSION}
	\SFRenormalized = \ln t \Psi_1(x) + \Psi_2(x) + \mathcal{R}(t,x).
\end{align}
Furthermore, there is a suitably strong $t$-dependent
Sobolev norm on the time slices $\Sigma_t$
such that the norm of the remainder term $\mathcal{R}$
vanishes as $t \downarrow 0$. In particular, 
$\mathcal{R}$ becomes negligible relative to $\ln t \Psi_1(x) + \Psi_2(x)$
as $t \downarrow 0$. 
We now sketch the proof of these phenomena 
by following the approach outlined in Subsect.\ \ref{SS:PREVIOUSWORK}.
We note that our analysis involves much simpler $t$ weights in the energies 
compared to the weights of \cites{fBpL2010a,fBpL2010c} because we are treating a simple linear scalar 
equation. 
We recall that the overall strategy of the proof is to construct a sequence of standard initial value problems that approximate the ``singular initial value problem with vanishing Cauchy data for $\mathcal{R}$ given along $\Sigma_0$.''
To begin our sketch of a proof, 
we use equation \eqref{E:SCALARFIELDWAVESIMPLIFIED} and the ansatz \eqref{E:SFASYMPTOTICEXPANSION} 
to deduce the following equation for $\mathcal{R}(t,x)$:
\begin{align} \label{E:ERRORTERMSCALARFIELDWAVESIMPLIFIED}
	- \partial_t (t \partial_t \mathcal{R})
	+ t^{1/3} \partial_x^2 \mathcal{R}
	& = - t^{1/3} \ln t \partial_x^2 \Psi_1(x) 
		- t^{1/3} \partial_x^2 \Psi_2(x).
\end{align}
We now derive an estimate for the energy $\mathscr{E}[\mathcal{R}](t) \geq 0$ defined by
\begin{align}
	\mathscr{E}^2[\mathcal{R}](t)
	:= 
	\int_{\Sigma_t}
		(t^{1/3} \partial_t \mathcal{R})^2
		+
		(\partial_x \mathcal{R})^2
	\, dx.
\end{align}
A straightforward integration by parts argument,
based on multiplying equation
\eqref{E:ERRORTERMSCALARFIELDWAVESIMPLIFIED} by
$t^{-1/3} \partial_t \mathcal{R}$,
yields that 
for $0 < t_1 < t_2 \leq 1$, we have
\begin{align} \label{E:ERRORTERMESTIMATE}
	\mathscr{E}[\mathcal{R}](t_2)
	& \leq
		\mathscr{E}[\mathcal{R}](t_1)
		+
		\left\lbrace
			\left\| 
				\partial_x^2 \Psi_1
			\right\|_{L^2}
			+
			\left\| 
				\partial_x^2 \Psi_2
			\right\|_{L^2}
		\right\rbrace
		\int_{s=t_1}^{t_2}
			(1 + \ln s) s^{-1/3} 
		\, ds
			\\
		& \leq
		\mathscr{E}[\mathcal{R}](t_1)
		+
		C
		\left\lbrace
			\left\| 
				\partial_x^2 \Psi_1
			\right\|_{L^2}
			+
			\left\| 
				\partial_x^2 \Psi_2
			\right\|_{L^2}
		\right\rbrace
		\left\lbrace
			t_2^p
			-
			t_1^p
		\right\rbrace,
		\notag
\end{align}
where $p$ is a constant chosen to be slightly smaller than $2/3$.
Inequality \eqref{E:ERRORTERMESTIMATE} is the main ingredient that one needs
to deduce the desired existence result and estimates for $\mathcal{R}$.
Note that the estimate \eqref{E:ERRORTERMESTIMATE}
loses one derivative relative to $\Psi_1$ and $\Psi_2$.
In a detailed proof of the desired results 
(see the methods of \cite{fBpL2010c}), 
one considers a sequence $\lbrace \mathcal{R}_n \rbrace_{n=0}^{\infty}$ of 
solutions to \eqref{E:ERRORTERMSCALARFIELDWAVESIMPLIFIED},
where $\mathcal{R}_n$ has $0$ Cauchy data on $\Sigma_{t_n}$
(and thus $\mathscr{E}[\mathcal{R}_n](t_n) = 0$)
and is a classical solution on
$[t_n,1]$. Here,
$\lbrace t_n \rbrace_{n=0}^{\infty}$ is a sequence of times in $(0,1]$ that decreases to $0$ as $n \to \infty$.
%
%
An argument similar to the one used to prove
\eqref{E:ERRORTERMESTIMATE} yields
that for $m < n$, we have
\begin{align} \label{E:SEQUENCEISCAUCHY}
	\sup_{t \in [t_m,1]}
	\mathscr{E}[\mathcal{R}_n - \mathcal{R}_m](t) 
	\leq
		C
		\left\lbrace
			\left\| 
				\partial_x^2 \Psi_1
			\right\|_{L^2}
			+
			\left\| 
				\partial_x^2 \Psi_2
			\right\|_{L^2}
		\right\rbrace
		\left\lbrace
			t_m^p
			-
			t_n^p
		\right\rbrace.
\end{align}
It follows from \eqref{E:SEQUENCEISCAUCHY} 
that for any $\epsilon > 0$,
$\lbrace \mathcal{R}_n \rbrace_{n=0}^{\infty}$
is Cauchy in the norm\footnote{Higher-order energy estimates for the sequence $\lbrace \mathcal{R}_n \rbrace_{n=0}^{\infty}$ 
can be obtained in a similar fashion.}	
\[
f \rightarrow
\sup_{t \in (\epsilon,1]}
		\left\lbrace
			\| t^{1/3} \partial_t f (t) \|_{L^2}
			+
			\| \partial_x f(t) \|_{L^2}
		\right\rbrace
\]
and thus converges\footnote{In the fully detailed construction of the analog of $\mathcal{R}$ for the 
nonlinear problems treated in \cites{fBpL2010a,fBpL2010c}, 
the authors extend the $\mathcal{R}_n$ to be $0$ on $[0,t_n)$ and show that
this extension implies that $\mathcal{R}_n$ is a weak solution on an interval $[0,\updelta)$.} 
to the desired solution $\mathcal{R}$.

\begin{remark}
\label{R:FREEDOMINCHOOSINGEXPONENTS}
	We could have instead derived energy estimates 
	by multiplying equation \eqref{E:ERRORTERMSCALARFIELDWAVESIMPLIFIED}
	by $t^{-P} \partial_t \mathcal{R}$ for any choice of $P \in [1/3,5/3)$, 
	and a similar argument would yield a uniform bound for the energy
		$
	\int_{\Sigma_t}
		(t^{\frac{1-P}{2}} \partial_t \mathcal{R})^2
		+
		(t^{\frac{1/3-P}{2}}\partial_x \mathcal{R})^2
	\, dx
	$
	for $t \in (0,1]$.
	We could even have allowed $P$ to mildly depend on $x$.
	This illustrates the freedom 
	(mentioned in Subsect.\ \ref{SS:PREVIOUSWORK})
	in choosing viable $t$-weights in the Fuchsian approach.
\end{remark}

It is not difficult to modify the above arguments so that they apply
if one includes the semilinear term\footnote{More precisely, when the term $t^{1/3} (\partial_x \SFRenormalized)^2$ is present,
one can show that the remainder term $\mathcal{R}$ exists and verifies estimates similar to the ones derived above
on a small slab $(0,\updelta] \times \mathbb{T}$, where $\updelta > 0$
depends on a Sobolev norm of $\Psi_1$ and $\Psi_2$. This argument requires higher-order energy estimates
because of the nonlinearity.} $t^{1/3} (\partial_x \SFRenormalized)^2$ on the right-hand side of
\eqref{E:SCALARFIELDWAVESIMPLIFIED};
this term is a model for the kinds of semilinear terms that one finds in the Einstein-scalar field system.
It would be interesting to know to what extent
the arguments can be extended to apply to the full linearized system of Prop.\ \ref{P:LINEARIZEDCMCEQUATIONS}
and the full nonlinear Einstein-scalar field system in three spatial dimensions. 
The framework of \cite{eAfBjIpL2013} provides a possible starting point
for establishing such an extension. However, that framework 
applies only to symmetric hyperbolic Fuchsian systems 
and thus it would need to be modified to treat the 
Einstein-scalar field system in gauges involving an elliptic or parabolic lapse PDE.

\section{Parabolic lapse gauges} 
\label{S:PARABOLICMONOTONICITY}
In this section, we introduce a new family of gauges
for the Einstein-scalar field system. We show that a version of the approximate
monotonicity identity also holds
in solutions to linearized (around the Kasner backgrounds) 
versions of the corresponding equations;
see Theorem~\ref{T:PARABOLICMONOTONICITYID}.
We also show that mildly singular energy estimates
without derivative loss hold for the linear solutions when the Kasner backgrounds
are nearly spatially isotropic; see Theorem~\ref{T:PARABOLICENERGYESTIMATES}.
Using these results, one could also prove linear stability results 
when the Kasner backgrounds
are nearly spatially isotropic,
that is, an analog of Theorem~\ref{T:CMCLINEARSTABILITY}. However, 
for brevity,
we do not explicitly provide such a result here;
given the results of 
Theorems~\ref{T:PARABOLICMONOTONICITYID} and \ref{T:PARABOLICENERGYESTIMATES}, 
one could prove linear stability by making minor modifications
to the proof of Theorem~\ref{T:CMCLINEARSTABILITY}.

The gauge that we study in this section involves 
a parabolic equation for the lapse variable $n$
that depends on a real parameter $\uplambda$.
The mildly singular energy estimates of
Theorem~\ref{T:PARABOLICENERGYESTIMATES}
are valid for near-FLRW Kasner backgrounds
when $2 < \uplambda < \infty$. As we will see, 
for $\uplambda > 0$,
\emph{the parabolic lapse PDEs are locally well-posed only in the past direction},
that is, for $t$ decreasing.
Formally, $\uplambda = \infty$ corresponds to the CMC lapse equation.
However, our proofs in this section are
somewhat different compared to our proofs in CMC gauge
and do not allow us to directly recover the CMC gauge results
by taking a limit $\uplambda \to \infty$.

\subsection{Choice of a gauge and the corresponding formulation of the Einstein-scalar field equations}
In formulating the nonlinear Einstein-scalar field equations in the new gauge,
we continue to use transported spatial coordinates and to decompose  
$\gfour = - n^2 dt^2 + g_{ab} dx^a dx^b$ as in \eqref{E:GFOURCMCTRANSPORTED}.

\subsubsection{Fixing the gauge}
We now fix the lapse gauge.

\begin{definition}[\textbf{Choice of a parabolic lapse gauge}]
\label{D:PARABOLICLAPSEGAUGE}
Let $\uplambda \neq 0$ be a real number. We now impose the following relation, which fixes the lapse gauge:
 \begin{align} \label{E:PARABOLICLAPSERELATION}
	\uplambda^{-1}(n-1) & = t \SecondFund_{\ a}^a + 1.
\end{align}
\end{definition}

\begin{remark}
	Note that the CMC-transported spatial coordinates gauge
	of Sect.\ \ref{S:CMCFORMULATIONOFEINSTEIN}
	corresponds to $\uplambda = \infty$.
\end{remark}

\subsubsection{Formulation of the Einstein-scalar field equations}
We now provide the (nonlinear) Einstein-scalar field equations relative to the gauge
\eqref{E:PARABOLICLAPSERELATION} with transported\footnote{By ``transported,'' we mean in the sense described
below equation \eqref{E:GFOURCMCTRANSPORTED}.} 
spatial coordinates.

\begin{proposition}[\textbf{The Einstein-scalar field equations in the gauge \eqref{E:PARABOLICLAPSERELATION} 
with transported spatial coordinates}]
\label{P:EINSTEINSFPARABOLIC}
Under the gauge condition \eqref{E:PARABOLICLAPSERELATION} 
and with transported spatial coordinates, 
the Einstein-scalar field system consists of the following equations.

The \textbf{Hamiltonian and momentum constraint equations} are respectively:
\begin{subequations}
\begin{align}
		R - \SecondFund_{\ b}^a \SecondFund_{\ a}^b + \underbrace{(\SecondFund_{\ a}^a)^2}_{
			t^{-2}\left\lbrace\uplambda^{-1}(n-1) - 1\right\rbrace^2} 
		& = \overbrace{(n^{-1} \partial_t \phi)^2 + g^{ab} \nabla_a \phi \nabla_b \phi}^{2 \Tfour(\Nml,\Nml)}, \label{E:PARABOLICHAMILTONIAN} \\
		\nabla_a \SecondFund_{\ i}^a - \underbrace{\nabla_i \SecondFund_{\ a}^a}_{\uplambda^{-1}t^{-1} \nabla_i n} & = 
		\underbrace{- n^{-1} \partial_t \phi \nabla_i \phi}_{- \Tfour(\Nml,\partial_i)}, \label{E:PARABOLICMOMENTUM}
\end{align}
\end{subequations}
$R$ denotes the scalar curvature of $g_{ij}$.

The \textbf{metric evolution equations} are:
\begin{subequations}
\begin{align}
	\partial_t g_{ij} & = - 2 n g_{ia}\SecondFund_{\ j}^a, \label{E:PARABOLICPARTIALTGCMC} \\
	\partial_t \SecondFund_{\ j}^i & = - g^{ia} \nabla_a \nabla_j n
		+ n \Big\lbrace \Ric_{\ j}^i + \underbrace{\SecondFund_{\ a}^a \SecondFund_{\ j}^i}_{
			t^{-1}\left\lbrace\uplambda^{-1}(n-1) - 1\right\rbrace \SecondFund_{\ j}^i} 
			 \underbrace{- g^{ia} \nabla_a \phi \nabla_j \phi}_{- T_{\ j}^i + (1/2)\ID_{\ j}^i \Tfour} 
			\Big\rbrace,  \label{E:PARABOLICPARTIALTKCMC}
\end{align}
\end{subequations}
where $\Ric_{\ j}^i$ denotes the Ricci curvature of $g_{ij}$,
$\ID_{\ j}^i = \mbox{\upshape diag}(1,1,1)$ denotes the identity transformation,
and $\Tfour := (\gfour^{-1})^{\alpha \beta} \Tfour_{\alpha \beta}$ denotes the trace of the energy-momentum tensor \eqref{E:EMTSCALARFIELD}.

The \textbf{volume form factor} $\sqrt{\mbox{\upshape det} g}$
verifies the auxiliary equation\footnote{This equation, which we do not use
in the present article, 
is implied by \eqref{E:PARTIALTGCMC} and 
the gauge condition \eqref{E:PARABOLICLAPSERELATION}.}
\begin{align} \label{E:PARABOLICVOLUMEFORMEVOLUTIONEQUATION}
	\partial_t \ln \left(t^{-1} \sqrt{\mbox{\upshape det} g} \right)
	& = (1 - \uplambda^{-1}) \frac{n - 1}{t}.
\end{align}

The \textbf{scalar field wave equation} is:
\begin{align} \label{E:PARABOLICSCALARFIELDCMC}
	\overbrace{- n^{-1} \partial_t(n^{-1} \partial_t \phi)}^{- \Dfour_{\Nml} \Dfour_{\Nml} \phi} + g^{ab} \nabla_a \nabla_b \phi 
		& = \overbrace{\frac{1}{t} n^{-1} 
					\left\lbrace 1 - \uplambda^{-1}(n-1) \right\rbrace \partial_t \phi}^{- \SecondFund_{\ a}^a \Dfour_{\Nml} \phi} 
		- n^{-1} g^{ab} \nabla_a n \nabla_b \phi.	
\end{align}

The \textbf{parabolic lapse equation} is:
\begin{align} \label{E:PARABOLICLAPSE}
	\uplambda^{-1} \frac{1}{t} \partial_t (n-1)
	+ g^{ab} \nabla_a \nabla_b (n - 1) 
		& = (n - 1) 
			\Big\lbrace 
				\frac{1}{t^2}(1 - \uplambda^{-1})
				+ R 
				- g^{ab} \nabla_a \phi \nabla_b \phi
		\Big\rbrace \\
	& \ \ 
		+  \uplambda^{-1}(\uplambda^{-1} - 2) \frac{1}{t^2} (n-1)^2
		+  \uplambda^{-2} \frac{1}{t^2} (n-1)^3
		+ R 
		- g^{ab}\nabla_a \phi \nabla_b \phi. \notag
\end{align}

When $\uplambda > 0$, the gauge condition \eqref{E:PARABOLICLAPSERELATION} 
and the constraint equations \eqref{E:PARABOLICHAMILTONIAN}-\eqref{E:PARABOLICMOMENTUM}
are preserved by the \textbf{past} flow of the remaining equations if they are verified by the data.
\end{proposition}

\begin{remark}
	We are primarily interested in the gauge \eqref{E:PARABOLICLAPSERELATION}
	when $\uplambda > 2$ since our main results rely on this inequality.
	Note that when $\uplambda > 0$, the parabolic equation
	\eqref{E:PARABOLICLAPSE} is locally well posed only in the \emph{past} direction.
\end{remark}	

\begin{remark}[\textbf{Data for the lapse}]
	In order to solve the equations of Prop.\ \ref{P:EINSTEINSFPARABOLIC},
	we must prescribe the lapse along the initial Cauchy hypersurface $\Sigma_1$.
	That is, $n|_{\Sigma_1}$ is not determined by the geometric data
	(see Subsect.\ \ref{SS:INITIALVALUEANDCMC} for discussion of the geometric data). This is in contrast 
	to the CMC-transported spatial coordinates gauge, in which 
	$n|_{\Sigma_1}$ is determined by the geometric data
	via the elliptic PDE \eqref{E:LAPSE}.
	A natural choice in the context 
	of proving the nonlinear stability of the FLRW solution's Big Bang singularity
	would be $n|_{\Sigma_1} = 1$.
\end{remark}	

\begin{proof}[Proof of Prop.\ \ref{P:EINSTEINSFPARABOLIC}]
	The proposition can be proved by making simple modifications
	to the standard arguments that yield Prop.\ \ref{P:EINSTEINSFCMC}.
\end{proof}

\subsection{Linearizing around the Kasner solutions}
In the next proposition, we linearize the equations of Prop.\ \ref{P:EINSTEINSFPARABOLIC}
around a Kasner solution \eqref{E:KASNER}. See Subsect.\ \ref{SS:LINEARIZEDQUANTITIES}
for some remarks on the linearization procedure.

\begin{proposition}[\textbf{The linearized Einstein-scalar field equations in the gauge \eqref{E:PARABOLICLAPSERELATION} 
with transported spatial coordinates}]
\label{P:PARABOLICLINEARIZEDCMCEQUATIONS}
Consider the equations of Prop.\ \ref{P:EINSTEINSFPARABOLIC}
linearized around a Kasner solution \eqref{E:KASNER}.
The linearized equations in the unknowns 
$(\LapseRenormalized,\grenormalized,\LinSecondFund,\SFRenormalized)$,
which are functions of $(t,x) \in (0,\infty) \times \mathbb{T}^3$,
take the following form
(see Def.\ \ref{D:LINEARIZEDVARIABLES} for the definitions of some of the quantities).

The \textbf{linearized parabolic gauge condition} \eqref{E:PARABOLICLAPSERELATION} is:
\begin{align} \label{E:PARABOLICTRACECONDITION}
	t \LinSecondFund_{\ a}^a & = \uplambda^{-1} \LapseRenormalized.
\end{align}

The \textbf{linearized versions of the Hamiltonian and momentum constraint equations} \eqref{E:PARABOLICHAMILTONIAN}-\eqref{E:PARABOLICMOMENTUM} are:
\begin{subequations}
\begin{align}
	  t^2 \currenormalized
		- 2 (t \tracefreeSecondFundKasner_{\ b}^a) (t \LinSecondFund_{\ a}^b)
		- 2 A (t \partial_t \SFRenormalized)
		+ 2 (A^2 - \uplambda^{-1}) \LapseRenormalized 
		& = 0,
		\label{E:PARABOLICLINEARIZEDLHAMILTONIAN} \\
	\partial_a (t \LinSecondFund_{\ i}^a) 
		& = \uplambda^{-1} \partial_i \LapseRenormalized
			- A \partial_i \SFRenormalized
			- \christrenormalizedarg{a}{a}{b} (t \tracefreeSecondFundKasner_{\ i}^b)
			+ \christrenormalizedarg{a}{b}{i} (t \tracefreeSecondFundKasner_{\ b}^a),
		\label{E:PARABOLICLINEARIZEDMOMENTUM}
		\\
		\gKasner^{ab} \partial_a (t \LinSecondFund_{\ b}^i)
		& = \uplambda^{-1} \gKasner^{ia} \partial_a \LapseRenormalized
			- A \gKasner^{ia} \partial_a \SFRenormalized
				\label{E:PARABOLICLINEARIZEDSECONDMOMENTUM} \\
		& \ \ 
			- \gKasner^{ab} \christrenormalizedarg{a}{i}{c} (t \tracefreeSecondFundKasner_{\ b}^c)
			+ \gKasner^{ab} \christrenormalizedarg{a}{c}{b} (t \tracefreeSecondFundKasner_{\ c}^i).
			\notag
\end{align}
\end{subequations}

The \textbf{linearized version of the lapse equation} \eqref{E:PARABOLICLAPSE} can be expressed in either of the following two forms:
\begin{subequations}
\begin{align} 
	2 A (t \partial_t \SFRenormalized) 
	+ 2 (t \tracefreeSecondFundKasner_{\ b}^a) (t \LinSecondFund_{\ a}^b)
	& = \uplambda^{-1} t \partial_t \LapseRenormalized
		+ t^2 \gKasner^{ab} \partial_a \partial_b \LapseRenormalized
		+ (2 A^2 - 1 - \uplambda^{-1}) \LapseRenormalized, 
	\label{E:PARABOLICLINEARIZEDLAPSE} 
		\\
	\uplambda^{-1} t \partial_t \LapseRenormalized
	+ t^2 \gKasner^{ab} \partial_a \partial_b \LapseRenormalized 
	- (1 - \uplambda^{-1}) \LapseRenormalized
	& = t^2 \currenormalized.
		\label{E:PARABOLICLINEARIZEDLAPSELOWER} 
\end{align}
\end{subequations}
Equation \eqref{E:PARABOLICLINEARIZEDLHAMILTONIAN} can be used to show that \eqref{E:PARABOLICLINEARIZEDLAPSE} 
is equivalent to \eqref{E:PARABOLICLINEARIZEDLAPSELOWER}.

The \textbf{linearized versions of the metric evolution equations} \eqref{E:PARABOLICPARTIALTGCMC}-\eqref{E:PARABOLICPARTIALTKCMC} are:
\begin{subequations}
\begin{align}
	\partial_t \grenormalized_{ij} 
		& = -2 t^{-1} (t \SecondFundKasner_{\ j}^a) \grenormalized_{ia}
			- 2 t^{-1} \gKasner_{ia} (t \LinSecondFund_{\ j}^a)
			- 2 t^{-1}  \gKasner_{ia} (t \SecondFundKasner_{\ j}^a) \LapseRenormalized, 
		\label{E:PARABOLICLINEARIZEDGEVOLUTION} \\
	\partial_t (t \LinSecondFund_{\ j}^i)
		& = - t \gKasner^{ia} \partial_a \partial_j \LapseRenormalized
		- (1 - \uplambda^{-1}) t^{-1} (t \SecondFundKasner_{\ j}^i) \LapseRenormalized 
		+ t \Ricrenormalizedarg{i}{j}.
		\label{E:PARABOLICLINEARIZEDKEVOLUTION}
\end{align}
\end{subequations}

The \textbf{linearized version of the scalar field wave equation} \eqref{E:PARABOLICSCALARFIELDCMC} is:
\begin{align} \label{E:PARABOLICSCALARFIELDWAVEDECOMPOSED}
	- \partial_t (t \partial_t \SFRenormalized)
	+ t \gKasner^{ab} \partial_a \partial_b \SFRenormalized 
	& = - A \partial_t \LapseRenormalized
		+ A (1 - \uplambda^{-1})t^{-1} \LapseRenormalized.
\end{align}
\end{proposition}

\begin{proof}
	The proof is essentially the same as that of Prop.\ \ref{P:LINEARIZEDCMCEQUATIONS}
	and we therefore omit the details.
	We point out that
	in the gauge \eqref{E:PARABOLICLAPSERELATION}
	(and therefore in Prop.\ \ref{P:PARABOLICLINEARIZEDCMCEQUATIONS} too), 
	the linearly small quantities are the same as the ones from Def.\ \ref{D:LINEARIZEDVARIABLES},
	except that 
	$t \LinSecondFund_{\ a}^a 
	:= t \SecondFund_{\ a}^a - t \SecondFundKasner_{\ a}^a
	= \uplambda^{-1}(n-1) = \uplambda^{-1} \LapseRenormalized
	$ 
	is now linearly small
	rather than completely vanishing as it did in Prop.\ \ref{P:LINEARIZEDCMCEQUATIONS}.
\end{proof}

\subsection{Energies and norms}
\label{SS:PARABOLICENERGIESANDNORMS}
In our analysis of solutions, we will use the energies
and norms featured in the next two definitions.
These controlling quantities
lead to slightly different estimates
for the lapse compared to the CMC gauge.
The main point is that 
we are no longer able to obtain
control of the highest-order analog of $\| \partial^2 \LapseRenormalized \|_{L_{\gKasner}^2}$
because of the nature of parabolic energy estimates.
We are, however, able to control a spacetime integral
of the highest-order analog of $\partial^2 \LapseRenormalized$, which is provided by
the highest-order analog of the first term on the second line of the 
right-hand side of \eqref{E:TOPORDERLAPSEENERGYESTIMATE}.

\begin{definition}[\textbf{Energies}]
\label{D:PARBOLICENERGY}
In terms of the energies defined in Def.\ \ref{D:ENERGIES},
we define the following energy 
$
\mathscr{E}_{(Almost \ Total);\smallparameter}(t) \geq 0
$
for $t \in (0,1]$:
\begin{align} \label{E:PARABOLICALMOSTOTALENERGY}	
		\mathscr{E}_{(Almost \ Total);\smallparameter}^2(t) 
		& := 
			\mathscr{E}_{(Scalar)}^2(t)
			+ 
			\mathscr{E}_{(Lapse)}^2(t)
			+
			\smallparameter \mathscr{E}_{(Metric)}^2(t).
\end{align}
As in Theorem~\ref{T:L2MILDENERGYBLOWUPCMCGAUGE},
$\smallparameter$ is a small positive constant that we will
choose below in order in to obtain the desired energy estimates.

We will also use an up-to-order $M$ energy.
Specifically, we view the energy
$\mathscr{E}_{(Almost \ Total);\smallparameter}$
defined in \eqref{E:PARABOLICALMOSTOTALENERGY}
as a functional of 
$\LinSecondFund, 
\partial \grenormalized, 
 \partial_t \SFRenormalized,
\partial \SFRenormalized,
\LapseRenormalized
$
(that is, 
$\mathscr{E}_{(Almost \ Total);\smallparameter}
= 
\mathscr{E}_{(Almost \ Total);\smallparameter}
[\LinSecondFund, 
	\partial \grenormalized, 
  \partial_t \SFRenormalized,
	\partial \SFRenormalized,
	\LapseRenormalized]$),
and we define
\begin{align} \label{E:PARABOLICORDERNENERGY}
		\mathscr{E}_{(Almost \ Total);\smallparameter;M}^2(t)
		& := \sum_{|\vec{I}| \leq M}
			\mathscr{E}_{(Almost \ Total);\smallparameter}^2
			[\partial_{\vec{I}} \LinSecondFund, \partial \partial_{\vec{I}} \grenormalized, 
				\partial_t \partial_{\vec{I}} \SFRenormalized,
			\partial \partial_{\vec{I}} \SFRenormalized, \partial_{\vec{I}} \LapseRenormalized](t).
\end{align}
\end{definition}

\begin{definition}[\textbf{Solution norms}]
\label{D:PARBOLICNORM}
In terms of the Sobolev norms of Def.\ \ref{D:SOBOLEVNORMS}, we define
the solution norms
\begin{align}
	\parabolichighnorm{M}(t) 
	& := 
			\left\| t \LinSecondFund \right\|_{H_{Frame}^M} 
			+ \| \partial \grenormalized  \|_{H_{Frame}^M} 
			+ \left\| t \partial_t \SFRenormalized \right\|_{H_{Frame}^M} 
			+ t^{2/3} \| \partial \SFRenormalized  \|_{H_{Frame}^M} 
			+ \sum_{p=0}^1 t^{(2/3)p} \left\| \LapseRenormalized \right\|_{H^{M+p}}.
		\label{E:PARABOLICHIGHNORM}  
\end{align}
\end{definition}

\begin{remark}
	Note that
	$\parabolichighnorm{0}$
	controls one derivative of $\LapseRenormalized$ while
	$\mathscr{E}_{(Almost \ Total);\smallparameter}$
	does not.
\end{remark}

\subsection{The approximate monotonicity identity}
\label{SS:APPROXIMATEMONOTONICITYPARABOLIC}
We now state our approximate monotonicity identity theorem for
solutions to the linear equations of Prop.\ \ref{P:PARABOLICLINEARIZEDCMCEQUATIONS}.
The theorem is a direct analog of Theorem~\ref{T:CMCMONOTONICITYID} in the CMC gauge.

\begin{theorem}[\textbf{The approximate monotonicity identity in the parabolic lapse gauge}]
	\label{T:PARABOLICMONOTONICITYID}
	Assume that the parabolic gauge parameter verifies $\uplambda \neq 0$.
	Then for any constant $\smallparameter > 0$,
	solutions to the linearized equations of Prop.\ \ref{P:PARABOLICLINEARIZEDCMCEQUATIONS}
	verify the following identity for $t \in (0,1]$:
	\begin{align} \label{E:PARABOLICMONOTONICITYID} 
	& 
	\int_{\Sigma_t} 
		(t \partial_t \SFRenormalized)^2
		+ 
		|t \partial \SFRenormalized|_{\gKasner}^2
	\, dx  
	+ \left\lbrace
			A^2 
			+ \frac{1}{2} \uplambda^{-1}(1 - \uplambda^{-1})
		\right\rbrace
		 	\int_{\Sigma_t}
		 		\LapseRenormalized^2
		 \, dx
		+
		\smallparameter
		\int_{\Sigma_t} 
			|t \LinSecondFund|_{\gKasner}^2 
			+ 
			\frac{1}{4} |t \partial \grenormalized|_{\gKasner}^2 
		\, dx
		+ 
		\int_{\Sigma_t}
			\mathcal{N}_1
		\, dx
		\\
	& = 
		\int_{\Sigma_1} 
		(\partial_t \SFRenormalized)^2
		+ t^2 |\partial \SFRenormalized|_{\gKasner}^2
	\, dx  
	+ \left\lbrace
			A^2 
			+ 
			\frac{1}{2} \uplambda^{-1} (1 - \uplambda^{-1})
		\right\rbrace
		 	\int_{\Sigma_1}
		 		\LapseRenormalized^2
		 \, dx
		+
		\smallparameter
		\int_{\Sigma_1} 
			|\LinSecondFund|_{\gKasner}^2 
			+ 
			\frac{1}{4} |t \partial \grenormalized|_{\gKasner}^2 
		\, dx
		+
		\int_{\Sigma_1}
			\mathcal{N}_1
		\, dx	
		\notag \\
		& \ \ 
		- 2
			\int_{s=t}^1
				s^{-1}
				\int_{\Sigma_s}  
					 |s \partial \SFRenormalized|_{\gKasner}^2
				\, dx
		 	\, ds
			- (1 + 2 \smallparameter \uplambda^{-1} - \uplambda^{-1})
			\int_{s=t}^1
				s^{-1}
				\int_{\Sigma_s}  
					 |s \partial \LapseRenormalized|_{\gKasner}^2
				\, dx
		 	\, ds 
		-
		\left\lbrace
			1 - \uplambda^{-2}
		\right\rbrace
		\int_{s=t}^1 
			s^{-1}
			\int_{\Sigma_s} 
				\LapseRenormalized^2
			\, dx
		\, ds
		\notag \\
		& \ \
		- 
		\frac{1}{2}
		\smallparameter
		\int_{s=t}^1
			s^{-1}
			\int_{\Sigma_s}
				 		|s \partial \grenormalized|_{\gKasner}^2
			\, dx
		\, ds
		 \notag	\\
		&  \ \
		+
		\sum_{i=2}^4
		\int_{s=t}^1 
			s^{-1}
			\int_{\Sigma_s} 
				\mathcal{N}_i
			\, dx
		\, ds
		+
		\smallparameter
		\sum_{i=5}^{12}
		\int_{s=t}^1 
			s^{-1}
			\int_{\Sigma_s} 
				\mathcal{N}_i
			\, dx
		\, ds,
	\notag
\end{align}
where the constant $0 \leq A \leq \sqrt{2/3}$ is defined by \eqref{E:KASNERHAMILTONIANCONSTRAINT}
and along $\Sigma_s$, we have
\begin{subequations}
\begin{align}
	\mathcal{N}_1 
	& = \mathcal{N}_1(s \partial_t \SFRenormalized, \LapseRenormalized)
	:= - 2 A (s \partial_t \SFRenormalized) \LapseRenormalized,
		\label{E:PARABOLICCUBICFORM1} \\
	\mathcal{N}_2 
	& = \mathcal{N}_2(s \tracefreeSecondFundKasner,s \LinSecondFund,\LapseRenormalized)
	:= - 2 (1 - \uplambda^{-1})
		(s \tracefreeSecondFundKasner_{\ b}^a) (s \LinSecondFund_{\ a}^b) \LapseRenormalized,
		\label{E:PARABOLICFORM2} \\
	\mathcal{N}_3 
	 & = \mathcal{N}_3(s \SecondFundKasner, s \partial \SFRenormalized,s \partial \SFRenormalized)
		:= -2 s^2 \gKasner^{ab} (s \SecondFundKasner_{\ b}^c) \partial_a \SFRenormalized \partial_c 	
	\SFRenormalized,
		\label{E:PARABOLICCUBICFORM3} \\
	\mathcal{N}_4 
	& = \mathcal{N}_4(s \partial \SFRenormalized,s \partial \LapseRenormalized)
		:= - 2 A s^2 \gKasner^{ab} \partial_a \SFRenormalized \partial_b \LapseRenormalized,
		\label{E:PARABOLICFORM4}
		\\
	\mathcal{N}_5 
& = \mathcal{N}_5(s \SecondFundKasner, s \partial \grenormalized, s \partial \grenormalized)
:=   	-\frac{1}{2}
						s^2
						\gKasner^{ab} \gKasner^{ij} \gKasner^{cf}
						(s \SecondFundKasner_{\ c}^e)
						\partial_e \grenormalized_{ai}
						\partial_f \grenormalized_{bj},
						\label{E:PARABOLICFORM5}	\\
\mathcal{N}_6 
& = \mathcal{N}_6(s \tracefreeSecondFundKasner,s \LinSecondFund,s \LinSecondFund)
:= 2 \gKasner_{ic} \gKasner^{ab} (s \tracefreeSecondFundKasner_{\ j}^c) 
				(s \LinSecondFund_{\ a}^i) (s \LinSecondFund_{\ b}^j)
			- 2 \gKasner_{ij} \gKasner^{ac} (s \tracefreeSecondFundKasner_{\ c}^b) 
				(s \LinSecondFund_{\ a}^i) (s \LinSecondFund_{\ b}^j),
					\label{E:PARABOLICFORM6} \\
	\mathcal{N}_7 
	& = \mathcal{N}_7(s \tracefreeSecondFundKasner,s \partial \grenormalized, s \partial \grenormalized)
	:= 		s^2 
					\gKasner_{ab} \gKasner^{ef} \gKasner^{ij} 
					(s \tracefreeSecondFundKasner_{\ c}^a)
					\christrenormalizedarg{i}{c}{j} 
					\christrenormalizedarg{e}{b}{f}
				- s^2
					\gKasner_{ab} \gKasner^{ef} \gKasner^{ij} 
					(s \tracefreeSecondFundKasner_{\ j}^c)
					\christrenormalizedarg{i}{a}{c} 
					\christrenormalizedarg{e}{b}{f}
		\label{E:PARABOLICFORM7} 			\\
	& \ \ + s^2 \gKasner^{ef}  (s \tracefreeSecondFundKasner_{\ c}^a) 
						\christrenormalizedarg{a}{c}{b} \christrenormalizedarg{e}{b}{f}	
					- s^2 \gKasner^{ef} (s \tracefreeSecondFundKasner_{\ b}^c) 
							\christrenormalizedarg{a}{a}{c} \christrenormalizedarg{e}{b}{f},
		\notag
			\\
	\mathcal{N}_8 
	& = 
	\mathcal{N}_8(s \tracefreeSecondFundKasner,s \partial \grenormalized,s \partial \LapseRenormalized) 
	:= 2 s^2 \gKasner^{ij} (s \tracefreeSecondFundKasner_{\ i}^b)
						\christrenormalizedarg{a}{a}{b} \partial_j \LapseRenormalized
						- 2 s^2 \gKasner^{ij}  (s \tracefreeSecondFundKasner_{\ b}^a) \christrenormalizedarg{a}{b}{i} 
							\partial_j \LapseRenormalized
							\label{E:PARABOLICFORM8} \\
		& \ \ + s^2 \gKasner^{ij} \gKasner^{ef} (s \SecondFundKasner_{\ j}^a)
				  	\partial_e \grenormalized_{ai} \partial_f \LapseRenormalized,
				  	\notag	\\
	\mathcal{N}_9 
	& = 
		\mathcal{N}_9(s \tracefreeSecondFundKasner,s \LinSecondFund,\LapseRenormalized)
		:= 2 (1 - \uplambda^{-1}) \gKasner_{ab} \gKasner^{ij} 
			(s \tracefreeSecondFundKasner_{\ i}^a) 
			(s \LinSecondFund_{\ j}^b) \LapseRenormalized,
				\label{E:PARABOLICFORM9} \\
	\mathcal{N}_{10}
	& = \mathcal{N}_{10}(s \partial \SFRenormalized, s \partial \LapseRenormalized)
	:= 	2 A s^2\gKasner^{ij} \partial_i \SFRenormalized \partial_j \LapseRenormalized,
				\label{E:PARABOLICFORM10} \\
	\mathcal{N}_{11} 
	& = \mathcal{N}_{11}(s \partial \grenormalized, s \partial \SFRenormalized)
	:= - 2 A s^2 \gKasner^{ef} \christrenormalizedarg{e}{a}{f} \partial_a \SFRenormalized,
		\label{E:PARABOLICFORM11}
			\\
	\mathcal{N}_{12} 
	& = \mathcal{N}_{12}(s \partial \grenormalized,s \partial \LapseRenormalized)
	:=  2 \uplambda^{-1} t \gKasner^{ef} \christrenormalizedarg{e}{a}{f} \partial_a \LapseRenormalized.
							\label{E:PARABOLICFORM12}
\end{align}	
\end{subequations}

\end{theorem}	

\begin{remark}
	The terms $\mathcal{N}_i$
	defined in \eqref{E:PARABOLICCUBICFORM1}-\eqref{E:PARABOLICFORM12}
	have different definitions than their counterparts
	from Sects.\ \ref{S:MONOTONICITYIDENTITIES}-\ref{S:ENERGYESTIMATES}.
\end{remark}

\begin{proof}[Proof of Theorem~\ref{T:PARABOLICMONOTONICITYID}]
	We will derive
	the identities \eqref{E:PARABOLICLAPSESFADDEDENERGYESTIMATE}
	and \eqref{E:PARABOLICMETRICENERGYID}
	below using independent arguments.
	To obtain \eqref{T:PARABOLICMONOTONICITYID},
	we simply add
	\eqref{E:PARABOLICLAPSESFADDEDENERGYESTIMATE} to
	$\smallparameter$ times \eqref{E:PARABOLICMETRICENERGYID}.
\end{proof}

As in the proof of Theorem~\ref{T:CMCMONOTONICITYID}, 
the most important step in the proof of Theorem~\ref{T:PARABOLICMONOTONICITYID} is
an energy identity for the linearized scalar field and lapse
that simultaneously yields favorably signed (to the past) integrals for both variables. 
We provide this identity in the next proposition.

\begin{proposition}[\textbf{The key integral identity for the linearized scalar field and linearized lapse
in the parabolic lapse gauge}]
	\label{P:PARABOLICENERGYESTIMATELAPSEANDSCALARFIELD}
	Assume that the parabolic gauge parameter verifies $\uplambda \neq 0$.
	Then solutions to the linearized equations of Prop.\ \ref{P:PARABOLICLINEARIZEDCMCEQUATIONS}
	verify the following identity for $t \in (0,1]$:
	\begin{align}
	& 
	\int_{\Sigma_t} 
		(t \partial_t \SFRenormalized)^2
		+ 
		|t \partial \SFRenormalized|_{\gKasner}^2
	\, dx  
	+ \left\lbrace
			A^2 
			+ \frac{1}{2} \uplambda^{-1}(1 - \uplambda^{-1})
		\right\rbrace
		 	\int_{\Sigma_t}
		 		\LapseRenormalized^2
		 \, dx
			+
		 	\int_{\Sigma_t}
		 		\mathcal{N}_1
		 	\, dx
		\label{E:PARABOLICLAPSESFADDEDENERGYESTIMATE} \\
	& = 
		\int_{\Sigma_1} 
		(\partial_t \SFRenormalized)^2
		+ t^2 |\partial \SFRenormalized|_{\gKasner}^2
	\, dx  
	+ \left\lbrace
			A^2 
			+ \frac{1}{2} \uplambda^{-1} (1 - \uplambda^{-1})
		\right\rbrace
		 	\int_{\Sigma_1}
		 		\LapseRenormalized^2
		 \, dx
			+
		 	\int_{\Sigma_1}
		 		\mathcal{N}_1
		 	\, dx	
		\notag \\
	& \ \ 
		- 2
			\int_{s=t}^1
				s^{-1}
				\int_{\Sigma_s}  
					 |s \partial \SFRenormalized|_{\gKasner}^2
				\, dx
		 	\, ds
			- (1 - \uplambda^{-1})
			\int_{s=t}^1
				s^{-1}
				\int_{\Sigma_s}  
					 |s \partial \LapseRenormalized|_{\gKasner}^2
				\, dx
		 	\, ds 
		-
		(1 - \uplambda^{-2})
		\int_{s=t}^1 
			s^{-1}
			\int_{\Sigma_s} 
				\LapseRenormalized^2
			\, dx
		\, ds
		\notag \\
	&  \ \
		+
		\sum_{i=2}^4
		\int_{s=t}^1 
			s^{-1}
			\int_{\Sigma_s} 
				\mathcal{N}_i
			\, dx
		\, ds,
		\notag
\end{align}
where the constant $0 \leq A \leq \sqrt{2/3}$ is defined by \eqref{E:KASNERHAMILTONIANCONSTRAINT}
and the terms 
$\mathcal{N}_1$,
$\mathcal{N}_2$,
$\mathcal{N}_3$,
and
$\mathcal{N}_4$
are defined in \eqref{E:PARABOLICCUBICFORM1}-\eqref{E:PARABOLICFORM4}.
\end{proposition}

\begin{proof}
The proof has some features in common with
our proof of Prop.\ \ref{P:ENERGYESTIMATELAPSEANDSCALARFIELD},
but other aspects of it are different.
Again, the main idea is to combine three integration by parts identities in
the right way. 
Throughout, we silently use the identities in \eqref{E:KASNERTIMEDERIVATIVEIDENTITIES}.
To obtain the first identity, we divide equation \eqref{E:PARABOLICLINEARIZEDLAPSE}
by $t$ and then replace $t$ with the integration variable $s$,
multiply by $(1 - \uplambda^{-1}) \LapseRenormalized$, 
and integrate by parts 
over $(s,x) \in [t,1] \times \mathbb{T}^3$ (we stress that $t \leq 1$) to deduce that
\begin{align} \label{E:PARABOLICLAPSEFIRSTENERGYIDENTITY}
		\frac{1}{2} \uplambda^{-1} (1 - \uplambda^{-1}) 
	\int_{\Sigma_t}
		\LapseRenormalized^2
	\, dx
	& =
		\frac{1}{2}
		\uplambda^{-1} (1 - \uplambda^{-1})  
		\int_{\Sigma_1}
			\LapseRenormalized^2
		\, dx
	\\
		& \ \ 
		- (1 - \uplambda^{-1})
			\int_{s=t}^1
				s^{-1}
				\int_{\Sigma_s}  
					 |s \partial \LapseRenormalized|_{\gKasner}^2
				\, dx
		 	\, ds \notag \\
		& \ \ 
		+ (2 A^2 - 1 - \uplambda^{-1}) (1 - \uplambda^{-1})
		\int_{s=t}^1 
			s^{-1}
			\int_{\Sigma_s} 
				\LapseRenormalized^2
			\, dx
		\, ds
		\notag \\
& \ \ 
	- 2 A (1 - \uplambda^{-1})
		\int_{s=t}^1 
			s^{-1}
			\int_{\Sigma_s} 
				(s \partial_t \SFRenormalized) \LapseRenormalized
			\, dx
		\, ds
		\notag \\
& \ \
	- 2 (1 - \uplambda^{-1})
		\int_{s=t}^1 
			s^{-1}
			\int_{\Sigma_s} 
				(s \tracefreeSecondFundKasner_{\ b}^a) (s \LinSecondFund_{\ a}^b) \LapseRenormalized
			\, dx
		\, ds.
		\notag
\end{align}

To obtain the second identity, we replace $t$ with the integration variable $s$ in equation 
\eqref{E:PARABOLICSCALARFIELDWAVEDECOMPOSED},
multiply by $- s \partial_t \SFRenormalized$,
and integrate by parts 
over $(s,x) \in [t,1] \times \mathbb{T}^3$ 
(we again stress that $t \leq 1$)
to deduce that
\begin{align}
	\int_{\Sigma_t} 
		(t \partial_t \SFRenormalized)^2
		+ t^2 |\partial \SFRenormalized|_{\gKasner}^2
	\, dx  
	& = \int_{\Sigma_1} 
				(\partial_t \SFRenormalized)^2 
				+ |\partial \SFRenormalized|_{\gKasner}^2
			\, dx 
		\label{E:PARABOLICFIRSTENERGYESTIMATEMODELSCALARFIELD} \\
	& \ \ 
		- 2
			\int_{s=t}^1
				s^{-1}
				\int_{\Sigma_s}  
					 |s \partial \SFRenormalized|_{\gKasner}^2
				 	 + s^2 \gKasner^{ab} (s \SecondFundKasner_{\ b}^c)
							\partial_a \SFRenormalized \partial_c \SFRenormalized
		 	 	\, dx
		 	\, ds
		\notag  \\
	& \ \
			- 2 A
			\int_{s=t}^1 
				\int_{\Sigma_s}
					(s \partial_t \SFRenormalized) \partial_t \LapseRenormalized
		 		\, dx 
		 	\, ds
			+ 2 A (1 - \uplambda^{-1})
			\int_{s=t}^1 
				s^{-1}
				\int_{\Sigma_s}
					(s \partial_t \SFRenormalized) \LapseRenormalized
		 		\, dx 
		 	\, ds.
		 \notag
\end{align}

Next, we multiply equation \eqref{E:PARABOLICSCALARFIELDWAVEDECOMPOSED} by $\LapseRenormalized$ 
to obtain the following identity:
\begin{align} \label{E:PARABOLICLINEARIZEDSCALARFIELDLAPSEMIXEDEQUATION}
	(t \partial_t \SFRenormalized) \partial_t \LapseRenormalized
	& = \partial_t (t \partial_t \SFRenormalized \LapseRenormalized)
		- \frac{1}{2} A \partial_t (\LapseRenormalized^2)
	 	- t \LapseRenormalized \gKasner^{ab} \partial_a \partial_b \SFRenormalized 
	 	+ A (1 - \uplambda^{-1}) t^{-1} \LapseRenormalized^2.
\end{align}

To obtain the third identity, we replace $t$ with the integration variable $s$ in equation \eqref{E:PARABOLICLINEARIZEDSCALARFIELDLAPSEMIXEDEQUATION},
multiply by $2A$,
and integrate by parts 
over $(s,x) \in [t,1] \times \mathbb{T}^3$
to deduce that 
\begin{align}  \label{E:PARABOLICLINEARIZEDSCALARFIELDLAPSEMIXEDEQUATIONIBPIDENTITY}
		&  - 2A
		 			\int_{\Sigma_t}
		 				(t \partial_t \SFRenormalized) \LapseRenormalized
		 			\, dx
		 +	A^2 
		 			\int_{\Sigma_t}
		 				\LapseRenormalized^2
		 			\, dx
		 		\\
		 	& = 
		 		- 2A
		 			\int_{\Sigma_1}
		 				\partial_t \SFRenormalized \LapseRenormalized
		 			\, dx
		 		+ A^2 
		 			\int_{\Sigma_1}
		 				\LapseRenormalized^2
		 			\, dx
		 			\notag \\
		& \ \
			+ 2 A
			\int_{s=t}^1 
				\int_{\Sigma_s}
					(s \partial_t \SFRenormalized) \partial_t \LapseRenormalized
		 		\, dx 
		 	\, ds
			\notag \\
		& \ \
			- 2A
			 \int_{s=t}^1 
				s^{-1}
				\int_{\Sigma_s}
					s^2 \gKasner^{ab} \partial_a \SFRenormalized \partial_b \LapseRenormalized
		 		\, dx 
		 	\, ds
		 - 2 A^2 (1 - \uplambda^{-1}) 
		 		\int_{s=t}^1 
					s^{-1}
					\int_{\Sigma_s}
						\LapseRenormalized^2
		 			\, dx 
		 		\, ds.
		 		\notag
\end{align}
Adding 
\eqref{E:PARABOLICLAPSEFIRSTENERGYIDENTITY},
\eqref{E:PARABOLICFIRSTENERGYESTIMATEMODELSCALARFIELD}, 
and \eqref{E:PARABOLICLINEARIZEDSCALARFIELDLAPSEMIXEDEQUATIONIBPIDENTITY}, 
and noting the cancellation of the integrals
$
\pm 2 A
			\int_{s=t}^1 
				\int_{\Sigma_s}
					(s \partial_t \SFRenormalized) \partial_t \LapseRenormalized
		 		\, dx 
		 	\, ds
$
and
$
\pm 2 A (1 - \uplambda^{-1})
		\int_{s=t}^1 
			s^{-1}
			\int_{\Sigma_s} 
				(s \partial_t \SFRenormalized) \LapseRenormalized
			\, dx
		\, ds,
$
we arrive at the desired identity \eqref{E:PARABOLICLAPSESFADDEDENERGYESTIMATE}.
\end{proof}

In the next proposition, we derive an energy identity 
for the linearized metric solution variables. It is a direct analog of
Prop.\ \ref{PCMC:LINEARIZEDMETRICENERGYESTIMATE}.

\begin{proposition}[\textbf{Energy identity for the linearized metric variables
in the parabolic lapse gauge}]
\label{P:PARABOLICLINEARIZEDMETRICENERGYESTIMATE}
Assume that the parabolic gauge parameter verifies $\uplambda \neq 0$.
Then solutions to the linearized equations of Prop.\ \ref{P:PARABOLICLINEARIZEDCMCEQUATIONS}
verify the following identity for $t \in (0,1]$:
\begin{align} \label{E:PARABOLICMETRICENERGYID}
	\int_{\Sigma_t}
		|t \LinSecondFund|_{\gKasner}^2
		+ 
		\frac{1}{4} |t \partial \grenormalized|_{\gKasner}^2
	\, dx
	& = 
		\int_{\Sigma_1}
			|\LinSecondFund|_{\gKasner}^2
			+ \frac{1}{4} |\partial \grenormalized|_{\gKasner}^2
		\, dx 
			\\
	& \ \
		- 
		\frac{1}{2}
		\int_{s=t}^1
			s^{-1}
			\int_{\Sigma_s}
				 		|s \partial \grenormalized|_{\gKasner}^2
			\, dx
		\, ds
		- 
		2 \uplambda^{-1}
		\int_{s=t}^1
			s^{-1}
			\int_{\Sigma_s}
				|s \partial \LapseRenormalized|_{\gKasner}^2
			\, dx
		\, ds
			\notag	\\
	&  \ \
		+
		\sum_{i=5}^{12}
		\int_{s=t}^1 
			s^{-1}
			\int_{\Sigma_s} 
				\mathcal{N}_i
			\, dx
		\, ds,
		\notag
\end{align}
where the constant $0 \leq A \leq \sqrt{2/3}$ is defined by \eqref{E:KASNERHAMILTONIANCONSTRAINT}
and the terms 
$\mathcal{N}_5$,
$\cdots$,
$\mathcal{N}_{12}$
are defined in \eqref{E:PARABOLICFORM5}-\eqref{E:PARABOLICFORM12}.

\end{proposition}

\begin{proof}[Proof of Prop.\ \ref{P:PARABOLICLINEARIZEDMETRICENERGYESTIMATE}]
We repeat the proof of Prop.\ \ref{PCMC:LINEARIZEDMETRICENERGYESTIMATE}
and take into account the few differences between the linearized
equations of Prop.\ \ref{P:LINEARIZEDCMCEQUATIONS}
and the linearized equations of Prop.\ \ref{P:PARABOLICLINEARIZEDCMCEQUATIONS}.
In particular, the identity 
\eqref{E:METRICENERGYIDDIFFERNTIALFORM} holds in the present context, but with
the next-to-last term
$- 2 t^{-1} \gKasner_{ab} \gKasner^{ij} 
(t \tracefreeSecondFundKasner_{\ i}^a)
(t \LinSecondFund_{\ j}^b)
\LapseRenormalized$
multiplied by the factor $1 - \uplambda^{-1}$
(coming from the second term on the right-hand side of \eqref{E:PARABOLICLINEARIZEDKEVOLUTION})
and two additional terms: \textbf{i)} the term 
$2 \uplambda^{-1} t |\partial \LapseRenormalized|_{\gKasner}^2$
coming from the analog of the step \eqref{E:FIRSTDIFFBYPARTS} and the presence 
of the term $\uplambda^{-1} \partial_i \LapseRenormalized$ on the 
right-hand side of equation \eqref{E:PARABOLICLINEARIZEDMOMENTUM}
and \textbf{ii)} the cross term
$- 2 \uplambda^{-1} t \gKasner^{ef} 
			\christrenormalizedarg{e}{a}{f}
			\partial_a \LapseRenormalized$
coming from the analog of steps \eqref{E:SECONDLINEARIZEDMOMENTUMDIFFBYPARTS}
and \eqref{E:FIRSTLINEARIZEDMOMENTUMDIFFBYPARTS}
and the presence 
of the term $\uplambda^{-1} \partial_i \LapseRenormalized$ on the 
right-hand side of equation \eqref{E:PARABOLICLINEARIZEDMOMENTUM}
and the term $\uplambda^{-1} \gKasner^{ia} \partial_a \LapseRenormalized$
on the right-hand side of \eqref{E:PARABOLICLINEARIZEDSECONDMOMENTUM}.
\end{proof}

\subsection{Mildly singular energy estimates without derivative loss for the linearized equations in the parabolic lapse gauge}
\label{SS:ENERGYESTIMATESWITOUTDERIVATVELOSSPARABOLIC}
In this subsection, we 
use the approximate monotonicity identity provided by Theorem~\ref{T:PARABOLICMONOTONICITYID}
to derive mildly singular energy estimates
for the linear solution when the Kasner background is nearly spatially isotropic.
The results are contained in Theorem~\ref{T:PARABOLICENERGYESTIMATES},
which is a direct analog of Theorem~\ref{T:L2MILDENERGYBLOWUPCMCGAUGE}.
We provide the proof of Theorem~\ref{T:PARABOLICENERGYESTIMATES} 
in Subsubsect.\ \ref{SSS:PROOFOFTHMPARABOLICENERGYESTIMATES}.

\begin{theorem}[\textbf{Mildly singular energy estimates without derivative loss for solutions to the linearized equations in the parabolic lapse gauge}]
\label{T:PARABOLICENERGYESTIMATES}
Consider a solution to the linear equations of Prop.\ \ref{P:PARABOLICLINEARIZEDCMCEQUATIONS}
corresponding to the data 
$\left(\LinSecondFund(1), \grenormalized(1), \partial_t \SFRenormalized(1),
\partial \SFRenormalized(1), \LapseRenormalized(1)\right)$
(given on $\Sigma_1 = \lbrace 1 \rbrace \times \mathbb{T}^3$).
Assume that the parabolic gauge parameter verifies 
$\uplambda \geq \uplambda_0$, where $\uplambda_0 > 2$.
There exist constants $\smallparameter_{\uplambda_0} > 0$,
$\tracefreeparameter_{\uplambda_0} > 0$,
$C_{\uplambda_0} > 0$, 
$c_{\uplambda_0} > 0$,
and $P_{\uplambda_0} > 0$
(depending on $\uplambda_0$)
such that if $0 \leq \tracefreeparameter \leq \tracefreeparameter_{\uplambda_0}$ 
and if the solution norm $\parabolichighnorm{0}(t)$ 
defined in \eqref{E:PARABOLICHIGHNORM} verifies
$\parabolichighnorm{0}(1) < \infty$,
then the energy
$\mathscr{E}_{(Almost \ Total);\smallparameter_{\uplambda_0}}(t)$
defined in \eqref{E:PARABOLICALMOSTOTALENERGY}
verifies the following inequality
for $t \in (0,1]$:
\label{T:PARABOLICL2MONOTONICITY}
\begin{align} \label{E:PARABOLICSECONDMODELENERGYGRONWALLREADY}
	\mathscr{E}_{(Almost \ Total);\smallparameter_{\uplambda_0}}^2(t) 
	& \leq 
	   C_{\uplambda_0} \mathscr{E}_{(Almost \ Total);\smallparameter_{\uplambda_0}}^2(1) 
		 \\
	& \ \ 
	\underbrace{
		- 
			P_{\uplambda_0}
			\smallparameter_{\uplambda_0} 
			\int_{s=t}^1 s^{-1} \int_{\Sigma_s} |s \partial \grenormalized|_{\gKasner}^2 \, dx \, ds}_{
			\mbox{\upshape Past-favorable sign}}
		\underbrace{- P_{\uplambda_0} \int_{s=t}^1 s^{-1} \int_{\Sigma_s} |s \partial \SFRenormalized|_{\gKasner}^2 \, dx \, ds}_{
		\mbox{\upshape Past-favorable sign}}
		 \notag \\
	& \ \ \underbrace{- P_{\uplambda_0} 
		\int_{s=t}^1 s^{-1} \int_{\Sigma_s} |s \partial \LapseRenormalized|_{\gKasner}^2 \, dx \, ds}_{
		\mbox{\upshape Past-favorable sign}}
	\underbrace{- P_{\uplambda_0} \int_{s=t}^1 s^{-1} \int_{\Sigma_s} \LapseRenormalized^2 \, dx \, ds}_{
		\mbox{\upshape Past-favorable sign}}
		\notag \\
	& \ \ + \underbrace{c_{\uplambda_0} \tracefreeparameter \int_{s=t}^1 s^{-1} 
				\mathscr{E}_{(Almost \ Total);\smallparameter_{\uplambda_0}}^2(s) \, ds}_{\mbox{\upshape Error integral that can create energy blowup}}. 
		\notag 
\end{align}	

Furthermore, the following estimate holds for $t \in (0,1]$:
\begin{align} \label{E:PARABOLICALMOSTTOTALENERGYGRONWALLED}
	\mathscr{E}_{(Almost \ Total);\smallparameter_{\uplambda_0}}(t)
	& \leq 
		C_{\uplambda_0}
		\mathscr{E}_{(Almost \ Total);\smallparameter_{\uplambda_0}}(1)
			t^{- c_{\uplambda_0} \tracefreeparameter}.
\end{align}

In addition, if $N \geq 0$ is an integer and 
the solution norm $\parabolichighnorm{N}(t)$ 
defined in \eqref{E:PARABOLICHIGHNORM} verifies
$\parabolichighnorm{N}(1) < \infty$,
then the energy 
$\mathscr{E}_{(Total);\smallparameter_{\uplambda_0};N}(t)$ defined in
\eqref{E:TOPORDERENERGY} verifies the following estimate for $t \in (0,1]$:
\begin{align} \label{E:PARABOLICTOTALENERGYGRONWALLED}
	\mathscr{E}_{(Total);\smallparameter_{\uplambda_0};N}(t)
	& \leq 
		\begin{cases}
		\frac{C_{\uplambda_0}}{\tracefreeparameter} \parabolichighnorm{N}(1) t^{- c_{\uplambda_0} \tracefreeparameter}
		& \mbox{\upshape if} \ \tracefreeparameter \neq 0,
			\\
		C_{\uplambda_0} \parabolichighnorm{N}(1) (1 + |\ln t|)
		& \mbox{\upshape if} \ \tracefreeparameter = 0.
	\end{cases}
\end{align}

In addition, if $N \geq 0$ is an integer and 
$\parabolichighnorm{N}(1) < \infty$,
then the following inequality holds for $t \in (0,1]$:
\begin{align} \label{E:PARABOLICHIGHNORMBOUND}
	\parabolichighnorm{N}(t) 
	& \leq 
	\begin{cases}
		\frac{C_{\uplambda_0}}{\tracefreeparameter} \parabolichighnorm{N}(1) t^{- c_{\uplambda_0} \tracefreeparameter}
		& \mbox{\upshape if} \ \tracefreeparameter \neq 0,
			\\
		C_{\uplambda_0} \parabolichighnorm{N}(1) (1 + |\ln t|)
		& \mbox{\upshape if} \ \tracefreeparameter = 0.
	\end{cases}
\end{align}	

\end{theorem}

\begin{remark}
	See the estimate \eqref{E:PARABOLICFIRSTENERGYESTIMATE}
	for more a more precise inequality that shows how the constants in the
	estimate \eqref{E:PARABOLICSECONDMODELENERGYGRONWALLREADY}
	depend on $\uplambda_0$ and on each other.
\end{remark}

\subsubsection{Preliminary estimates and identities for the proof of Theorem~\ref{T:PARABOLICENERGYESTIMATES}}
\label{SSS:PRELIMINARYFORPARABOLICLAPSEGAUGEENERGYESTIMATE}
In our proof of Theorem~\ref{T:PARABOLICENERGYESTIMATES}, we use the following
comparison lemma, which can be proved
by using arguments similar to the ones we used to prove Lemma~\ref{L:ENERGYNORMCOMPARISON}
(except that clearly we no do not use the elliptic estimate provided by 
Lemma~\ref{L:TOPORDERLAPSEINTEGRATIONBYPARTSINEQUALITY});
we omit the simple proof.

\begin{lemma}[\textbf{Parabolic energy-norm comparison lemma}]
\label{L:PARABOLICENERGYNORMCOMPARISON}
Let $N \geq 0$ be an integer and let
$\tracefreeparameter \geq 0$ be as defined in \eqref{E:TRACEFREEPARAMETER}.
There exist constants
$C > 0$ and $c > 0$, depending on $\smallparameter$,
such that the following comparison estimates hold for 
the norm $\parabolichighnorm{N}(t)$ defined in \eqref{E:PARABOLICHIGHNORM} 
and the energy $\mathscr{E}_{(Total);\smallparameter;N}(t)$ defined in \eqref{E:TOPORDERENERGY}
for $t \in (0,1]$:
\begin{subequations}
\begin{align}  \label{E:PARABOLICENERGYNORMCOMPARISON}
	\mathscr{E}_{(Total);\smallparameter;N}(t)
	& \leq C t^{-c \tracefreeparameter} \parabolichighnorm{N}(t),
		\\
	\parabolichighnorm{N}(t)
	& \leq C t^{-c \tracefreeparameter}
		\mathscr{E}_{(Total);\smallparameter;N}(t).
		\label{E:PARABOLICNORMENERGYCOMPARISON}
\end{align}
\end{subequations}

\end{lemma}
\hfill $\qed$

We will also use the following
simple parabolic energy estimate, which 
can be used to derive
top-order $L^2$ estimates for the linearized lapse
variable.

\begin{lemma}[\textbf{Parabolic energy estimate for $\LapseRenormalized$}]
	\label{L:TOPORDERLAPSEENERGYESTIMATE}
	There exists a constant $C > 0$ such that if $\tracefreeparameter \geq 0$
	(see definition \ref{E:TRACEFREEPARAMETER})
	and if the parabolic gauge parameter verifies 
	$\uplambda \geq 1$,
	then solutions $\LapseRenormalized$ to the linear parabolic equation \eqref{E:PARABOLICLINEARIZEDLAPSE} 
	verify the following inequality for $t \in (0,1]$:
	\begin{align} \label{E:TOPORDERLAPSEENERGYESTIMATE}
			\uplambda^{-1}
			\int_{\Sigma_t}
				|t \partial \LapseRenormalized|_{\gKasner}^2
			\, dx
			& \leq
				\uplambda^{-1}
				\int_{\Sigma_1}
					|\partial \LapseRenormalized|_{\gKasner}^2
				\, dx
					\\
		& \ \
			- \int_{s=t}^1
				s^{-1}
				\int_{\Sigma_s}
					|s^2 \partial^2 \LapseRenormalized|_{\gKasner}^2
				\, dx
				\, ds
			-
			\uplambda^{-1}
			\left(\frac{4}{3} - 2 \tracefreeparameter \right)
			\int_{s=t}^1
				s^{-1}
				\int_{\Sigma_s}
					|s \partial \LapseRenormalized|_{\gKasner}^2	
				\, dx
			\, ds	
			\notag \\
		& \ \ 
			+ C
				\int_{s=t}^1
					s^{-1}
					\int_{\Sigma_s}
						\left|
							(s \tracefreeSecondFundKasner_{\ b}^a) (s \LinSecondFund_{\ a}^b)
						\right|^2
				\, dx
				\, ds			
			+ C
				\int_{s=t}^1
					s^{-1}
					\int_{\Sigma_s}
						(s\partial_t \SFRenormalized)^2
				\, dx
				\, ds
				\notag \\
	& \ \
			+ C
				\int_{s=t}^1
					s^{-1}
					\int_{\Sigma_s}
						\LapseRenormalized^2
					\, dx
				\, ds.
				\notag 
		\end{align}
\end{lemma}

\begin{proof}
	Integrating by parts over $[t,1] \times \mathbb{T}^3$
	(we stress that $t \leq 1$) we deduce (without using any equation) 
	\begin{align}  \label{E:IBPNOEQUATION}
		\uplambda^{-1}
		\int_{\Sigma_t}
			|t \partial \LapseRenormalized|_{\gKasner}^2
		\, dx
		& = 
		\uplambda^{-1}
		\int_{\Sigma_1}
			|\partial \LapseRenormalized|_{\gKasner}^2
		\, dx
		- 2 \uplambda^{-1}
			\int_{s=t}^1
				s^{-1}
				\int_{\Sigma_s}
					|s \partial \LapseRenormalized|_{\gKasner}^2	
					+ s^2 \gKasner^{ab} (s \SecondFundKasner_{\ b}^c)
							\partial_a \LapseRenormalized \partial_c \LapseRenormalized
				\, dx
			\, ds
				\\
		& \ \ 
			+ 2 
				\int_{s=t}^1
				\int_{\Sigma_s}
					s \gKasner^{ef} \partial_e \partial_f \LapseRenormalized
					(\uplambda^{-1} s \partial_t \LapseRenormalized)
				\, dx
			\, ds.
			\notag
	\end{align}
	Using equation \eqref{E:PARABOLICLINEARIZEDLAPSE} 
	to substitute for the product $\uplambda^{-1} s \partial_t \LapseRenormalized$ 
	in the last integrand on the right-hand side of \eqref{E:IBPNOEQUATION} 
	and integrating by parts over $\Sigma_s$ on the resulting integrand product
	$\left\lbrace \gKasner^{ef} \partial_e \partial_f \LapseRenormalized \right\rbrace^2$,
	we deduce 
	\begin{align} \label{E:PARABOLICLAPSEEQUATIONINSERTED}
		\uplambda^{-1}
		\int_{\Sigma_t}
			|t \partial \LapseRenormalized|_{\gKasner}^2
		\, dx
		& = 
		\uplambda^{-1}
		\int_{\Sigma_1}
			|\partial \LapseRenormalized|_{\gKasner}^2
		\, dx
			\\
		& \ \
			- 2 \uplambda^{-1}
			\int_{s=t}^1
				s^{-1}
				\int_{\Sigma_s}
					|s \partial \LapseRenormalized|_{\gKasner}^2	
					+ s^2 \gKasner^{ab} (s \SecondFundKasner_{\ b}^c)
							\partial_a \LapseRenormalized \partial_c \LapseRenormalized
				\, dx
			\, ds
				\notag \\
	& \ \ 
			- 2
			\int_{s=t}^1
				s^{-1}
				\int_{\Sigma_s}
					|s^2 \partial^2 \LapseRenormalized|_{\gKasner}^2
				\, dx
			\, ds
			\notag
				\\
		& \ \
			- (2 A^2 - 1 - \uplambda^{-1}) 
				\int_{s=t}^1
				s^{-1}
				\int_{\Sigma_s}
					\LapseRenormalized
					(s^2 \gKasner^{ef} \partial_e \partial_f \LapseRenormalized)
				\, dx
			\, ds
			\notag \\
		& \ \ 
			+ 4 A
				\int_{s=t}^1
					s^{-1}
					\int_{\Sigma_s}
						(s \partial_t \SFRenormalized)
						(s^2 \gKasner^{ef} \partial_e \partial_f \LapseRenormalized)
					\, dx
				\, ds
				\notag \\
	& \ \
			+ 4
				\int_{s=t}^1
					s^{-1}
					\int_{\Sigma_s}
						(s \tracefreeSecondFundKasner_{\ b}^a) (s \LinSecondFund_{\ a}^b)
						(s^2 \gKasner^{ef} \partial_e \partial_f \LapseRenormalized)
					\, dx
					\, ds.
						\notag
	\end{align}
	Arguing as in the proof of \eqref{E:N2POINTWISE}
	(in particular using the fact that
	the eigenvalues of $t \SecondFundKasner_{\ j}^i$ are 
	$\geq - q_{Max} \geq - \left\lbrace \frac{1}{3} + \tracefreeparameter \right\rbrace$),
	we estimate the second integral on the right-hand side of \eqref{E:PARABOLICLAPSEEQUATIONINSERTED}
	as follows:
	\begin{align} \label{E:PARBOLICPARTIALLapseRenormalizedGAMESTIAMTE}
		- 2 \uplambda^{-1}
			\int_{s=t}^1
				s^{-1}
				\int_{\Sigma_s}
					|s \partial \LapseRenormalized|_{\gKasner}^2	
					+ s^2 \gKasner^{ab} (s \SecondFundKasner_{\ b}^c)
							\partial_a \LapseRenormalized \partial_c \LapseRenormalized
				\, dx
			\, ds
			& \leq 
				- \uplambda^{-1} \left(\frac{4}{3} - 2 \tracefreeparameter \right)
				\int_{s=t}^1
					s^{-1}
				 	\int_{\Sigma_s}
						|s \partial \LapseRenormalized|_{\gKasner}^2	
				 	\, dx
				\, ds.
		\end{align}
		Using the simple estimate $A \leq \sqrt{\frac{2}{3}}$,
		Young's inequality,
		and the simple estimate 
		$ \| \gKasner^{ef} \partial_e \partial_f \LapseRenormalized \|_{L^2} 
		\lesssim \| \partial^2 \LapseRenormalized \|_{L_{\gKasner}^2}$,
		we deduce that the four integrals on the 
		third through sixth lines of the right-hand side of \eqref{E:PARABOLICLAPSEEQUATIONINSERTED}
		are collectively bounded by 
		\begin{align} \label{E:PARABOLICLAPSECOLLECTIVEINTEGRALBOUND}
			& \leq
			- 
			\int_{s=t}^1
				s^{-1}
				\int_{\Sigma_s}
					|s^2 \partial^2 \LapseRenormalized|_{\gKasner}^2
				\, dx
			\, ds	
				\\
			& + C
				\int_{s=t}^1
					s^{-1}
					\int_{\Sigma_s}
						\LapseRenormalized^2
					\, dx
				\, ds
			+ C
				\int_{s=t}^1
					s^{-1}
					\int_{\Sigma_s}
						(s \partial_t \SFRenormalized)^2
					\, dx
				\, ds
				+ C
				\int_{s=t}^1
					s^{-1}
					\int_{\Sigma_s}
						\left|
							(s \tracefreeSecondFundKasner_{\ b}^a) (s \LinSecondFund_{\ a}^b)
						\right|^2
					\, dx
				\, ds.
				\notag
		\end{align}
		The desired inequality \eqref{E:TOPORDERLAPSEENERGYESTIMATE}
		now follows easily from \eqref{E:TRACEFREEPARAMETER} and
		\eqref{E:PARABOLICLAPSEEQUATIONINSERTED}
		and inequalities \eqref{E:PARBOLICPARTIALLapseRenormalizedGAMESTIAMTE} 
		and \eqref{E:PARABOLICLAPSECOLLECTIVEINTEGRALBOUND}.

\end{proof}

\subsubsection{Proof of Theorem~\ref{T:PARABOLICENERGYESTIMATES}}
	\label{SSS:PROOFOFTHMPARABOLICENERGYESTIMATES}
		We first note that the following pointwise
		estimates hold for the integrand terms
		$\mathcal{N}_i$, $i=1,2,\cdots,12$
		defined in \eqref{E:PARABOLICCUBICFORM1}-\eqref{E:PARABOLICFORM12},
		where the constants $C > 0$ are independent of
		$\uplambda \geq 1$ and $\smallparameter$:
	\begin{align}
	|\mathcal{N}_1| 
	& \leq 
		\frac{A^2}{A^2 + \frac{1}{4} \uplambda^{-1}(1 - \uplambda^{-1})}
			(t \partial_t \SFRenormalized)^2
		+
		\left\lbrace
			A^2 + \frac{1}{4} \uplambda^{-1}(1 - \uplambda^{-1})
		\right\rbrace
		\LapseRenormalized^2,
		\label{E:POINTWISEPARABOLICCUBICFORM1} \\
	|\mathcal{N}_2| 
	& \leq 
		(1 - \uplambda^{-1}) 
		\tracefreeparameter
		\smallparameter
		|s \LinSecondFund|_{\gKasner}^2
		+
		(1 - \uplambda^{-1}) 
		\frac{\tracefreeparameter}{\smallparameter}
		\LapseRenormalized^2,
		\label{E:POINTWISEPARABOLICFORM2} \\
	\mathcal{N}_3 
	& \leq 
			\left(\frac{2}{3} + 2 \tracefreeparameter \right) 
			|\partial \SFRenormalized|_{\gKasner}^2,
		\label{E:POINTWISEPARABOLICCUBICFORM3} \\
	|\mathcal{N}_4| 
	&  \leq
				\left\lbrace
					\frac{\uplambda}{\uplambda - 2 + \frac{1}{A^2}}
				\right\rbrace
				|s \partial \SFRenormalized|_{\gKasner}^2
				+ 
				\left\lbrace
					\frac{A^2 \uplambda - 2 A^2 + 1}{\uplambda}
				\right\rbrace
				|s \partial \LapseRenormalized|_{\gKasner}^2,
		\label{E:POINTWISEPARABOLICFORM4}
		\\
\smallparameter \mathcal{N}_5 
& \leq
\left(\frac{1}{6} + \frac{1}{2} \tracefreeparameter \right) 
				\smallparameter
				|s \partial \grenormalized|_{\gKasner}^2,
						\label{E:POINTWISEPARABOLICFORM5}	\\
\smallparameter |\mathcal{N}_6|
			& \leq 
				C \tracefreeparameter \smallparameter |s \LinSecondFund|_{\gKasner}^2,
					\label{E:POINTWISEPARABOLICFORM6} \\
	\smallparameter |\mathcal{N}_7|
			& \leq C \tracefreeparameter \smallparameter |s \partial \grenormalized|_{\gKasner}^2,
		\label{E:POINTWISEPARABOLICFORM7} 			\\
	\smallparameter 
	|\mathcal{N}_8|
			& \leq 
				\frac{1}{18} 
				\smallparameter
				|s \partial \grenormalized|_{\gKasner}^2
				+
				C 
				\smallparameter
				|s \partial \LapseRenormalized|_{\gKasner}^2,
							\label{E:POINTWISEPARABOLICFORM8} \\
	\smallparameter |\mathcal{N}_9| 
	& \leq 
		C (1 - \uplambda^{-1}) 
		\tracefreeparameter
		\smallparameter
		|s \LinSecondFund|_{\gKasner}^2
		+
		C (1 - \uplambda^{-1}) 
		\tracefreeparameter
		\smallparameter
		\LapseRenormalized^2,
				\label{E:POINTWISEPARABOLICFORM9} \\
	\smallparameter |\mathcal{N}_{10}|
	& 		\leq
				C
				\smallparameter
				|s \partial \SFRenormalized|_{\gKasner}^2
				+
				C
				\smallparameter
				|s \partial \LapseRenormalized|_{\gKasner}^2,
				\label{E:POINTWISEPARABOLICFORM10} \\
	\smallparameter |\mathcal{N}_{11}| 
	& \leq 
		\frac{1}{18} \smallparameter |s \partial \grenormalized|_{\gKasner}^2
		+
		C \smallparameter |s \partial \SFRenormalized|_{\gKasner}^2,
		\label{E:POINTWISEPOINTWISEPARABOLICFORM11}
			\\
  \smallparameter |\mathcal{N}_{12}| 
	& \leq 
		\frac{1}{18} \uplambda^{-1} \smallparameter |s \partial \grenormalized|_{\gKasner}^2
		+
		C \uplambda^{-1} \smallparameter |s \partial \SFRenormalized|_{\gKasner}^2.
		\label{E:POINTWISEPOINTWISEPARABOLICFORM12}
\end{align}	
The estimates \eqref{E:POINTWISEPARABOLICCUBICFORM1}-\eqref{E:POINTWISEPOINTWISEPARABOLICFORM12}
can be derived by using essentially the same reasoning that we
used to prove 
\eqref{E:N1POINTWISE}-\eqref{E:N10POINTWISE}
and we therefore omit the details.
Note, however, that the $\mathcal{N}_i$ have different definitions in
\eqref{E:POINTWISEPARABOLICCUBICFORM1}-\eqref{E:POINTWISEPOINTWISEPARABOLICFORM12} than they do
in \eqref{E:N1POINTWISE}-\eqref{E:N10POINTWISE}.

We now claim that there exist constants $C > 0$ and $c > 0$
such that the following estimate holds when 
$\smallparameter > 0$
and
$\uplambda \geq 1$:
\begin{align}  \label{E:PARABOLICFIRSTENERGYESTIMATE} 
	& 
	\left\lbrace
		\frac{\frac{1}{4} \uplambda^{-1}(1 - \uplambda^{-1})}{A^2 + \frac{1}{4} \uplambda^{-1}(1 - \uplambda^{-1})}
	\right\rbrace
	\int_{\Sigma_t} 
		(t \partial_t \SFRenormalized)^2
	\, dx  
	+
	\int_{\Sigma_t}
		|t \partial \SFRenormalized|_{\gKasner}^2
	\, dx
	+ 	
		\frac{1}{4} \uplambda^{-1}(1 - \uplambda^{-1})
		\int_{\Sigma_t}
			\LapseRenormalized^2
		\, dx
			\\
	& \ \
		+
		\smallparameter
		\int_{\Sigma_t} 
			|t \LinSecondFund|_{\gKasner}^2 
		\, dx
		+
		\frac{1}{4} 
		\smallparameter
		\int_{\Sigma_t} 
			|t \partial \grenormalized|_{\gKasner}^2 
		\, dx
			\notag \\
	& \leq  
		\left\lbrace
			1 + \frac{A^2}{A^2 + \frac{1}{4} \uplambda^{-1}(1 - \uplambda^{-1})}
		\right\rbrace
		\int_{\Sigma_1} 
			(\partial_t \SFRenormalized)^2
		\, dx  
		+
		\int_{\Sigma_1}
			|\partial \SFRenormalized|_{\gKasner}^2
		\, dx
	+ \left\lbrace
			2 A^2 
			+ 
		\frac{3}{4} \uplambda^{-1} (1 - \uplambda^{-1})
		\right\rbrace
		 	\int_{\Sigma_1}
		 		\LapseRenormalized^2
		 \, dx
			\notag \\
	& \ \
		+
		\smallparameter
		\int_{\Sigma_1} 
			|\LinSecondFund|_{\gKasner}^2 
		\, dx
		+
		\frac{1}{4} 
		\smallparameter
		\int_{\Sigma_1} 
			|\partial \grenormalized|_{\gKasner}^2 
		\, dx
		\notag \\
		& \ \ 
		- \left\lbrace 
				\frac{A^2 \uplambda - 8 A^2 + 4}{3 \left[A^2 (\uplambda - 2) + 1 \right]}
				-
				C \tracefreeparameter 
				-
				C \smallparameter
			\right\rbrace
			\int_{s=t}^1
				s^{-1}
				\int_{\Sigma_s}  
					 |s \partial \SFRenormalized|_{\gKasner}^2
				\, dx
		 	\, ds
				\notag \\
	& \ \
			- 
			\left\lbrace
				(\uplambda-2)
				\frac{1 - A^2}{\uplambda}
				+ 
				2 \smallparameter \uplambda^{-1} 
				-
				C \smallparameter
			\right\rbrace
			\int_{s=t}^1
				s^{-1}
				\int_{\Sigma_s}  
					 |s \partial \LapseRenormalized|_{\gKasner}^2
				\, dx
		 	\, ds 
				\notag \\
	& \ \
		-
		\left\lbrace 
			1 
			- 
			\uplambda^{-2} 
			-
			C (1 - \uplambda^{-1}) 
			\tracefreeparameter
			\smallparameter
			-
			(1 - \uplambda^{-1}) 
			\frac{\tracefreeparameter}{\smallparameter}
		\right\rbrace
		\int_{s=t}^1 
			s^{-1}
			\int_{\Sigma_s} 
				\LapseRenormalized^2
			\, dx
		\, ds
			\notag \\
	& \ \
		- 
		\smallparameter
		\left\lbrace
			\frac{2}{9}
			-
			\uplambda^{-1} \frac{1}{18}
			-
			C \tracefreeparameter 
		\right\rbrace
		\int_{s=t}^1
			s^{-1}
			\int_{\Sigma_s}
				 		|s \partial \grenormalized|_{\gKasner}^2
			\, dx
		\, ds
		 \notag	\\
		& \ \
		+ 
		\left\lbrace
			C 
			(1 - \uplambda^{-1}) 
			\tracefreeparameter 
			\smallparameter
			+
			C
			\tracefreeparameter 
			\smallparameter
		\right\rbrace
		\int_{s=t}^1 
			s^{-1} 
			\int_{\Sigma_s}
				|s \LinSecondFund|_{\gKasner}^2
			\, dx
		 \, ds.
		\notag
\end{align}
	To obtain \eqref{E:PARABOLICFIRSTENERGYESTIMATE},
	we simply substitute
	the estimates \eqref{E:POINTWISEPARABOLICCUBICFORM1}-\eqref{E:POINTWISEPOINTWISEPARABOLICFORM12}
	into the approximate monotonicity identity \eqref{E:PARABOLICMONOTONICITYID} 
	and keep careful track of the coefficients.
	
	Next, we note that by \eqref{E:TRACEFREEPARAMETER},
	if $2 < \uplambda_0 \leq \uplambda$
	and $\tracefreeparameter$ is sufficiently small in a manner that is independent of $\uplambda_0$,
	then the factor
	$
	\frac{A^2 \uplambda - 8 A^2 + 4}{3 \left[A^2 (\uplambda - 2) + 1 \right]}
	$
	in front of the integral
	$
	\int_{s=t}^1
				s^{-1}
				\int_{\Sigma_s}  
					 |s \partial \SFRenormalized|_{\gKasner}^2
				\, dx
		 	\, ds
	$
	on the right-hand side of \eqref{E:PARABOLICFIRSTENERGYESTIMATE} 
	is uniformly positive 
	(with a lower bound that \emph{does} depend on $\uplambda_0$)
	and increases to
	$
	\frac{1}{3}
	$
	as $\uplambda \to \infty$.
	From this observation and definition \eqref{E:PARABOLICALMOSTOTALENERGY},
	we see that 
	if $2 < \uplambda_0 \leq \uplambda$,
	then the desired estimate \eqref{E:PARABOLICSECONDMODELENERGYGRONWALLREADY}
	follows from \eqref{E:PARABOLICFIRSTENERGYESTIMATE}
	by first choosing $\smallparameter := \smallparameter_{\uplambda_0}$
	to be sufficiently small in a manner that depends on $\uplambda_0$ 
	and then choosing 
	$\tracefreeparameter$ to be sufficiently small in a manner that
	depends on $\uplambda_0$ and $\smallparameter_{\uplambda_0}$.
	
	The estimate \eqref{E:PARABOLICALMOSTTOTALENERGYGRONWALLED}
	then follows from \eqref{E:PARABOLICSECONDMODELENERGYGRONWALLREADY}
	and Gronwall's inequality.
	
	Our next goal is to prove the estimate \eqref{E:PARABOLICTOTALENERGYGRONWALLED}
	for $\mathscr{E}_{(Total);\smallparameter_{\uplambda_0};N}(t)$.
	As a first step, we will use the estimate
	\eqref{E:PARABOLICALMOSTTOTALENERGYGRONWALLED}
	to control the top-order terms in
$\mathscr{E}_{(Total);\smallparameter_{\uplambda_0};0}$
(see definition \eqref{E:TOPORDERENERGY})
that are not present in the definition \eqref{E:PARABOLICORDERNENERGY}
of $\mathscr{E}_{(Almost \ Total);\smallparameter_{\uplambda_0};0}$,
namely the term $\mathscr{E}_{(\partial Lapse)}^2(t)$
defined in \eqref{E:LINEARIZEDPARTIALLAPSEENERGY}.
To this end, we insert the estimates
implied by \eqref{E:PARABOLICALMOSTTOTALENERGYGRONWALLED}
into the last three integrals on the
right-hand side of
\eqref{E:TOPORDERLAPSEENERGYESTIMATE},
carry out straightforward computations,
and use Lemma \ref{L:PARABOLICENERGYNORMCOMPARISON} at $t=1$,
thereby deducing that
\begin{align} \label{E:ORDER0PARABOLICTOTALENERGYGRONWALLED}
	\mathscr{E}_{(Total);\smallparameter_{\uplambda_0};0}(t)
	& \leq 
		\begin{cases}
		\frac{C_{\uplambda_0}}{\tracefreeparameter} \parabolichighnorm{0}(1) t^{- c_{\uplambda_0} \tracefreeparameter}
		& \mbox{\upshape if} \ \tracefreeparameter \neq 0,
			\\
		C_{\uplambda_0} \parabolichighnorm{0}(1) (1 + |\ln t|)
		& \mbox{\upshape if} \ \tracefreeparameter = 0.
	\end{cases}
\end{align}
Next, we note that since the $\partial_{\vec{I}}$-differentiated quantities
$\partial_{\vec{I}} \LinSecondFund, 
	\partial \partial_{\vec{I}} \grenormalized, 
	\partial_{\vec{I}} \SFRenormalized,
	\partial_{\vec{I}}\LapseRenormalized$
verify the same linear equations as their non-differentiated counterparts
(for reasons similar to the ones given in the proof of Cor.\ \ref{C:MONOTONICITY}),	
it follows the energy of the $\partial_{\vec{I}}$-differentiated linear solution variables
verifies an analog of the estimate
\eqref{E:ORDER0PARABOLICTOTALENERGYGRONWALLED}.
Summing these estimates 
for $|\vec{I}| \leq N$ 
and appealing to the definition \eqref{E:TOPORDERENERGY} of $\mathscr{E}_{(Total);\smallparameter_{\uplambda_0};N}(t)$,
we arrive at the desired estimate \eqref{E:PARABOLICTOTALENERGYGRONWALLED}.
Finally, we note that inequality \eqref{E:PARABOLICHIGHNORMBOUND}
follows from inequality \eqref{E:PARABOLICTOTALENERGYGRONWALLED}
and Lemma~\ref{L:PARABOLICENERGYNORMCOMPARISON}.
This completes the proof of Theorem~\ref{T:PARABOLICENERGYESTIMATES}.
	
\hfill $\qed$

\section*{Acknowledgments}
The authors thank Mihalis Dafermos for 
offering enlightening comments on an earlier version of this work.
They also thank the anonymous referees for providing valuable feedback that helped
improve the exposition. IR gratefully acknowledges support from NSF grant \# DMS-1001500.
JS gratefully acknowledges support from NSF grant \# DMS-1162211,
from NSF CAREER grant \# DMS-1454419,
from a Sloan Research Fellowship provided by the Alfred P. Sloan foundation,
and from a Solomon Buchsbaum grant administered by the Massachusetts Institute of Technology.

\bibliographystyle{amsalpha}
\bibliography{JBib}

\end{document}